\documentclass[12pt,a4paper]{amsart}
\usepackage{amssymb}
\usepackage[cmtip,all]{xy}
\usepackage{hyperref}

\calclayout

\newcommand{\+}{\protect\nobreakdash-}
\renewcommand{\:}{\colon}

\newcommand{\rarrow}{\longrightarrow}

\newcommand{\ot}{\otimes}
\renewcommand{\d}{\partial}

\newcommand{\lrarrow}{\mskip.5\thinmuskip\relbar\joinrel\relbar\joinrel
 \rightarrow\mskip.5\thinmuskip\relax}
\newcommand{\llarrow}{\mskip.5\thinmuskip\leftarrow\joinrel\relbar
 \joinrel\relbar\mskip.5\thinmuskip\relax}

\DeclareMathOperator{\Hom}{Hom}
\DeclareMathOperator{\Ext}{Ext}
\DeclareMathOperator{\cHom}{\mathcal{H}\textit{om}}
\DeclareMathOperator{\Spec}{Spec}
\DeclareMathOperator{\Supp}{Supp}
\DeclareMathOperator{\gr}{gr}
\DeclareMathOperator{\Fil}{\mathsf{Fil}}
\DeclareMathOperator{\im}{im}
\DeclareMathOperator{\Sh}{Sh}

\newcommand{\Modl}{{\operatorname{\mathsf{--Mod}}}}
\newcommand{\Modr}{{\operatorname{\mathsf{Mod--}}}}
\newcommand{\Qcoh}{{\operatorname{\mathsf{--Qcoh}}}}
\newcommand{\Ctrh}{{\operatorname{\mathsf{--Ctrh}}}}
\newcommand{\bMod}{{\operatorname{\mathsf{--Mod--}}}}
\newcommand{\tors}{{\operatorname{\mathsf{-tors}}}}

\newcommand{\st}{\mathsf{st}}
\newcommand{\qu}{\mathsf{qu}}
\newcommand{\id}{\mathrm{id}}

\newcommand{\D}{\mathcal D}
\newcommand{\E}{\mathcal E}
\newcommand{\cO}{\mathcal O}
\newcommand{\U}{\mathcal U}
\newcommand{\V}{\mathcal V}
\newcommand{\fP}{\mathfrak P}
\newcommand{\fU}{\mathfrak U}
\newcommand{\fV}{\mathfrak V}

\newcommand{\sS}{\mathsf S}

\newcommand{\p}{\mathfrak p}
\newcommand{\q}{\mathfrak q}

\newcommand{\n}{\mathfrak n}

\newcommand{\boZ}{\mathbb Z}

\newcommand{\Section}[1]{\bigskip\section{#1}\medskip}
\theoremstyle{plain}
\newtheorem{thm}{Theorem}[section]
\newtheorem{prop}[thm]{Proposition}
\newtheorem{lem}[thm]{Lemma}
\newtheorem{cor}[thm]{Corollary}
\theoremstyle{definition}
\newtheorem{rem}[thm]{Remark}

\newtheorem{ex}[thm]{Example}
\newtheorem{exs}[thm]{Examples}

\begin{document}

\title{Torsion modules and differential operators \\
in infinitely many variables}

\author{Leonid Positselski}

\address{Institute of Mathematics, Czech Academy of Sciences \\
\v Zitn\'a~25, 115~67 Prague~1 \\ Czech Republic} 

\email{positselski@math.cas.cz}

\begin{abstract}
 This paper grew out of the author's work on~\cite{Pdomc}.
 Differential operators in the sense of Grothendieck acting between
modules over a commutative ring can be interpreted as torsion elements
in the bimodule of all operators with respect to the diagonal ideal in
the tensor square of the ring.
 Various notions of torsion modules for an infinitely generated ideal
in a commutative ring lead to various notions of differential operators.
 We discuss differential operators of transfinite orders and
differential operators having no global order at all, but only local
orders with respect to specific elements of the ring.
 Many examples are presented.
 In particular, we prove that every ordinal can be realized as
the order of a differential operator acting on the algebra of
polynomials in infinitely many variables over a field.
 We also discuss extension of differential operators to localizations
of rings and modules, and to colocalizations of modules.
\end{abstract}

\maketitle

\tableofcontents

\section*{Introduction}
\medskip

\setcounter{subsection}{-1}
\subsection{{}} \label{introd-beginning-subsecn}
 Let $k$~be a field of characteristic zero and $R=k[x_1,\dotsc,x_m]$
be a ring of polynomials in finitely many variables over~$k$.
 \emph{Differential operators} acting on $k[x_1,\dotsc,x_m]$ can be
defined straightforwardly as elements of the subring in the ring $E$
of $k$\+linear endomorphisms of $R$ spanned by the operators of
multiplication with functions $f\in R$ and the operators of partial
derivatives $\d/\d x_j$, \,$1\le j\le m$.

 In this context, the \emph{order} of a differential operator is defined
as its degree as a polynomial in the partial derivatives.
 So the functions $f\in R$, viewed as differential operators
$R\rarrow R$, have order~$0$, while the partial derivatives
$\d/\d x_j$ have order~$1$.
 For example, $x\frac{d}{dx}\:k[x]\rarrow k[x]$ is a differential
operator of order~$1$, while $\frac{d^2}{dx^2}\:k[x]\rarrow k[x]$ is
a differential operator of order~$2$ on the polynomial ring $k[x]$
in one variable~$x$.

 An abstract coordinate-free definition of \emph{differential operators
in the sense of Grothendieck}~\cite[Proposition~IV.16.8.8(b)]{EGAIV},
\cite[Section Tag~09CH]{SP} is formulated as follows.
 Let $K\rarrow R$ be a homomorphism of commutative rings, and let $U$
and $V$ be two $R$\+modules.
 Consider the $R$\+$R$\+bimodule $E=\Hom_K(U,V)$ of $K$\+linear maps
$U\rarrow V$.
 The $R$\+$R$\+subbimodule $F_n\D_{R/K}(U,V)\subset\Hom_K(U,V)$ of
\emph{$K$\+linear $R$\+differential operators of order\/~$\le n$}
is defined inductively by the rules
\begin{itemize}
\item $F_n\D_{R/K}(U,V)=0$ for $n<0$;
\item for any integer $n\ge0$, a $K$\+linear map $e\:U\rarrow V$ belongs
to $F_n\D_{R/K}(U,V)$ if and only if, for every element $r\in R$,
the $K$\+linear map $re-er\:U\rarrow V$ belongs to
$F_{n-1}\D_{R/K}(U,V)$.
\end{itemize}
 So $F_0\D_{R/K}(U,V)=\Hom_R(U,V)\subset\Hom_K(U,V)$ is
the $R$\+$R$\+subbimodule of $R$\+linear maps
in $\Hom_K(U,V)$, and one has $F_{n-1}\D_{R/K}(U,V)\subset
F_n\D_{R/K}(U,V)$ for all $n\ge0$.

 Let us say that an $R$\+$R$\+bimodule $M$ is
an \emph{$R$\+$R$\+bimodule over~$K$} if the left and right actions
of $K$ in $M$ agree, that is $lm=ml$ for all $m\in M$ and $l\in K$.
 In particular, $\Hom_K(U,V)$ is an $R$\+$R$\+bimodule over~$K$.
 The category of $R$\+$R$\+bimodules over $K$ is naturally equivalent
(in fact, isomorphic) to the category of modules over the ring
$T=R\ot_KR$.
 Denote by $I\subset T$ the kernel ideal of the natural (multiplication)
ring homomorphism $R\ot_KR\rarrow R$.
 Then an element $e\in E=\Hom_K(U,V)$ belongs to $F_n\D_{R/K}(U,V)$
if and only if $I^{n+1}e=0$ in~$E$.
 This observation can be found in the paper~\cite[Section~1.1]{BB}.

 So one can say that \emph{the $R$\+$R$\+subbimodule of differential
operators in\/ $\Hom_K(U,V)$ is the submodule of $I$\+torsion elements
in the $T$\+module $E=\Hom_K(U,V)$}.
 What are the $I$\+torsion elements in a $T$\+module?

\subsection{{}} \label{introd-torsion-modules-subsecn}
 Let $I$ be an ideal in a commutative ring~$T$.
 Two definitions of an \emph{$I$\+torsion $T$\+module} can be found in
the recent literature.
 Porta, Shaul, and Yekutieli~\cite[Section~3]{PSY} say that
a $T$\+module $M$ is $I$\+torsion if for every element $m\in M$ there
exists an integer $n\ge0$ such that $I^{n+1}m=0$ in~$M$.
 In this paper, we call such $T$\+modules \emph{strongly $I$\+torsion}.

 The present author's definition in~\cite[Section~1]{Pmgm} says that
a $T$\+module $M$ is \emph{$I$\+torsion} if for every pair of elements
$m\in M$ and $s\in I$ there exists an integer $n\ge0$ such that
$s^{n+1}m=0$ in~$M$.
 This definition describes a wider class of $T$\+modules than
the one in~\cite{PSY}.
 When the ideal $I\subset T$ is finitely generated, the classes of
$I$\+torsion and strongly $I$\+torsion $T$\+modules coincide.

 For any (possibly infinitely generated) ideal $I\subset T$, the class
of $I$\+torsion $T$\+modules $T\Modl_{I\tors}$ is closed under
subobjects, quotients, extensions, and infinite direct sums in
the module category $T\Modl$.
 The class of strongly $I$\+torsion $T$\+modules $T\Modl_{I\tors}^\st$
is also closed under subobjects, quotients, and infinite direct sums
in $T\Modl$, but it \emph{need not} be closed under extensions
(cf.~\cite[item~(2) in the Erratum]{PSY}).

 Closing the class of strongly $I$\+torsion $T$\+modules under
subobjects, quotients, extensions, and infinite direct sums in $T\Modl$, 
one obtains what we call the class of \emph{quite $I$\+torsion}
$T$\+modules $T\Modl_{I\tors}^\qu$.
 Equivalently, $T\Modl_{I\tors}^\qu$ is the closure of the class of
all $T$\+modules annihilated by $I$ under extensions and filtered
direct limits in $T\Modl$.
 Quite $I$\+torsion $T$\+modules $M$ can be characterized by
the following property: for every element $m\in M$ and every
sequence of elements $s_0$, $s_1$, $s_2$,~\dots~$\in I$ (indexed
by the nonnegative integers), there exists an integer $n\ge0$
such that $s_ns_{n-1}\dotsm s_1s_0m=0$ in~$M$.
 Generally speaking, the class of quite $I$\+torsion $T$\+modules
sits strictly in between the classes of strongly $I$\+torsion
and $I$\+torsion $T$\+modules,
\begin{equation} \label{torsion-modules-inclusions-inequalities}
 T\Modl_{I\tors}^\st\varsubsetneq T\Modl_{I\tors}^\qu
 \varsubsetneq T\Modl_{I\tors}.
\end{equation}

 For any $T$\+module $M$ and any ideal $I\subset T$, the natural
ordinal-indexed increasing filtration $F^{(I)}$ on $M$ is defined
by the rules
\begin{itemize}
\item $F^{(I)}_0M=\{\,m\in M\mid Im=0\,\}$;
\item $F^{(I)}_\alpha M=\{\,m\in M\mid Im\subset
\bigcup_{\beta<\alpha}F^{(I)}_\beta M\,\}$ for every ordinal $\alpha>0$.
\end{itemize}
 When the ideal $I$ is finitely generated, the filtration $F^{(I)}$
does not go beyond the first infinite ordinal~$\omega$; in fact, one
has $F^{(I)}_\alpha M=\bigcup_{n<\omega}F^{(I)}_n M$ for all
ordinals $\alpha\ge\omega$.
 For infinitely generated ideals $I$, this is no longer true.

 In particular, specializing to the case of a principal ideal generated
by an element $s\in T$, we obtain the definition of the natural
increasing filtration $F^{(s)}$ on $M$, indexed by the integers and
defined by the rules
\begin{itemize}
\item $F^{(s)}_n M=0$ for $n<0$;
\item $F^{(s)}_n M=\{\,m\in M\mid sm\in F^{(s)}_{n-1} M\,\}$
for $n\ge0$.
\end{itemize}
 In other words, $F^{(s)}_nM\subset M$ is the submodule of all
elements annihilated by~$s^{n+1}$.
 Similarly, $F^{(I)}_nM\subset M$ is the submodule of all
elements annihilated by $I^{n+1}$ when $n<\omega$ is an integer.

 A $T$\+module $M$ is strongly $I$\+torsion (i.~e., $I$\+torsion in
the sense of~\cite[Section~3]{PSY}) if and only if
$M=\bigcup_{n<\omega}F^{(I)}_nM$.
 A $T$\+module $M$ is quite $I$\+torsion if and only if
$M=\bigcup_\beta F^{(I)}_\beta M$, where the union is taken over
all ordinals~$\beta$ (equivalently, this means that there exists
an ordinal~$\alpha$ such that $F^{(I)}_\alpha M=M$).
 A $T$\+module $M$ is $I$\+torsion (i.~e., $I$\+torsion in
the sense of~\cite[Section~1]{Pmgm}) if and only if, for every
element $s\in I$, one has $M=\bigcup_n F^{(s)}_nM$.

 We say that an element $m\in M$ is \emph{strongly $I$\+torsion}
if $m\in\bigcup_{n<\omega}F^{(I)}_nM$.
 The submodule of all strongly $I$\+torsion elements in $M$ is
denoted by $\Gamma_I^\st(M)=\bigcup_{n<\omega}F^{(I)}_nM$.
 Furthermore, we say that an element $m\in M$ is
\emph{quite $I$\+torsion} if $m\in\bigcup_\beta F^{(I)}_\beta M$ (where
the union is taken over all ordinals~$\beta$).
 The submodule of all quite $I$\+torsion elements in $M$ is denoted
by $\Gamma_I^\qu(M)=\bigcup_\beta F^{(I)}_\beta M$.
 Finally, we say that an element $m\in M$ is \emph{$I$\+torsion}
if, for every element $s\in I$, one has $m\in\bigcup_n F^{(s)}_nM$.
 The submodule of all $I$\+torsion elements in $M$ is denoted by
$\Gamma_I(M)=\bigcap_{s\in I}\bigcup_n F^{(s)}_n M$.

\subsection{{}} \label{introd-diff-operators-again}
 Now we can return to differential operators.
 As in Section~\ref{introd-beginning-subsecn}, let $K\rarrow R$ be
a homomorphism of commutative rings, and let $U$ and $V$ be two
$R$\+modules.
 Consider the $R$\+$R$\+bimodule $E=\Hom_K(U,V)$.
 So $E$ is naturally a $T$\+module, where $T=R\ot_KR$.
 Let $I\subset T$ be the kernel ideal of the natural ring homomorphism
$R\ot_KR\rarrow R$.

 Following the discussion in Section~\ref{introd-beginning-subsecn},
for any integer $n\ge0$, the $R$\+$R$\+subbimodule of differential
operators of order~$\le n$ is $F_n\D_{R/K}(U,V)=F_n^{(I)}E$.
 For every ordinal~$\beta$, we put $F_\beta\D_{R/K}(U,V)=
F_\beta^{(I)}E$, and call the elements of $F_\beta\D_{R/K}(U,V)$
the \emph{$K$\+linear $R$\+differential operators of} (\emph{ordinal})
\emph{order\/~$\le\beta$}.

 Explicitly, the $R$\+$R$\+subbimodules $F_\beta\D_{R/K}(U,V)\subset
\Hom_K(U,V)$ can be defined inductively by the rules
\begin{itemize}
\item $F_0\D_{R/K}(U,V)=\Hom_R(U,V)\subset\Hom_K(U,V)$;
\item $F_\alpha\D_{R/K}(U,V)=\{\,e\in\Hom_K(U,V)\mid
re-er\in\bigcup_{\beta<\alpha}F_\beta\D_{R/K}(U,V)$
for all $r\in R\,\}$ for every ordinal $\alpha>0$.
\end{itemize}

 We call the elements of $\bigcup_{n<\omega}F_n\D_{R/K}(U,V)=
\Gamma_I^\st(E)$ the \emph{$K$\+linear strongly $R$\+differential
operators}, and the elements of
$\bigcup_\beta F_\beta\D_{R/K}(U,V)=\Gamma_I^\qu(E)$
the \emph{$K$\+linear quite $R$\+differential operators}.
 So the strongly differential operators have integer orders, while
the quite differential operators have ordinal orders.
 Let us introduce the notation $\D^\st_{R/K}(U,V)=\Gamma_I^\st(E)$
and $\D^\qu_{R/K}(U,V)=\Gamma_I^\qu(E)$.

 Finally, we define the $R$\+$R$\+subbimodule of \emph{$K$\+linear
$R$\+differential operators} as $\D_{R/K}(U,V)=\Gamma_I(E)=
\bigcap_{r\in R}\bigcup_{n\ge0}F_n^{(r)}(\Hom_K(U,V))$.
 Here $F_n^{(r)}E=F_n^{(r)}(\Hom_K(U,V))\subset\Hom_K(U,V)$ is
a shorthand notation for the $R$\+$R$\+subbimodule
$F_n^{(r\ot1-1\ot r)}E\subset E$, with $r\ot1-1\ot r\in R\ot_KR=T$.
 Thus, generally speaking, our $K$\+linear $R$\+differential operators
$U\rarrow V$ have no orders at all, but only (finite) orders with
respect to specific elements $r\in R$.
 Explicitly, we put {\hbadness=2550
\begin{itemize}
\item $F^{(r)}_n\D_{R/K}(U,V)=0$ for integers $n<0$;
\item $F^{(r)}_n\D_{R/K}(U,V)=\{\,e\in\D_{R/K}(U,V)\mid
re-er\in F^{(r)}_{n-1}\D_{R/K}(U,V)\,\}$ for integers $n\ge0$.
\end{itemize}
 So} one has $\D_{R/K}(U,V)=\bigcup_{n\ge0}F_n^{(r)}\D_{R/K}(U,V)$
for every $r\in R$.
 We call the elements of $F^{(r)}_n\D_{R/K}(U,V)$ the \emph{$K$\+linear
$R$\+differential operators of $r$\+order\/~$\le n$}.

\subsection{{}} \label{introd-examples-subsecn}
 Generally speaking, one has
\begin{equation} \label{diff-operators-inclusions-inequalities}
 \D^\st_{R/K}(U,V)\varsubsetneq\D^\qu_{R/K}(U,V)
 \varsubsetneq\D_{R/K}(U,V).
\end{equation}
 Here the inclusions in~\eqref{diff-operators-inclusions-inequalities}
follow from the inclusions of the classes of torsion modules in
formula~\eqref{torsion-modules-inclusions-inequalities} in
Section~\ref{introd-torsion-modules-subsecn}, while the inequalities
in~\eqref{diff-operators-inclusions-inequalities} imply the inequalities
in~\eqref{torsion-modules-inclusions-inequalities}.

 Let us demonstrate some counterexamples showing that the inclusions
in~\eqref{diff-operators-inclusions-inequalities} are indeed strict.
 In all of these examples, $K=k$ is a field of characteristic zero
and $R$ the ring of polynomials in infinitely many variables.
 Unless otherwise mentioned, the set of variables is countable.
 The operators act from $R$ to $R$, so $U=V=R$.

 Take $R=k[x_1,x_2,x_3,\dotsc]$.
 Then the infinitary Laplace operator
$$
 D_2=\sum_{i=1}^\infty \frac{\d^2}{\d x_i^2}
$$
is obviously well-defined as a $k$\+linear map $R\rarrow R$.
 It is a strongly differential operator of order~$2$.
 See Example~\ref{laplace-operator-example} in the main body of
the paper.

 The infinite sum
$$
 D_\omega=\sum_{i=1}^\infty \frac{\d^i}{\d x_i^i}
$$
is also well-defined as a $k$\+linear map $R\rarrow R$.
 It is a quite differential operator of infinite order~$\omega$.
 Indeed, for any polynomial $f\in R$, the commutator $[f,D_\omega]$
is a finite linear combination of compositions of $x_j$ and $\d/\d x_j$,
\,$j=1$, $2$, $3$,~\dots; however, this finitary differential operator
can have arbitrarily high finite order.
 Clearly, for $f=x_i$, the commutator $[x_i,D_\omega]$ is
a differential operator of order $i-1$.
 See Example~\ref{D-omega-operator-example}.

 Pick an integer $n\ge0$, and put $R=k[x_1,x_2,x_2,\dotsc;y]$.
 Then the infinite sum
$$
 D_{\omega+n}=\sum_{i=1}^\infty
 \frac{\d^n}{\d y^n}\,\frac{\d^i}{\d x_i^i}
 =\frac{\d^n}{\d y^n}\,\sum_{i=1}^\infty \frac{\d^i}{\d x_i^i}
$$
is well-defined as a $k$\+linear map $R\rarrow R$.
 It is a quite differential operator of ordinal order $\omega+n$.
 Indeed, one has $[y,D_{\omega+n}]=-nD_{\omega+n-1}$, so
$[y,[y,\dotsm[y,D_{\omega+n}]\dotsm]]=(-1)^n n!D_\omega$
($n$~nested brackets).
 See Example~\ref{D-omega+n-operator-example}.

 Let $R$ be the ring of polynomials in two countably infinite
families of variables, $R=k[x_1,x_2,x_2,\dotsc;y_1,y_2,y_3,\dotsc]$.
 Then the infinite sum
$$
 D_{\omega+\omega}=\sum_{j=1}^\infty\sum_{i=1}^\infty
 \frac{\d^j}{\d y_j^j}\,\frac{\d^i}{\d x_i^i}
 =\sum_{j=1}^\infty \frac{\d^j}{\d y_j^j}
 \,\sum_{i=1}^\infty \frac{\d^i}{\d x_i^i}
$$
is well-defined as a $k$\+linear map $R\rarrow R$.
 It is a quite differential operator of ordinal order $\omega+\omega$
(see Example~\ref{D-omega+omega-operator-example}).

 Moreover, for each ordinal~$\alpha$ there exists a differential
operator $D_\alpha\:R\rarrow R$ of order~$\alpha$, where
$R=k[(x_i)_{i\in\Lambda}]$ is the ring of polynomials
in variables~$x_i$ indexed by a suitable set~$\Lambda$.
 This is the result of Theorem~\ref{all-ordinal-orders-possible}.
 It follows that the ordinal-indexed filtration $F^{(I)}$ from
Section~\ref{introd-torsion-modules-subsecn} on a module $M$ over
a commutative ring $T$ with an ideal $I$ can have arbitrary large
length.
 This means that for every ordinal~$\alpha$ there exists a commutative
ring $T$ with an ideal $I\subset T$ and a $T$\+module $M$ such that
$\bigcup_{\beta<\alpha}F^{(I)}_\beta M\varsubsetneq F^{(I)}_\alpha M$.

 Take again $R=k[x_1,x_2,x_3,\dotsc]$ to be the ring of polynomials
in a countable family of variables.
 Then the infinite sum
$$
 D_\infty=\frac{\d}{\d x_1}+\frac{\d^2}{\d x_1\,\d x_2}
 +\frac{\d^3}{\d x_1\,\d x_2\,\d x_3}+\dotsb
$$
is well-defined as a $k$\+linear map $R\rarrow R$.
 It is a differential operator but \emph{not} a quite differential
operator.
 So $D_\infty$ does not even have an ordinal order, but only (finite)
orders with respect to specific functions $f\in R$.
 In fact, one has $[x_i,D_\infty]\ne0$ but $[x_i,[x_i,D_\infty]]=0$ for
all $i\ge1$; so $D_\infty$ is a differential operator of order~$1$
with respect to every coordinate function $x_i\in R$.
 See Example~\ref{D-infinity-operator-example}.

 For comparison, let us mention an obvious example of an infinite linear
combination of compositions of derivatives (in one variable) that is
well-defined as an operator $k[x]\rarrow k[x]$ but \emph{is not
a differential operator at all}.
 Put
$$
 \Sh=\sum_{i=0}^\infty \frac{1}{i!}\frac{d^i}{d x^i}.
$$
 By Taylor's formula, one has $(\Sh f)(x)=f(x+1)$ for all $f\in k[x]$.
 So $\Sh$ is \emph{not} a differential operator in any meaningful sense
of the word (see Examples~\ref{nondifferential-operators-examples}).
 This example demonstrates the necessity of the formal definitions of
various classes of differential operators above.

\subsection{{}}
 Let $U$ and $V$ be two modules over a commutative ring $R$,
and let $D\:U\rarrow V$ be (what in our terminology is called)
a strongly $R$\+differential operator of order~$n$.
 According to~\cite[Lemma Tag~0G36]{SP}, for any multiplicative
subset $\Sigma\subset R$, the operator $D\:U\rarrow V$
can be uniquely extended to a strongly $\Sigma^{-1}R$\+differential
operator $\Sigma^{-1}U\rarrow\Sigma^{-1}V$ of order at most~$n$.

 In Section~\ref{localizing-diffops-secn} of this paper, we prove two
different generalizations of this lemma.
 Firstly, if $D\:U\rarrow V$ is an $R$\+differential operator and
$R\rarrow S$ is a flat epimorphism of commutative rings
(in the sense of~\cite[Sections~XI.1\+-3]{Ste}), then the operator $D$
can be uniquely extended to an $S$\+differential operator $S\ot_RU
\rarrow S\ot_RV$.
 This is the result of our Theorem~\ref{localizing-diff-operators-thm}.

 Secondly, if $D\:U\rarrow V$ is a quite $R$\+differential operator
of ordinal order~$\alpha$ and $R\rarrow S$ is a flat epimorphism
of commutative rings, then the operator $D$ can be uniquely extended
to a quite $S$\+differential operator $S\ot_RU\rarrow S\ot_RV$ of
ordinal order at most~$\alpha$.
 This is our Proposition~\ref{localizing-quite-diff-operators-prop}.
 It includes the result of~\cite[Lemma Tag~0G36]{SP} as the special
case for $S=\Sigma^{-1}R$ and $\alpha=n$.

 To illustrate the informal claim that the results of
Theorem~\ref{localizing-diff-operators-thm} and
Proposition~\ref{localizing-quite-diff-operators-prop} are nontrivial,
consider the shift operator $\Sh\:k[x]\rarrow k[x]$ from the end
of Section~\ref{introd-examples-subsecn} above.
 Put $R=k[x]$ and $S=k[x,x^{-1}]$; so $S$ is the ring of rational
functions~$f$ of one variable~$x$ over a field~$k$ such that
the denominator of~$f$ is a power of~$x$.
 Then, of course, one can extend the operator $\Sh\:R\rarrow R$
to some $k$\+linear operator $S\rarrow S$ in many ways.
 But there seems to be no natural, generally applicable way to do it.
 Simply put, given a rational function $f(x)\in k[x,x^{-1}]$,
the rational function $f(x+1)$ usually does not belong to $k[x,x^{-1}]$.

\subsection{{}} \label{introd-localizations-of-diff-operators}
 The question of localizing differential operators is relevant in
the following geometric context.
 Let $f\:X\rarrow T$ be a morphism of schemes, and let $\U$ and $\V$
be quasi-coherent sheaves on~$X$.
 Then the sheaf of $\cO_X$\+modules $\E=\cHom_{\cO_X}(\U,\V)$ on $X$
can be constructed by the rule $\E(Y)=\Hom_{\cO_X(Y)}(\U(Y),\V(Y))$
for all affine open subschemes $Y\subset X$ (though the sheaf $\E$
is \emph{not} in general quasi-coherent).
 This is the sheaf of strongly differential operators $\U\rarrow\V$
of order~$0$.

 One would like to be able to construct the sheaves of differential
operators $\U\rarrow\V$ of higher ordinal orders~$\alpha$, or even
the sheaf of arbitrary differential operators (without order),
linear over~$T$.
 Specifically, let $W\subset T$ and $Y\subset X$ be affine open
subschemes such that $f(Y)\subset W$.
 Put $\D_{X/T}(\U,\V)(Y)=\D_{\cO_X(Y)/\cO_T(W)}(\U(Y),\V(V))$.
 This is, in the notation of Section~\ref{introd-diff-operators-again},
the $\cO_X(Y)$\+$\cO_X(Y)$\+bimodule of all $\cO_T(W)$\+linear
$\cO_X(Y)$\+differential operators $\U(Y)\rarrow\V(Y)$.

 One would like to define $\D_{X/T}(\U,\V)$ as a sheaf of
$\cO_X$\+$\cO_X$\+bimodules on~$X$.
 The first step in this direction would be to construct the restriction
map $\D_{X/T}(\U,\V)(Y)\allowbreak\rarrow\D_{X/T}(\U,\V)(Y')$ for any
pair of affine open subschemes $Y'\subset Y\subset X$.
 This involves extending any given $\cO_T(W)$\+linear
$\cO_X(Y)$\+differential operator $\U(Y)\rarrow\V(Y)$ to
an $\cO_T(W)$\+linear $\cO_X(Y')$\+differential operator
$\U(Y')\rarrow\V(Y')$.

 Notice that $\cO_X(Y)\rarrow\cO_X(Y')$ is a flat epimorphism of
commutative rings, while $\U(Y')=\cO_X(Y')\ot_{\cO_X(Y)}\U(Y)$
and $\V(Y')=\cO_X(Y')\ot_{\cO_X(Y)}\V(Y)$ for quasi-coherent
sheaves $\U$ and~$\V$.
 So our Theorem~\ref{localizing-diff-operators-thm} claims
that the desired natural way of extending $\cO_X(Y)$\+differential
operators $\U(Y)\rarrow\V(Y)$ to $\cO_X(Y')$\+differential operators
$\U(Y')\rarrow\V(Y')$ exists.

 In other words, Theorem~\ref{localizing-diff-operators-thm}
establishes the existence of a presheaf of
$\cO_X$\+$\cO_X$\+bi\-modu\-les $\D_{X/T}(\U,\V)$, defined on
the topology base of $X$ consisting of affine open subschemes.
 Our further result,
Proposition~\ref{sheaf-axiom-differential-operators}
(or Proposition~\ref{sheaf-axiom-quite-differential-operators} for
quite differential operators of any given ordinal order) claims
that the sheaf axiom for affine open coverings of affine open
subschemes of $X$ is satisfied for $\D_{X/T}(\U,\V)$.
 Using the classical technique of extension of sheaves from a topology
base~\cite[Section~0.3.2]{EGAI}, one can construct the desired sheaf
of $\cO_X$\+$\cO_X$\+bimodules $\D_{X/T}(\U,\V)$ on~$X$.

 Let us emphasize that the discussion above pertains to the question
of $\D_{X/T}(\U,\V)$ being a \emph{sheaf of $\cO_X$\+$\cO_X$\+bimodules}
and \emph{not} a quasi-coherent sheaf.
 The quasi-coherence property (with respect to the left, or
equivalently, to the right $\cO_X$\+module structure) for the sheaf
of $\cO_X$\+$\cO_X$\+bimodules $\D_{X/T}(\U,\V)$ was established in
the classical work~\cite[Propositions~IV.16.8.6 and~IV.16.8.8]{EGAIV}
under much more restrictive assumptions.
 To wit, if $X\rarrow T$ is a morphism of schemes locally of finite
presentation, $\U$ is a locally finitely presented quasi-coherent
sheaf, and $\V$ is a quasi-coherent sheaf on $X$, then the sheaf of
(in our terminology) strongly differential operators
$\D_{X/T}^\st(\U,\V)$ on $X$ is quasi-coherent.
 Under such assumptions on the morphism of schemes $X\rarrow T$,
the classes of differential operators, quite differential operators,
and strongly differential operators coincide; so one actually has
$\D_{X/T}(\U,\V)=\D_{X/T}^\qu(\U,\V)=\D_{X/T}^\st(\U,\V)$.
 In the settings with an infinite number of variables, where our
three notions of differential operators differ from each other,
the sheaves of differential operators are usually \emph{not}
quasi-coherent.

\subsection{{}} \label{introd-colocalizations-of-diff-operators}
 In the final Section~\ref{colocalizing-diffops-secn} of this paper,
we consider the dual-analogous question of lifting of differential
operators to \emph{colocalizations} of modules.
 For any $R$\+differential operator $D\:U\rarrow V$ and any flat
epimorphism of commutative rings $R\rarrow S$, we show in
Theorem~\ref{colocalizing-diff-operators-thm} that the operator $D$
can be uniquely lifted to an $S$\+differential operator
$\Hom_R(S,U)\rarrow\Hom_R(S,V)$.
 By Proposition~\ref{colocalizing-quite-diff-operators-prop}, any
quite $R$\+differential operator $D\:U\rarrow V$ of ordinal
order~$\alpha$ can be uniquely lifted to a quite $S$\+differential
operator $\Hom_R(S,U)\rarrow\Hom_R(S,V)$ of ordinal order at
most~$\alpha$.

 From the geometric standpoint, the question about colocalizing
differential operators arises in the context of \emph{contraherent
cosheaves}~\cite{Pcosh,Pphil}.
 For a contraherent cosheaf $\fP$ on a scheme $X$ and a pair of
affine open subschemes $Y'\subset Y\subset X$,
the $\cO_X(Y')$\+module of cosections $\fP[Y']$ is computed as
the colocalization $\fP[Y']=\Hom_{\cO_X(Y)}(\cO_X(Y'),\fP[Y])$.

 Given a morphism of schemes $X\rarrow T$ and two contraherent
cosheaves $\fU$ and $\fV$ on $X$, one would like to be able to
construct a sheaf of $\cO_X$\+$\cO_X$\+bimodules $\D_{X/T}(\fU,\fV)$
on~$X$.
 The bimodules of sections of $\D_{X/T}(\fU,\fV)$ over affine open
subschemes $Y\subset X$ are supposed to be computed as
$\D_{X/T}(\fU,\fV)(Y)=\D_{\cO_X(Y)/\cO_T(W)}(\fU[Y],\fV[Y])$.
 Thus, in order to construct the restriction map
$\D_{X/T}(\fU,\fV)(Y)\rarrow\D_{X/T}(\fU,\fV)(Y')$ for a pair of
affine open subschemes $Y'\subset Y\subset X$, one needs to be able
to lift any given $\cO_T(W)$\+linear $\cO_X(Y)$\+differential
operator $\fU[Y]\rarrow\fV[Y]$ to an $\cO_T(W)$\+linear
$\cO_X(Y')$\+differential operator $\fU[Y']\rarrow\fV[Y']$.

 This means lifting differential operators to colocalizations
of modules.
 In our Theorem~\ref{colocalizing-diff-operators-thm}, we prove that
this can be done in a unique way.
 The similar result for quite differential operators is
Proposition~\ref{colocalizing-quite-diff-operators-prop}.
 We further proceed to check the sheaf axiom for affine open
coverings of affine open subschemes $Y\subset X$, essentially
establishing the existence of the sheaf of $\cO_X$\+$\cO_X$\+bimodules
$\D_{X/T}(\fU,\fV)$ on~$X$ (in
Proposition~\ref{cosheaf-axiom-differential-operators}; or
Proposition~\ref{cosheaf-axiom-quite-differential-operators} for
quite differential operators).

 To avoid any possible confusion, let us \emph{warn} the reader that
the notion of a contraherent cosheaf involves its own set of technical
aspects.
 In particular, while a quasi-coherent sheaf on an affine scheme $Y$
can be assigned to any $\cO(Y)$\+module, contraherent cosheaves on $Y$
correspond to \emph{contraadjusted} $\cO(Y)$\+modules only.
 As far as (co)homological conditions on modules go,
the contraadjustedness condition is quite mild, and contraadjusted
modules over commutative rings are ubiquitous~\cite[Section~1.1]{Pcosh},
\cite[Section~2]{Pcta}, \cite[Section~4.3]{Pphil}; but one \emph{cannot}
assign a contraherent cosheaf on an affine scheme $Y$ to
an $\cO(Y)$\+module that is not contraadjusted.
 More generally, for any contraherent cosheaf $\fU$ on a scheme $X$
and any affine open subscheme $Y\subset X$, the $\cO_X(Y)$\+module
$\fU[Y]$ must be contraadjusted.
 (The word ``contraadjusted'' means ``adjusted to contraherent
cosheaves''.)

 Hence the contraadjustedness assumption in
Propositions~\ref{cosheaf-axiom-differential-operators}
and~\ref{cosheaf-axiom-quite-differential-operators}.
 See Remark~\ref{contraadjustedness-necessary-remark} for
a counterexample showing that this assumption is necessary.
 We will continue this discussion in the next paragraph.

\subsection{{}}
 Another (co)sheaf-theoretic interpretation of the results about
localizations and colocalizations of differential operators can be
formulated as follows.
 Let $Y$ be an affine scheme over a commutative ring~$K$.
 Then it is well-known that the category $Y\Qcoh$ of quasi-coherent
sheaves on $Y$ and $\cO_Y$\+linear morphisms of such sheaves is
equivalent to the category of modules over the commutative
ring~$\cO(Y)$.
 Similarly, the category $Y\Ctrh$ of contraherent cosheaves on $Y$ and
$\cO_Y$\+linear morphisms of such cosheaves is equivalent to
the category of $\cO(Y)$\+modules satisfying a mild additional condition
called ``contraadjustedness''~\cite[Corollary~2.2.3]{Pcosh},
\cite[formulas~(10\+-12) in Section~5.6]{Pphil}.

 One may be interested in extending the category $Y\Qcoh$ to a category
with the same objects and $K$\+linear $\cO_Y$\+differential operators
as morphisms.
 Here a \emph{$K$\+linear $\cO_Y$\+differential operator}
$D\:\U\rarrow\V$ acting between two quasi-coherent sheaves $\U$ and $\V$
on $Y$ is a map of sheaves of $K$\+modules $D\:\U\rarrow\V$ such that,
for every affine open subscheme $Y'\subset Y$, the $K$\+linear map
$D(Y')\:\U(Y')\rarrow\V(Y')$ is an $\cO_Y(Y')$\+differential operator
between the $\cO_Y(Y')$\+modules $\U(Y')$ and $\V(Y')$.
 In this context, the theorem about existence and uniqueness of
localizations of differential operators (our
Theorem~\ref{localizing-diff-operators-thm}) implies that the category
of quasi-coherent sheaves on $Y$ and $K$\+linear $\cO_Y$\+differential
operators between them is equivalent to the category of
$\cO(Y)$\+modules and $K$\+linear $\cO(Y)$\+differential operators.

 Similarly, given two contraherent cosheaves $\fU$ and $\fV$ on $Y$,
one can say that a map of cosheaves of $K$\+modules $D\:\fU\rarrow\fV$
is a \emph{$K$\+linear $\cO_Y$\+differential operator} if, for
every affine open subscheme $Y'\subset Y$, the $K$\+linear map
$D[Y']\:\fU[Y']\rarrow\fV[Y']$ is an $\cO_Y(Y')$\+differential operator
between the $\cO_Y(Y')$\+modules $\fU[Y']$ and $\fV[Y']$.
 Then the theorem of existence and uniqueness of colocalizations of
differential operators (our
Theorem~\ref{colocalizing-diff-operators-thm}) implies that the category
of contraherent cosheaves on $Y$ and $K$\+linear $\cO_Y$\+differential
operators between them is equivalent to the category of contraadjusted
$\cO(Y)$\+modules and $K$\+linear $\cO(Y)$\+differential operators
between them.

\subsection{{}}
 Before we finish this Introduction, let us mention a few additional
references, for the benefit of a reader with background in algebra or
commutative algebra who may be not very familiar with the basic concepts
of set theory.
 The notions of ordinals and cardinals are used throughout this paper,
and particularly in Section~\ref{bounding-realizing-ordinals-secn}.
 For a very basic intuitive introduction to these concepts, we suggest
the book~\cite[Appendix to Chapter~II]{Man}.
 Specific classical constructions and results relevant to the exposition
in this paper can be found in the initial chapters of
the books~\cite{Lev} and~\cite{Kun}.
 See the beginning of Section~\ref{bounding-realizing-ordinals-secn}
for a further discussion.

\subsection*{Acknowledgement}
 I~am grateful to Michal Hrbek for very helpful discussions, and
particularly for communicating Example~\ref{torsion-hrbek-counterex}
to me and giving the permission to include it in this paper.
 I~also wish to thank Jan \v St\!'ov\'\i\v cek for illuminating
conversations.
 Last but not least, I~would like to thank the anonymous reviewer
for careful reading of the manuscript and many helpful suggestions.
 The author is supported by the GA\v CR project 23-05148S
and the Institute of Mathematics, Czech Academy of Sciences
(research plan RVO:~67985840).

\Section{Three Classes of Torsion Modules}

\subsection{Torsion modules} \label{torsion-modules-subsecn}
 Let $T$ be a commutative ring and $M$ be a $T$\+module.
 Let $s\in T$ be an element.
 An element $m\in M$ is said to be \emph{$s$\+torsion} if there exists
an integer $n\ge0$ such that $s^{n+1}m=0$ in~$M$.
 We denote by $F^{(s)}_nM\subset M$ the submodule of all elements
annihilated by $s^{n+1}$ in~$M$, or equivalently, define by induction
\begin{itemize}
\item $F^{(s)}_nM=0$ for all integers $n<0$;
\item $F^{(s)}_nM=\{\,m\in M\mid sm\in F^{(s)}_{n-1}M\,\}$ for all
integers $n\ge0$.
\end{itemize}
 So $F^{(s)}_0M$ is the submodule of all elements annihilated by~$s$
in~$M$, and $F^{(s)}_{n-1}M\subset F^{(s)}_nM$ for all $n\ge0$.
 A $T$\+module $M$ is said to be \emph{$s$\+torsion} if all its
elements are $s$\+torsion, i.~e., if $M=\bigcup_{n\ge0}F^{(s)}_nM$.

 We denote by $T\Modl$ the abelian category of $T$\+modules and
by $T\Modl_{s\tors}\subset T\Modl$ the full subcategory of $s$\+torsion
$T$\+modules.

\begin{lem} \label{s-torsion-extension-closed}
 The class of $s$\+torsion $T$\+modules $T\Modl_{s\tors}$ is closed
under subobjects, quotients, extensions, and infinite direct sums in
$T\Modl$.
\end{lem}

\begin{proof}
 We will only prove the closedness under extensions (all the other
properties being obvious).
 Let $0\rarrow L\rarrow M\rarrow N\rarrow0$ be a short exact sequence
of $T$\+modules with $s$\+torsion $T$\+modules $L$ and $N$.
 Let $m\in M$ be an element.
 Since the $T$\+module $N$ is $s$\+torsion, there exists an integer
$n_1\ge1$ such that the coset $m+L\in N$ is annihilated by~$s^{n_1}$,
that is, $s^{n_1}(m+L)=0$ in~$N$.
 Hence $s^{n_1}m\in L\subset M$.
 Since the $T$\+module $L$ is $s$\+torsion, there exists an integer
$n_2\ge1$ such that the element $s^{n_1}m\in L$ is annihilated
by~$s^{n_2}$, that is, $s^{n_2}(s^{n_1}m)=0$ in~$L$.
 Now we have $s^{n_1+n_2}m=0$ in~$M$; thus $m$~is an $s$\+torsion
element in~$M$.
\end{proof}

\begin{lem} \label{torsion-checked-on-generators-of-ideal}
 Let $I\subset T$ be an ideal generated by some set of elements
$s_j\in I$.
 Let $M$ be a $T$\+module that is $s_j$\+torsion for every~$j$.
 Then $M$ is $s$\+torsion for every element $s\in I$.
\end{lem}

\begin{proof}
 We have $s=\sum_{a=1}^b t_as_{j_a}$ for some integer $b\ge1$,
some elements $t_a\in T$, and some indices~$j_a$.
 Let $m\in M$ be an element and $n_a\ge0$ be some integers such that
$s_{j_a}^{n_a+1}m=0$ in~$M$.
 Put $n=\sum_{a=1}^bn_a$.
 Then $s^{n+1}m=0$ in~$M$.
\end{proof}

\begin{lem} \label{torsion-finitely-generated-ideal}
 Let $I\subset T$ be an ideal generated by some \emph{finite} set of
elements $s_1$,~\dots,~$s_b$.
 Let $M$ be a $T$\+module and $m\in M$ be an element that is
$s_j$\+torsion for every~$j$.
 Then there exists an integer $n\ge0$ such that $I^{n+1}m=0$ in~$M$.
\end{lem}

\begin{proof}
 Similarly to the previous proof, let $n_a\ge0$ be some integers
such that $s_a^{n_a+1}m=0$ in~$M$.
 Put $n=\sum_{a=1}^bn_a$.
 Then $I^{n+1}m=0$ in~$M$.
\end{proof}

 Let $I\subset T$ be an ideal.
 A $T$\+module $M$ is said to be \emph{$I$\+torsion} (in the sense
of~\cite[Section~1]{Pmgm}) if $M$ is $s$\+torsion for every $s\in I$.
 According to Lemma~\ref{torsion-checked-on-generators-of-ideal},
it suffices to check this condition for any chosen set of
generators~$s_j$ of the ideal $I\subset T$.
 We denote by $T\Modl_{I\tors}\subset T\Modl$ the full subcategory
of $I$\+torsion $T$\+modules.

 By Lemma~\ref{s-torsion-extension-closed}, the class of $I$\+torsion
$T$\+modules $T\Modl_{I\tors}$ is closed under subobjects, quotients,
extensions, and infinite direct sums in $T\Modl$.
 This assertion can be expressed by saying that $T\Modl_{I\tors}$ is
a \emph{localizing subcategory}, or in the terminology
of~\cite[Sections~VI.2\+-3]{Ste}, a \emph{hereditary torsion class}
in $T\Modl$.

 As usual, we denote by $\Spec T$ the prime spectrum of $T$, i.~e.,
the set of all prime ideals $\p\subset T$ (endowed with the Zariski
topology).
 The notation $T_\p=(T\setminus\p)^{-1}T$ stands for the localization
of the commutative ring $T$ at a prime ideal~$\p$.
 For a $T$\+module $M$, we put $M_\p=T_\p\ot_TM$.
 The (\emph{set-theoretic}) \emph{support} $\Supp_T M\subset\Spec T$
of a $T$\+module $M$ is defined as the set of all prime ideals
$\p\in\Spec T$ for which $M_\p\ne0$.
 One says that a $T$\+module $M$ is \emph{supported set-theoretically}
in a closed subset $Z\subset\Spec T$ if $\Supp_TM\subset Z$.

 The following lemma shows that our definition of $I$\+torsion
$T$\+modules is natural from the geometric standpoint.

\begin{lem} \label{torsion-module-support}
 A $T$\+module $M$ is $I$\+torsion if and only if its support is
contained in the closed subset\/ $\Spec T/I\subset\Spec T$.
\end{lem}

\begin{proof}
 Let $m\in M$ be an element and $J_m\subset T$ be its annihilator
ideal in~$T$.
 The element $m\in M$ is $I$\+torsion if and only if, for every
element $s\in I$, the ideal $J_m$ contains some power $s^n$ of
the element~$s$.
 In other words, the element~$m$ is $I$\+torsion if and only if
the radical of the ideal $J_m$ contains~$I$.
 It is well-known that the radical of any ideal $J\subset T$ is
equal to the intersection of all prime ideals $\p\subset T$
containing~$J$.
 So the element~$m$ is $I$\+torsion if and only if every prime ideal
containing $J_m$ contains~$I$.

 On the other hand, the element~$m$ vanishes in the localization
$M_\p$ of the module $M$ at a prime ideal~$\p$ if and only if
there exists $t\in T\setminus\p$ such that $tm=0$ in~$M$.
 In other words, $m$~vanishes in $M_\p$ if and only if $J_m$ is not
contained in~$\p$.
 So a prime ideal~$\p$ belongs to $\Supp_TM$ if and only if there
exists an element $m\in M$ with $J_m\subset\p$.

 Now if all elements $m\in M$ are $I$\+torsion, then $J_m\subset\p$
implies $I\subset\p$, so $\p\in\Spec T/I$.
 If there is an element $m\in M$ that is not $I$\+torsion, then there
exists $\p\in\Spec T$ such that $J_m\subset\p$ but $I\not\subset\p$,
hence $\p\in\Supp_TM$ but $\p\notin\Spec T/I$.
\end{proof}

 For any $T$\+module $M$ and element $s\in T$, we denote by
$\Gamma_s(M)\subset M$ the (obviously unique) maximal $s$\+torsion
submodule in~$M$.
 So we have $\Gamma_s(M)=\bigcup_{n\ge0}F^{(s)}_nM$.
 For any $T$\+module $M$ and ideal $I\subset T$, we denote by
$\Gamma_I(M)\subset M$ the (obviously unique) maximal
$I$\+torsion submodule in~$M$.
 So $\Gamma_I(M)=\bigcap_{s\in I}\Gamma_s(M)$.

 Since the full subcategory $T\Modl_{I\tors}$ is closed under extensions
in $T\Modl$, the quotient module $M/\Gamma_I(M)$ has no nonzero
$I$\+torsion elements.
 In other words, we have $\Gamma_I(M/\Gamma_I(M))=0$ for any
$T$\+module $M$ and any ideal $I\subset T$.
 Let us explain these assertions in some more detail.

 Let $N\subset M/\Gamma_I(M)$ be an $I$\+torsion $T$\+submodule in
$M/\Gamma_I(M)$.
 Denote by $L$ the preimage of $N$ under the surjective $T$\+module
map $M\rarrow M/\Gamma_I(M)$.
 Then we have a short exact sequence of $T$\+modules $0\rarrow
\Gamma_I(M)\rarrow L\rarrow N\rarrow0$.
 As the $T$\+modules $\Gamma_I(M)$ and $N$ are $I$\+torsion, it follows
by virtue of Lemma~\ref{s-torsion-extension-closed} that $L$ is
an $I$\+torsion $T$\+module, too.
 Now we have $\Gamma_I(M)\subset L\subset M$, and $\Gamma_I(M)$ is
the maximal $I$\+torsion submodule of~$M$.
 Thus $L=\Gamma_I(M)$ and $N=0$.
 
 Alternatively, one can simply say that if $m+\Gamma_I(M)$ is
an $I$\+torsion element in $M/\Gamma_I(M)$, then for every $s\in I$
there exists $n\ge1$ such that $s^nm\in\Gamma_I(M)$.
 As $\Gamma_I(M)$ is an $I$\+torsion $T$\+module, there is
another integer $n'\ge1$ such that $s^{n+n'}m=0$ in~$M$.
 Thus $m$~is an $I$\+torsion element of $M$, so $m\in\Gamma_I(M)$.

\begin{lem} \label{Gamma-I-left-exact}
 The functor\/ $\Gamma_I\:T\Modl\rarrow T\Modl$ is left exact.
 In particular, for any $T$\+module $M$ and any $T$\+submodule
$N\subset M$ one has\/ $\Gamma_I(N)=N\cap\Gamma_I(M)$.
\end{lem}

\begin{proof}
 The functor $\Gamma_I\:T\Modl\rarrow T\Modl_{I\tors}$ is left exact
as a right adjoint functor to the identity inclusion functor of
abelian categories $T\Modl_{I\tors}\rarrow T\Modl$.
 As the identity inclusion functor is exact, it follows that
the functor $\Gamma_I\:T\Modl\rarrow T\Modl$ is left exact, too.
 In particular, the functor $\Gamma_I$ takes the short exact sequence
$0\rarrow N\rarrow M\rarrow M/N\rarrow0$ to a left exact sequence
$0\rarrow\Gamma_I(N)\rarrow\Gamma_I(M)\rarrow\Gamma_I(M/N)$.
 As the natural map $\Gamma_I(M/N)\rarrow M/N$ is injective,
it follows that $\Gamma_I(N)=N\cap\Gamma_I(M)$.
 The latter equality is also easy to see directly from the definition
(the left-exactness of $\Gamma_I$ is easy to check directly from
the definition, too).
\end{proof}

\subsection{Strongly torsion modules} \label{strongly-torsion-subsecn}
 We will say that an element $m\in M$ is \emph{strongly $I$\+torsion}
if there exists an integer $n\ge0$ such that $I^{n+1}m=0$ in~$M$.
 We denote by $F^{(I)}_nM\subset M$ the submodule of all elements
annihilated by $I^{n+1}$ in $M$, or equivalently, define by induction
\begin{itemize}
\item $F^{(I)}_nM=0$ for all integers $n<0$;
\item $F^{(I)}_nM=\{\,m\in M\mid Im\subset F^{(I)}_{n-1}M\,\}$ for all
integers $n\ge0$.
\end{itemize}
 So $F^{(I)}_0M$ is the submodule of all elements annihilated by~$I$
in $M$, and $F^{(I)}_{n-1}M\subset F^{(I)}_nM$ for all $n\ge0$.
 We will say that a $T$\+module $M$ is \emph{strongly $I$\+torsion}
(or ``$I$\+torsion'' in the sense of~\cite[Section~3]{PSY}) if
all its elements are strongly $I$\+torsion, i.~e., if
$M=\bigcup_{n\ge0}F^{(I)}_nM$.
 We denote the full subcategory of strongly $I$\+torsion $T$\+modules
by $T\Modl_{I\tors}^\st\subset T\Modl$.

 It is clear from Lemma~\ref{torsion-finitely-generated-ideal} that,
for a finitely generated ideal $I\subset T$, the classes of
$I$\+torsion and strongly $I$\+torsion $T$\+modules coincide.
 In particular, it follows that, for a finitely generated ideal $I$,
the class of strongly $I$\+torsion $T$\+modules $T\Modl_{I\tors}^\st$
is closed under extensions in $T\Modl$ \,\cite[Erratum, item~(2)]{PSY}.
 For an infinitely generated ideal $I\subset T$, this \emph{need not}
be the case.

\begin{ex} \label{torsion-hrbek-counterex}
 The following example, communicated to the author by M.~Hrbek, is
reproduced here with his kind permission.
 Let $k$~be a field and $T=k[x_1,x_2,x_3,\dotsc]/J$ be the quotient
ring of the ring of polynomials in countably many variables by
the ideal of relations $J$ spanned by the polynomials $x_i^{i+1}$ for
all $i\ge1$ and $x_ix_j$ for all $i\ne j$, \ $i$, $j\ge1$.
 Let $I$ be the ideal generated by the elements $x_1$, $x_2$,
$x_3$,~\dots\ in~$T$.

 Then we have $I=\bigoplus_{n=1}^\infty I_n$, where $I_n=(x_n)\subset T$
is the principal ideal generated by~$x_n$ in~$T$.
 Each ideal $I_n$ is a $T$\+module of finite length, annihilated
by the power $I^n$ of the ideal~$I$.
 It is clear that both $I$ and $T/I$ are strongly $I$\+torsion
$T$\+modules.
 However, $T$ is \emph{not} a strongly $I$\+torsion $T$\+module,
because $I$ is not a nilpotent ideal: $I^n\ne0$ for all $n\ge1$.
 In fact, $I$ is the maximal strongly $I$\+torsion $T$\+submodule
of~$T$ (in the notation introduced several paragraphs below,
one has $\Gamma_I^\st(T)=I$).

 Now one can also take $T'=k[x_1,x_2,x_3,\dotsc]$ and consider
the ideal $I'\subset T'$ generated by the elements $x_1$, $x_2$,
$x_3$,~\dots\ in~$T'$.
 Then both $I$ and $T/I$ are strongly $I'$\+torsion $T'$\+modules,
but $T$ is not a strongly $I'$\+torsion $T'$\+module.
\end{ex}

 Another counterexample showing that the class of strongly $I$\+torsion
$T$\+modules need not be closed under extensions in $T\Modl$ when
the ideal $I$ is infinitely generated can be extracted from
Example~\ref{D-omega-operator-example} below.

 Clearly, the class of strongly $I$\+torsion $T$\+modules
$T\Modl_{I\tors}^\st$ is closed under subobjects, quotients,
and infinite direct sums in $T\Modl$.
 In other words, $T\Modl_{I\tors}^\st$ is a \emph{hereditary pretorsion
class} in $T\Modl$ (in the terminology similar
to~\cite[Sections~VI.1\+-3]{Ste}).

 For any $T$\+module $M$ and ideal $I\subset T$, we denote by
$\Gamma_I^\st(M)$ the (obviously unique) maximal strongly $I$\+torsion
submodule in~$M$.
 So we have $\Gamma_I^\st(M)=\bigcup_{n\ge0}F^{(I)}_nM$ (where the direct
union is taken over the nonnegative integers~$n$).

 Example~\ref{torsion-hrbek-counterex} shows that the quotient module
$M/\Gamma_I^\st(M)$ can contain nonzero strongly $I$\+torsion elements.
 So the $T$\+module $\Gamma_I^\st(M/\Gamma_I^\st(M))$ may well be
nonzero.

\begin{lem} \label{Gamma-I-st-left-exact}
 The functor\/ $\Gamma_I^\st\:T\Modl\rarrow T\Modl$ is left exact.
 In particular, for any $T$\+module $M$ and any $T$\+submodule
$N\subset M$ one has\/ $\Gamma_I^\st(N)=N\cap\Gamma_I^\st(M)$.
 Moreover, for every integer $n\ge0$ one has
$F^{(I)}_nN=N\cap F^{(I)}_nM$.
\end{lem}

\begin{proof}
 The proof of the first two assertions is completely similar to
that of Lemma~\ref{Gamma-I-left-exact}.
 For the final assertion, see the proof of
Lemma~\ref{Gamma-I-qu-left-exact} below.
\end{proof}

\subsection{Quite torsion modules} \label{quite-torsion-subsecn}
 Let $T$ be an associative ring and $M$ be a left $T$\+module.
 Our notation for ordinal-indexed filtrations in this paper
is slightly different from the most conventional one (as
in~\cite[Definition~6.1]{GT}).
 We prefer our somewhat unconventional notation (taken
from~\cite[Section~1.1]{Pcosh}) for the purposes of the present paper,
because it is compatible with what we think is the most natural notation
for ordinal orders of quite differential operators.
 
 An \emph{ordinal-indexed increasing filtration} on a $T$\+module $M$ is
a family of submodules $F_\beta M\subset M$ indexed by ordinals~$\beta$ 
such that $F_\gamma M\subset F_\beta M$ for all ordinals
$\gamma\le\beta$ and $M=\bigcup_\beta F_\beta M$ (where the direct union
is taken over all ordinals~$\beta$).
 These conditions imply existence of an ordinal~$\alpha$ for which
$M=\bigcup_{\beta<\alpha}F_\beta M$; the filtration $F$ is then said to
be \emph{indexed by the ordinal~$\alpha$}.

 The \emph{successive quotient modules} of an increasing filtration $F$
on $M$ indexed by an ordinal~$\alpha$ are the quotient modules
$$
 S_\beta=\gr^F_\beta M=
 F_\beta M\Big/\bigcup\nolimits_{\gamma<\beta}F_\gamma M,
 \qquad 0\le\beta<\alpha.
$$
 In particular, $S_0=F_0M$.
 The $T$\+module $M$ is said to be \emph{filtered by}
the $T$\+modules~$S_\beta$.
 In a different terminology, the $T$\+module $M$ is said to be
a \emph{transfinitely iterated extension} (\emph{in the sense of
the inductive limit}) of the $T$\+modules~$S_\beta$.

 Let $\sS\subset T\Modl$ be a class of $T$\+modules.
 It is convenient to assume that the zero module belongs to~$\sS$.
 Then the class of all $T$\+modules filtered by modules (isomorphic to)
modules from $\sS$ is denoted by $\Fil(\sS)\subset T\Modl$.

\begin{lem} \label{Fil-of-quotient-closed}
 Let $T$ be an associative ring and\/ $\sS\subset T\Modl$ be a class
of $T$\+modules closed under quotients.
 Then the following four classes of $T$\+modules coincide:
\begin{enumerate}
\item the closure of\/ $\sS$ under quotients, extensions, and infinite
direct sums in $T\Modl$;
\item the closure of\/ $\sS$ under extensions and filtered direct limits
in $T\Modl$;
\item the closure of\/ $\sS$ under transfinitely iterated extensions
in $T\Modl$;
\item the class\/ $\Fil(\sS)\subset T\Modl$.
\end{enumerate}
 If the class\/ $\sS$ is also closed under submodules in $T\Modl$, then
the class of $T$\+modules defined by any one of
the rules~\textup{(1\+-4)} is closed under submodules as well.
\end{lem}

\begin{proof}
 The inclusions $(1)\supset(2)\supset(3)\supset(4)$ hold for any class
of $T$\+modules~$\sS$ (containing the zero module).
 Indeed, filtered direct limits can be constructed as quotients of
infinite direct sums, and transfinitely iterated extensions can be
obtained by combining extensions and filtered direct limits.
 Moreover, the class $\Fil(\sS)$ is closed under transfinitely iterated
extensions for any class of $T$\+modules~$\sS$; so one always has
$(3)=(4)$.

 It remains to show that the class $\Fil(\sS)$ is closed under
quotients, extensions, and infinite direct sums whenever the class
$\sS$ is closed under quotients.
 Both extensions and infinite direct sums are special cases of
transfinitely iterated extensions.
 Concerning quotients, any ordinal-indexed filtration $F$ on
a $T$\+module $M$ induces an ordinal-indexed filtration $F$ on any
quotient $T$\+module $N$ of $M$, and the $T$\+modules $\gr^F_\beta N$
are quotient modules of the $T$\+modules $\gr^F_\beta M$.
 Hence the class $\Fil(\sS)$ is closed under quotients whenever
the class $\sS$~is.

 Similarly, any ordinal-indexed filtration $F$ on $M$ induces
an ordinal-indexed filtration $F$ on any submodule $L$ of $M$,
and the $T$\+modules $\gr^F_\beta L$ are submodules of the $T$\+modules
$\gr^F_\beta M$.
 This proves the last assertion of the lemma.
\end{proof}

 Now we can return to our setting of a commutative ring $T$ with
an ideal $I\subset T$.
 For any $T$\+module $M$, we define by transfinite induction
\begin{itemize}
\item $F^{(I)}_0M=\{\,m\in M\mid Im=0\,\}$;
\item $F^{(I)}_\alpha M=\{\,m\in M\mid Im\subset
\bigcup_{\beta<\alpha}F^{(I)}_\beta M\,\}$ for all ordinals $\alpha>0$.
\end{itemize}
 This notation agrees with the notation $F^{(I)}_nM$ for nonnegative
integers~$n$ from Section~\ref{strongly-torsion-subsecn}.
 One has $F^{(I)}_\beta M\subset F^{(I)}_\alpha M$ for all ordinals
$\beta\le\alpha$.

 We will say that an element $m\in M$ is \emph{quite $I$\+torsion} if
there exists an ordinal~$\beta$ such that $m\in F^{(I)}_\beta M$.
 A $T$\+module $M$ is said to be \emph{quite $I$\+torsion} if all
its elements are quite $I$\+torsion, i.~e.,
$M=\bigcup_\beta F^{(I)}_\beta M$.
 If this is the case, then there exists an ordinal~$\alpha$ such that
$M=\bigcup_{\beta<\alpha} F^{(I)}_\beta M$.
 We denote the full subcategory of quite $I$\+torsion $T$\+modules
by $T\Modl_{I\tors}^\qu\subset T\Modl$.

 It will follow from Theorem~\ref{all-ordinal-orders-possible} below
that the minimal ordinal~$\alpha$ satisfying the condition from
the previous paragraph for a given $T$\+module $M$ can be arbitrarily
large, depending on a commutative ring $T$ with an ideal $I\subset T$.
 When the ring $T$ and the ideal $I$ are fixed, the ordinal~$\alpha$
is bounded by the successor cardinal of the cardinality of a set of 
generators of~$I$;
see Proposition~\ref{quite-torsion-module-filtration-bound}.

 Consider the class $T/I\Modl\subset T\Modl$ of all $T$\+modules
annihilated by~$I$.
 By Lemma~\ref{Fil-of-quotient-closed}, the class of quite $I$\+torsion
$T$\+modules $T\Modl_{I\tors}^\qu$ is the closure of the class
$T/I\Modl$ under extensions and filtered direct limits, or equivalently,
under extensions, infinite direct sums, and quotients.
 The class $T\Modl_{I\tors}^\qu$ is closed under subobjects, quotients,
extensions, and infinite direct sums in $T\Modl$.
 In other words, $T\Modl_{I\tors}^\qu$ is always a localizing
subcategory, or equivalently, a hereditary torsion class in $T\Modl$.

 Obviously, any strongly $I$\+torsion $T$\+module is quite $I$\+torsion.
 Example~\ref{torsion-hrbek-counterex} shows that a quite $I$\+torsion
$T$\+module \emph{need not} be strongly $I$\+torsion.

 The following proposition characterizes quite $I$\+torsion modules by
a property resembling the \emph{T\+nilpotence} condition from Bass'
paper~\cite{Bas}.

\begin{prop} \label{quite-torsion-T-nilpotence}
 Let $T$ be a commutative ring and $I\subset T$ be an ideal.
 Then a $T$\+module $M$ is quite $I$\+torsion if and only if, for
every element $m\in M$ and every sequence of elements $s_0$, $s_1$,
$s_2$,~\dots~$\in I$ (indexed by the nonnegative integers) there exists
an integer $n\ge0$ such that $s_ns_{n-1}\dotsm s_1s_0m=0$ in~$M$.
 It suffices to check the latter condition for sequences of
elements~$s_i$, \,$i\ge0$, belonging to any given set of generators
$G\subset I$ of the ideal $I\subset T$.
\end{prop}

\begin{proof}
 ``If'': for the sake of contradiction, assume that there is
an element $m\in M$ such that $m\notin F^{(I)}_\alpha M$ for every
ordinal~$\alpha$.
 Then there exists an element $s\in G$ such that $sm\notin
F^{(I)}_\alpha M$ for every ordinal~$\alpha$.
 Indeed, if for every $s\in G$ there existed an ordinal~$\beta_s$
such that $sm\in F^{(I)}_{\beta_s}M$ then, choosing an ordinal~$\alpha$
such that $\alpha>\beta_s$ for all $s\in G$, we would have
$sm\in\bigcup_{\beta<\alpha} F^{(I)}_\beta M$ for all $s\in G$.
 This would imply the inclusion $Im\subset\bigcup_{\beta<\alpha}
F^{(I)}_\beta M$, and it would follow that $m\in F^{(I)}_\alpha M$
by the definition.
 Now put $s_0=s$.
 Similarly, there exists an element $s_1\in G$ such that $s_1s_0m\notin
F^{(I)}_\alpha M$ for every ordinal~$\alpha$.
 Proceeding in this way, we construct a sequence of elements $s_0$, 
$s_1$, $s_2$,~\dots~$\in G$ such that $s_ns_{n-1}\dotsm s_1s_0m\ne0$
in $M$ for all $n\ge0$.

 ``Only if'': suppose given an element $m\in M$ and a sequence of
elements $s_0$, $s_1$, $s_2$,~\dots~$\in I$.
 Let~$\alpha$ be an ordinal such that $m\in F^{(I)}_\alpha M$.
 If $\alpha>0$ then, by the definition of $F^{(I)}_\alpha M$, there
exists an ordinal $\beta_0<\alpha$ such that
$s_0m\in F^{(I)}_{\beta_0}M$.
 Similarly, if $\beta_0>0$ then there exists an ordinal
$\beta_1<\beta_0$ such that $s_1s_0m\in F^{(I)}_{\beta_1}M$.
 As any descending chain of ordinals terminates, we can conclude
that there exists an integer $n\ge0$ such that $\beta_n=0$.
 Then $s_ns_{n-1}\dotsm s_0m\in F^{(I)}_0M$, and it follows that
$s_{n+1}s_ns_{n-1}\dotsm s_0m=0$ in~$M$.
\end{proof}

 It follows from Lemma~\ref{s-torsion-extension-closed} or
Lemma~\ref{quite-torsion-T-nilpotence} that all quite $I$\+torsion
$T$\+modules are $I$\+torsion.
 So, in view of Lemma~\ref{torsion-finitely-generated-ideal}, for
a finitely generated ideal $I\subset T$, all the three classes of
$I$\+torsion, quite $I$\+torsion, and strongly $I$\+torsion $T$\+modules
coincide.
 For an infinitely generated ideal $I$, an $I$\+torsion $T$\+module
\emph{need not} be quite $I$\+torsion, as the following counterexample
demonstrates.

\begin{ex} \label{torsion-not-quite-torsion-counterex}
 Let $T=k[x_1,x_2,x_3,\dotsc]$ be the ring of polynomials in countably
many variables over a field~$k$, and let $I\subset T$ be the ideal
generated by the elements $x_1$, $x_2$, $x_3$,~\dots\ in~$T$.
 Consider the ideal $J\subset T$ generated by the elements
$x_1$, $x_2^2$, $x_3^3$,~\dots\ in $T$, or alternatively, the ideal
$J'\subset T$ generated by the elements $x_1^2$, $x_2^2$,
$x_3^2$,~\dots\ in~$T$.
 Then the quotient $T$\+module $M=T/J$ or $M=T/J'$ is $I$\+torsion (as
one can see from Lemma~\ref{torsion-checked-on-generators-of-ideal}).
 However, the $T$\+module $M$ is \emph{not} quite $I$\+torsion.
 In fact, in both the cases $M=T/J$ or $M=T/J'$, \emph{there are no
nonzero elements annihilated by~$I$ in~$M$}.
 So $F_0^{(I)}M=0$, and consequently, $F_\beta^{(I)}M=0$ for all
ordinals~$\beta$.
\end{ex}

 Another counterexample showing that $I$\+torsion $T$\+modules need
not be quite $I$\+torsion when the ideal $I$ is infinitely generated
can be extracted from Example~\ref{D-infinity-operator-example} below.

 Thus we have proved the strict inclusions
\begin{equation} \label{torsion-modules-strict-inclusions-repeated}
 T\Modl_{I\tors}^\st\varsubsetneq T\Modl_{I\tors}^\qu
 \varsubsetneq T\Modl_{I\tors}
\end{equation}
in the general case, as promised
in~\eqref{torsion-modules-inclusions-inequalities}.

 For any $T$\+module $M$ and any ideal $I\subset T$, we denote by
$\Gamma_I^\qu(M)$ the (obviously unique) maximal quite $I$\+torsion
submodule in~$M$.
 So we have $\Gamma_I^\qu(M)=\bigcup_\beta F^{(I)}_\beta M$ (where
the direct union is taken over all ordinals~$\beta$).

 Since the full subcategory $T\Modl_{I\tors}^\qu$ is closed under
extensions in $T\Modl$, the quotient module $M/\Gamma_I^\qu(M)$ has
no nonzero quite $I$\+torsion elements.
 In other words, we have $\Gamma_I^\qu(M/\Gamma_I^\qu(M))=0$ for
any $T$\+module $M$ and any ideal $I\subset T$.

 The argument supporting the assertions in the previous paragraph is
similar to the one in Section~\ref{torsion-modules-subsecn}.
 Any quite $I$\+torsion element in a $T$\+module spans a quite
$I$\+torsion cyclic submodule.
 Let $N\subset M/\Gamma_I^\qu(M)$ be a quite $I$\+torsion
$T$\+submodule in $M/\Gamma_I^\qu(M)$.
 Denote by $L$ the preimage of $N$ under the surjective $T$\+module
map $M\rarrow M/\Gamma_I^\qu(M)$.
 Then the short exact sequence of $T$\+modules $0\rarrow\Gamma_I^\qu(M)
\rarrow L\rarrow N\rarrow0$ with quite $I$\+torsion $T$\+modules
$\Gamma_I^\qu(M)$ and $N$ implies that $L$ is also a quite $I$\+torsion
$T$\+module (following the arguments above in this section based on
Lemma~\ref{Fil-of-quotient-closed}).
 Now we have $\Gamma_I^\qu(M)\subset L\subset M$, and $\Gamma_I^\qu(M)$
is the maximal quite $I$\+torsion submodule of~$M$.
 Thus $L=\Gamma_I^\qu(M)$ and $N=0$.

\begin{lem} \label{Gamma-I-qu-left-exact}
 The functor\/ $\Gamma_I^\qu\:T\Modl\rarrow T\Modl$ is left exact.
 In particular, for any $T$\+module $M$ and any $T$\+submodule
$N\subset M$ one has\/ $\Gamma_I^\qu(N)=N\cap\Gamma_I^\qu(M)$.
 Moreover, for every ordinal~$\beta$ one has
$F^{(I)}_\beta N=N\cap F^{(I)}_\beta M$.
\end{lem}

\begin{proof}
 The proof of the first two assertions is completely similar to
that of Lemmas~\ref{Gamma-I-left-exact}
and~\ref{Gamma-I-st-left-exact}.
 To prove the final assertion, one can observe that, for any
ordinal~$\alpha$, the functor $M\longmapsto\bigcup_{\beta<\alpha}
F^{(I)}_\beta M$ is left exact as the right adjoint functor to
the identity inclusion of the full subcategory of $T$\+modules $L$
for which $L=\bigcup_{\beta<\alpha}F^{(I)}_\beta L$ (which is also
a hereditary pretorsion class and an abelian category) into $T\Modl$.
 Alternatively, the assertion is easy to see directly from
the definition.
\end{proof}

\Section{Three Classes of Quasi-Modules}

 A more common terminology for what we call (\emph{strong})
\emph{quasi-modules} is ``differential
bimodules''~\cite[Section~1.1]{BB}.
 We prefer to use the terms ``quasi-modules'' and ``quasi-algebras''
(with various qualifiers) in order to avoid a confusion with
DG\+modules and DG\+algebras in the contexts such as
in~\cite[Sections~2.3\+-2.4]{Pedg} and~\cite{Pdomc}.

 Let $K\rarrow R$ be a homomorphism of commutative rings.
 We will denote by $R\bMod R$ the abelian category of
$R$\+$R$\+bimodules.

 We will say that an $R$\+$R$\+bimodule $B$ is
an \emph{$R$\+$R$\+bimodule over~$K$} if the left and right actions
of $K$ in $B$ agree.
 The category of $R$\+$R$\+bimodules over $K$ will be denoted by
$R_K\bMod{}_KR=(R\ot_KR)\Modl\subset R\bMod R$.

 Put $T=R\ot_KR$, so $R_K\bMod{}_KR=T\Modl$.
 Denote by $I\subset T$ the kernel ideal of the natural (multiplication)
ring homomorphism $R\ot_KR\rarrow R$.

\begin{lem} \label{diagonal-ideal-generators}
 Let $r_j\in R$ be a set of elements generating $R$ as
a unital $K$\+algebra.
 Then the ideal $I\subset T$ is generated by the elements
$r_j\ot1-1\ot r_j\in R\ot_KR$.
\end{lem}

\begin{proof}
 By the definition of the tensor product, the kernel of the map
$R\ot_KR\rarrow R\ot_RR=R$, viewed as a $K$\+module (or an abelian
group), is generated by the elements $pq\ot r-p\ot qr\in R\ot_KR$,
where $p$, $q$, $r\in R$.
 Therefore, the $R$\+$R$\+bimodule $I$ is generated by the elements
$q\ot1-1\ot q\in R\ot_KR$, where $q\in R$.
 For any element $l\in K$, we have $l\ot1-1\ot l=0$ in $R\ot_KR$.
 It remains to compute that, for any two elements $p$ and $q\in R$,
the element $pq\ot1-1\ot pq=(pq\ot1-p\ot q)+(p\ot q-1\ot pq)=
p(q\ot1-1\ot q)+(p\ot1-1\ot p)q$ belongs to the $R$\+$R$\+subbimodule
generated by the elements $p\ot1-1\ot p$ and $q\ot1-1\ot q$
in $R\ot_KR$.
\end{proof}

 Given two $R$\+$R$\+bimodules $A$ and $B$, we denote by
$A\ot_RB$ the tensor product of the $R$\+module $A$ with its right
$R$\+module structure and the $R$\+module $B$ with its left $R$\+module
structure.
 The tensor product $A\ot_RB$ is viewed as an $R$\+$R$\+bimodule
with the left action of $R$ in $A\ot_RB$ induced by the left action
of $R$ in $A$ and the right action of $R$ in $A\ot_RB$ induced by
the right action of $R$ in~$B$.
 This is the usual and the most intuitive notation.

\subsection{Quasi-modules} \label{quasi-modules-subsecn}
 Let $R$ be a commutative ring and $B$ be an $R$\+$R$\+bimodule.
 Given an element $r\in R$, we define by induction
\begin{itemize}
\item $F^{(r)}_nB=0$ for all integers $n<0$;
\item $F^{(r)}_nB=\{\,b\in B\mid rb-br\in F^{(r)}_{n-1}B\,\}$
for all integers $n\ge0$.
\end{itemize}
 So $F^{(r)}_nB$ is an $R$\+$R$\+subbimodule in $B$, and
$F^{(r)}_{n-1}B\subset F^{(r)}_nB$ for all $n\ge0$.

 For a pair of elements $r\in R$ and $b\in B$, denote for brevity by
$\theta_r(b)$ the element $rb-br\in B$,
\begin{equation} \label{theta-defined}
 \theta_r(b)=rb-br.
\end{equation}
 Then one has $b\in F^{(r)}_nB$ if and only if $(\theta_r)^{n+1}(b)=0$
in~$B$.

 Given a homomorphism of commutative rings $K\rarrow R$ such that
$B$ is an $R$\+$R$\+bimodule over $K$, we can consider $B$ is a module
over the ring $T=R\ot_KR$.
 Put $s=r\ot1-1\ot r\in T$.
 Then the notation $F^{(r)}_nB$ stands for what would be denoted by
$F^{(s)}_nB$ in the notation of Section~\ref{torsion-modules-subsecn}.

 We will say that an $R$\+$R$\+bimodule $B$ is a \emph{quasi-module}
over $R$ if $B=\bigcup_{n\ge0} F^{(r)}_nB$ for every $r\in R$.
 The same terminology is used in~\cite[Section~1.4]{Pdomc}.
 So $B$ is a quasi-module over $R$ if and only if, for every $b\in B$
and $r\in R$, there exists an integer $n\ge0$ such that
$(\theta_r)^{n+1}(b)=0$.
 In other words, the unique maximal quasi-module subbimodule of
an arbitrary $R$\+$R$\+bimodule $B$ can be constructed as
the intersection $\bigcap_{r\in R}\bigcup_{n\ge0} F^{(r)}_nB\subset B$;
and $B$ is a quasi-module if and only if
$B=\bigcap_{r\in R}\bigcup_{n\ge0} F^{(r)}_nB$.

\begin{cor} \label{quasi-module-torsion-module}
 Let $K\rarrow R$ be a homomorphism of commutative rings and $B$ be
an $R$\+$R$\+bimodule over~$K$.
 Put $T=R\ot_KR$, and let $I\subset T$ be the kernel ideal of
the ring homomorphism $R\ot_KR\rarrow R$.
 Then $B$ is a quasi-module over $R$ if and only if $B$ is
an $I$\+torsion $T$\+module.
\end{cor}

\begin{proof}
 Follows from Lemmas~\ref{torsion-checked-on-generators-of-ideal}
and~\ref{diagonal-ideal-generators}.
\end{proof}

 Notice that the property of an $R$\+$R$\+bimodule to be a quasi-module
over $R$, by the definition, does not depend on any base ring~$K$.
 On the other hand, the criterion in terms of being an $I$\+torsion
$T$\+module, as per Corollary~\ref{quasi-module-torsion-module},
of course, depends on~$K$.
 So one can use any choice of a commutative ring $K$ with a ring
homomorphism $K\rarrow R$ for which $B$ is an $R$\+$R$\+bimodule
over $K$ to check the quasi-module condition
using Corollary~\ref{quasi-module-torsion-module}, and the result
will be the same.

 In view of the discussion in Section~\ref{torsion-modules-subsecn},
it follows from Corollary~\ref{quasi-module-torsion-module}
(e.~g., for the base ring $K=\boZ$) that the full subcategory of
quasi-modules is closed under subobjects, quotients, extensions, and
infinite direct sums in $R\bMod R$.
 So quasi-modules form a localizing subcategory, or in the terminology
of~\cite[Sections~VI.2\+-3]{Ste}, a hereditary torsion class in
$R\bMod R$.

\begin{cor}
 Let $K\rarrow R$ be a homomorphism of commutative rings and $B$ be
an $R$\+$R$\+bimodule over~$K$.
 Put $T=R\ot_KR$, and let $I\subset T$ be the kernel ideal of
the ring homomorphism $R\ot_KR\rarrow R$.
 Then $B$ is a quasi-module over $R$ if and only if $B$ as a $T$\+module
is supported set-theoretically in the image of the diagonal closed
immersion of affine schemes\/ $\Spec R\rarrow
\Spec R\times_{\Spec K}\Spec R=\Spec T$.
\end{cor}

\begin{proof}
 Follows from Lemma~\ref{torsion-module-support}
and Corollary~\ref{quasi-module-torsion-module}.
\end{proof}

\begin{prop} \label{quasi-modules-tensor-product}
 Let $A$ and $B$ be two quasi-modules over~$R$.
 Then the $R$\+$R$\+bimodule of tensor product $A\ot_RB$ is also
a quasi-module over~$R$.
\end{prop}

\begin{proof}
 Given an element $r\in R$, define an increasing filtration $G^{(r)}$
on the $R$\+$R$\+bi\-mod\-ule $A\ot_RB$ by the rule
$$
 G^{(r)}_n(A\ot_RB)=
 \sum\nolimits_{i+j=n}\im(F^{(r)}_iA\ot_R F^{(r)}_jB
 \to A\ot_RB), \qquad n\in\boZ,
$$
where $\im({-})$ stands for the image of the natural map.
 Then, for every $n\ge0$, there is a natural surjective map of
$R$\+$R$\+bimodules
$$
 \bigoplus\nolimits_{i+j=n}
 (F^{(r)}_iA/F^{(r)}_{i-1}A)\ot_R(F^{(r)}_jB/F^{(r)}_{j-1}B)
 \lrarrow G^{(r)}_n(A\ot_RB)/G^{(r)}_{n-1}(A\ot_RB).
$$
 The successive quotient $R$\+$R$\+bimodules
$F^{(r)}_iA/F^{(r)}_{i-1}A$ and $F^{(r)}_jB/F^{(r)}_{j-1}B$ have
the property that the left and right actions of the element $r\in R$
in each of them agree.
 In follows that the left and right actions of~$r$ in the successive
quotient $R$\+$R$\+bimodule $G^{(r)}_n(A\ot_RB)/G^{(r)}_{n-1}(A\ot_RB)$
agree as well.
 Therefore, $G^{(r)}_n(A\ot_RB)\subset F^{(r)}_n(A\ot_RB)$
for all $n\ge0$.
 Furthermore, we have $A\ot_RB=\bigcup_{n\ge0}G^{(r)}_n(A\ot_RB)$,
since $A=\bigcup_{i\ge0}F^{(r)}_iA$ and $B=\bigcup_{j\ge0}F^{(r)}_jB$.
 Thus $A\ot_RB=\bigcup_{n\ge0}F^{(r)}_n(A\ot_RB)$, as desired.
\end{proof}

\subsection{Strong quasi-modules} \label{strong-quasi-modules-subsecn}
 Let $R$ be a commutative ring and $B$ be an $R$\+$R$\+bi\-mod\-ule.
 We define a natural increasing filtration $F$ on $B$ (indexed by
the integers $n\in\boZ$) by the rules
\begin{itemize}
\item $F_nB=0$ for all integers $n<0$;
\item $F_nB=\{\,b\in B\mid rb-br\in F_{n-1}B$ for all $r\in R\,\}$
for all integers $n\ge0$.
\end{itemize}
 So $F_nB$ is an $R$\+$R$\+subbimodule in $B$, and
$F_{n-1}B\subset F_nB$ for all $n\ge0$.
 In particular, $F_0B$ is the (obviously unique) maximal
$R$\+$R$\+subbimodule in $B$ on which the left and right actions
of $R$ agree.

 In the notation of
formula~\eqref{theta-defined} from Section~\ref{quasi-modules-subsecn},
an element $b\in B$ belongs to $F_nB$ if and only if $\theta_{r_n}
\theta_{r_{n-1}}\dotsm\theta_{r_1}\theta_{r_0}(b)=0$ for all elements
$r_0$, $r_1$,~\dots, $r_n\in R$.

 We will say that $B$ is a \emph{strong quasi-module} over $R$ if
$B=\bigcup_{n\ge0}F_nB$.
 What we call strong quasi-modules in this paper were called simply
``quasi-modules'' in the paper~\cite[Section~2.3]{Pedg}.

 Given a homomorphism of commutative rings $K\rarrow R$ such that
$B$ is an $R$\+$R$\+bimodule over $K$, we can consider $B$ as a module
over the ring $T=R\ot_KR$.
 Let $I\subset R$ be the kernel ideal of the natural ring homomorphism
$R\ot_KR\rarrow R$.
 Then the notation $F_nB$ stands for what would be denoted by
$F^{(I)}_nB$ in the notation of Section~\ref{strongly-torsion-subsecn}.
 This is clear from Lemma~\ref{diagonal-ideal-generators}.
 Thus $B$ is a strong quasi-module over $R$ if and only if $B$ is
a strongly $I$\+torsion $T$\+module.

 Let $B$ be an $R$\+$R$\+bimodule over~$K$.
 Assuming that $R$ is a finitely generated $K$\+algebra,
Lemmas~\ref{torsion-finitely-generated-ideal}
and~\ref{diagonal-ideal-generators} imply that $B$ is a quasi-module
over $R$ if and only if $B$ is a strong quasi-module over~$R$.
 In particular, it follows that, for a finitely generated commutative
$K$\+algebra $R$, the class of strong quasi-modules is closed under
extensions in $R_K\bMod{}_KR$.
 For an infinitely generated commutative $K$\+algebra $R$, this
\emph{need not} be the case.

\begin{ex} \label{quasi-module-hrbek-counterex}
 The following example is based on
Example~\ref{torsion-hrbek-counterex}.
 Let $K=k$ be a field and $R=k[x_1,x_2,x_3,\dotsc]$ be the ring of
polynomials in countably many variables over~$k$.
 The ring $T=R\ot_kR$ can be naturally identified with the ring of
polynomials in two countably infinite families of variables,
$T=k[y_1,y_2,y_3,\dotsc;z_1,z_2,z_3,\dotsc]$; under the identification
of $R$\+$R$\+bimodules over~$k$ with $T$\+modules, the left action of
the elements $x_i\in R$ corresponds to the action of the elements
$y_i\in T$, and the right action of the elements $x_i\in R$ corresponds
to the action of the elements $z_i\in T$, \,$i\ge1$.
 So the ring homomorphism $R\ot_kR\rarrow R$ takes $y_i$ and~$z_i$
to~$x_i$, and the kernel ideal $I\subset T$ of this ring homomorphism
is generated by the elements $t_i=y_i-z_i\in T$.

 Let $J\subset T$ be the ideal generated by the polynomials
$t_i^{i+1}\in T$ for all $i\ge1$ and $t_it_j\in T$ for all $i\ne j$,
\,$i$, $j\ge1$.
 So we have $J\subset I\subset T$.
 Then $T/I=R$ is obviously a strong quasi-module over~$R$.
 Let us check that $I/J$ is a strong quasi-module over $R$, i.~e.,
a strongly $I$\+torsion $T$\+module.
 Indeed, we have a $T$\+module direct sum decomposition
$I/J=\bigoplus_{n=1}^\infty I_n/J$, where $I_n=(t_n)+J\subset T$
is the ideal generated by the element $t_n\in T$ and the ideal~$J$.
 Now the $T$\+module $I_n/J$ is annihilated by the power $I^n$ of
the ideal~$I$.
 A direct sum of strongly $I$\+torsion $T$\+modules is a strongly
$I$\+torsion $T$\+module.

 However, the extension $T/J$ of the $R$\+$R$\+bimodules $T/I$ and $I/J$
is \emph{not} a strong quasi-module over~$R$; i.~e., $T/J$ is not
a strongly $I$\+torsion $T$\+module.
 Indeed, the ideal $I/J$ in the ring $T/J$ is not nilpotent: one has
$(I/J)^n\ne0$ for all $n\ge1$.
 So, for any integer $n\ge0$, the element $1\in T/J$ does not belong
to $F_n(T/J)=F_n^{(I)}(T/J)$.
\end{ex}

 Another counterexample showing that the class of strong quasi-modules
over $R$ need not be closed under extensions in $R_K\bMod{}_KR$ when
$R$ is an infinitely generated commutative $K$\+algebra can be
extracted from Example~\ref{D-omega-operator-example} below.

 Clearly, the class of strong quasi-modules over $R$ is closed under
subobjects, quotients, and infinite direct sums in $R\bMod R$.
 In other words, strong quasi-modules form a hereditary pretorsion
class (in the terminology similar to~\cite[Sections~VI.1\+-3]{Ste}).

\begin{prop} \label{strong-quasi-modules-tensor-product}
 Let $A$ and $B$ be two strong quasi-modules over~$R$.
 Then the $R$\+$R$\+bimodule of tensor product $A\ot_RB$ is also
a strong quasi-module over~$R$.
\end{prop}

\begin{proof}
 This is~\cite[Lemma~2.1]{Pedg}.
 The argument is similar to the proof of
Proposition~\ref{quasi-modules-tensor-product}.
 Define an increasing filtration $G$ on the $R$\+$R$\+bimodule
$A\ot_RB$ by the rule
$$
 G_n(A\ot_RB)=\sum\nolimits_{i+j=n}\im(F_iA\ot_R F_jB\to A\ot_RB),
 \qquad n\in\boZ.
$$
 Then, for every $n\ge0$, there is a natural surjective map of
$R$\+$R$\+bimodules
$$
 \bigoplus\nolimits_{i+j=n} (F_iA/F_{i-1}A)\ot_R(F_jB/F_{j-1}B)
 \lrarrow G_n(A\ot_RB)/G_{n-1}(A\ot_RB).
$$
 Therefore, the left and right actions of $R$ in the successive
quotient $R$\+$R$\+bimodule $G_n(A\ot_RB)/G_{n-1}(A\ot_RB)$ agree.
 It follows that $G_n(A\ot_RB)\subset F_n(A\ot_RB)$
for all $n\ge0$.
 Furthermore, we have $A\ot_RB=\bigcup_{n\ge0}G_n(A\ot_RB)$,
since $A=\bigcup_{i\ge0}F_iA$ and $B=\bigcup_{j\ge0}F_jB$ by
assumptions.
 Thus $A\ot_RB=\bigcup_{n\ge0}F_n(A\ot_RB)$, as desired.
\end{proof}

\subsection{Quite quasi-modules} \label{quite-quasi-modules-subsecn}
 Let $R$ be a commutative ring and $B$ be an $R$\+$R$\+bi\-mod\-ule.
 We define a natural ordinal-indexed increasing filtration $F$ on $B$
by the rules
\begin{itemize}
\item $F_0B=\{\,b\in B\mid rb-br=0$ for all $r\in R\,\}$;
\item $F_\alpha B=\{\,b\in B\mid rb-br\in
\bigcup_{\beta<\alpha}F_\beta B$ for all $r\in R\,\}$ for all
ordinals $\alpha>0$.
\end{itemize}
 This notation agrees with the notation $F_nB$ for nonnegative
integers~$n$ from Section~\ref{strong-quasi-modules-subsecn}.
 One has $F_\beta B\subset F_\alpha B$ for all ordinals
$\beta\le\alpha$.

 We will say that $B$ is a \emph{quite quasi-module} over $R$ if
$B=\bigcup_\beta F_\beta B$, where the direct union is taken over
all ordinals~$\beta$.
 If $B$ is a quite quasi-module over $R$, then there exists
an ordinal~$\alpha$ such that $B=\bigcup_{\beta<\alpha}F_\beta B$.
 It will follow from Theorem~\ref{all-ordinal-orders-possible} below
that the minimal ordinal~$\alpha$ satisfying this condition for a given
$R$\+$R$\+bimodule $B$ can be arbitrarily large, depending on
a ring~$R$.
 When the ring $R$ is fixed, the ordinal $\alpha$~is bounded;
see Corollary~\ref{quite-quasi-module-filtration-bound}.

 Given a homomorphism of commutative rings $K\rarrow R$ such that
$B$ is an $R$\+$R$\+bimodule over $K$, we can consider $B$ is a module
over the ring $T=R\ot_KR$.
 Let $I\subset R$ be the kernel ideal of the natural ring homomorphism
$R\ot_KR\rarrow R$.
 Then the notation $F_\alpha B$ stands for what would be denoted by
$F^{(I)}_\alpha B$ in the notation of
Section~\ref{quite-torsion-subsecn}.
 This follows from Lemma~\ref{diagonal-ideal-generators}.
 Thus $B$ is a quite quasi-module over $R$ if and only if $B$ is
a quite $I$\+torsion $T$\+module.

 Any $R$\+module can be viewed as an $R$\+$R$\+bimodule in which
the left and right actions of the ring $R$ agree; so we have
$R\Modl\subset R_K\bMod{}_KR$.
 In view of the preceding paragraph and the discussion in
Section~\ref{quite-torsion-subsecn}, the class of quite quasi-modules
in $R_K\bMod{}_KR$ is the closure of the class $R\Modl\subset
R_K\bMod{}_KR$ under extensions and filtered direct limits, or
equivalently, under extensions, infinite direct sums, and quotients.
 The class of quite quasi-modules is closed under subobjects,
quotients, extensions, and infinite direct sums in $R\bMod R$.
 In other words, quite quasi-modules form a localizing subcategory, or
equivalently, a hereditary torsion class in $R_K\bMod{}_KR$.

 Obviously, any strong quasi-module is a quite quasi-module over~$R$.
 Example~\ref{quasi-module-hrbek-counterex} shows that a quite
quasi-module need not be a strong quasi-module.

 The following corollary is stated in the notation of
formula~\eqref{theta-defined} from the beginning of
Section~\ref{quasi-modules-subsecn}.

\begin{cor} \label{quite-quasi-module-T-nilpotence}
 Let $R$ be a commutative ring and $B$ be an $R$\+$R$\+bimodule.
 Then $B$ is a quite quasi-module over $R$ if and only if, for every
element $b\in B$ and every sequence of elements $r_0$, $r_1$,
$r_2$,~\dots~$\in R$ (indexed by the nonnegative integers) there
exists an integer $n\ge0$ such that
$\theta_{r_n}\theta_{r_{n-1}}\dotsm\theta_{r_1}\theta_{r_0}(b)=0$
in~$B$.
 If $K\rarrow R$ is a homomorphism of commutative rings such that
$B$ is an $R$\+$R$\+bimodule over~$K$, then it suffices to check
the previous condition for sequences of elements~$r_i$, \,$i\ge0$,
belonging to any given set of generators $G\subset R$ of the unital
$K$\+algebra~$R$.
\end{cor}

\begin{proof}
 Follows from Proposition~\ref{quite-torsion-T-nilpotence}
and Lemma~\ref{diagonal-ideal-generators}.
\end{proof}

 All quite quasi-modules over $R$ are quasi-modules over $R$
(by Corollary~\ref{quite-quasi-module-T-nilpotence}; or since all
quite $I$\+torsion $T$\+modules are $I$\+torsion,
see the discussion in Section~\ref{quite-torsion-subsecn}).
 For a commutative ring homomorphism $K\rarrow R$ making $R$
a finitely generated $K$\+algebra, an $R$\+$R$\+bimodule over $K$
is a quasi-module if and only if it is a quite quasi-module and
if and only if it is a strong quasi-module (by
Lemmas~\ref{torsion-finitely-generated-ideal}
and~\ref{diagonal-ideal-generators}, cf.\
Section~\ref{quite-torsion-subsecn}).
 Generally speaking, a quasi-module over $R$ \emph{need not} be
a quite quasi-module, as the following counterexample demonstrates.

\begin{ex}
 This example is based on
Example~\ref{torsion-not-quite-torsion-counterex}.
 Let $K=k$ be a field and $R=k[x_1,x_2,x_3,\dotsc]$ be the ring of
polynomials in countably many variables over~$k$.
 Consider the ring $T=R\ot_kR$; we use the notation
$T=k[y_1,y_2,y_3,\dotsc;z_1,z_2,z_3,\dotsc]$ and $t_i=y_i-z_i$ from
Example~\ref{quasi-module-hrbek-counterex}.
 Consider the ideal $J\subset T$ generated by the elements $t_1$,
$t_2^2$, $t_3^3$,~\dots~$\in T$, or alternatively, the ideal
$J'\subset T$ generated by the elements $t_1^2$, $t_2^2$,
$t_3^2$,~\dots~$\in T$.
 Then the quotient $T$\+module $B=T/J$ or $B=T/J'$ is $I$\+torsion
by Lemma~\ref{torsion-checked-on-generators-of-ideal}; so
the $R$\+$R$\+bimodule $B$ is a quasi-module over~$R$.
 However, the $R$\+$R$\+bimodule $B$ is \emph{not} a quite
quasi-module over~$R$.
 In fact, in both the cases $B=T/J$ or $B=T/J'$, \emph{there are
no nonzero elements $b\in B$ such that $rb-br=0$ for all $r\in R$}.
 So $F_0B=0$, and consequently $F_\beta B=0$ for all
ordinals~$\beta$.
\end{ex}

 Another counterexample showing that quasi-modules over $R$ need
not be quite quasi-modules (outside of the case of $R$\+$R$\+bimodules
over $K$ for a finitely generated commutative $K$\+algebra~$R$)
can be extracted from Example~\ref{D-infinity-operator-example} below.

\begin{prop} \label{quite-quasi-modules-tensor-product}
 Let $A$ and $B$ be two quite quasi-modules over~$R$.
 Then the $R$\+$R$\+bimodule of tensor product $A\ot_RB$ is also
a quite quasi-module over~$R$.
\end{prop}

\begin{proof}
 Let $\alpha$~be an ordinal such that $A=\bigcup_{\gamma<\alpha}
F_\gamma A$, and let $\beta$~be an ordinal such that
$B=\bigcup_{\delta<\beta}F_\delta B$.
 Consider the Cartesian product of two ordinals $\alpha\times\beta$,
and order it according to the rule $(\gamma',\delta')<
(\gamma'',\delta'')$ if either $\delta'<\delta''$, or $\delta'=\delta''$
and $\gamma'<\gamma''$.
 This ordering makes $\alpha\times\beta$ a well-ordered set, so it is
order isomorphic to an ordinal~$\eta$.
 We will use the notation presuming that $\alpha\times\beta$
is identified with~$\eta$.

 Define an $\eta$\+indexed filtration $G$ on the $R$\+$R$\+bimodule
$A\ot_RB$ by the rule
$$
 G_\zeta(A\ot_RB)=\sum\nolimits_{(\gamma,\delta)\le\zeta}
 \im(F_\gamma A\ot_R F_\delta B\to A\ot_RB),
 \qquad 0\le\zeta<\eta.
$$
 Then, for every $\zeta=(\gamma,\delta)<\eta$, there is a surjective
map of $R$\+$R$\+bimodules
$$
 \Bigl(F_\gamma A\Big/
 \bigcup\nolimits_{\iota<\gamma}F_\iota A\Bigr)\ot_R
 \Bigl(F_\delta A\Big/
 \bigcup\nolimits_{\epsilon<\delta}F_\epsilon B\Bigr) \lrarrow
 G_\zeta(A\ot_RB)\Big/\bigcup\nolimits_{\xi<\zeta}G_\xi(A\ot_RB).
$$
 Therefore, the left and right actions of $R$ in the successive
quotient $R$\+$R$\+bimodule $G_\zeta(A\ot_RB)\big/
\bigcup_{\xi<\zeta}G_\xi(A\ot_RB)$ agree.
 It follows that $G_\zeta(A\ot_RB)\subset F_\zeta(A\ot_RB)$ for
all $\zeta<\eta$.
 Furthermore, we have $A\ot_RB=\bigcup_{\zeta<\eta}G_\zeta(A\ot_RB)$,
since $A=\bigcup_{\gamma<\alpha}F_\gamma A$ and
$B=\bigcup_{\delta<\beta}F_\delta B$.
 Thus $A\ot_RB=\bigcup_{\zeta<\eta}F_\zeta(A\ot_RB)$, as desired.
 
 Notice the difference between the present proof and the proof of
Proposition~\ref{strong-quasi-modules-tensor-product}, in that
the proof of Proposition~\ref{strong-quasi-modules-tensor-product}
used the \emph{addition} of the nonnegative integers~$i$ and~$j$ in
the indices of the filtrations $F$ on $A$ and $B$, while the proof
of the present proposition uses the \emph{multiplication} of ordinals.
 The point is that the addition of integers is \emph{strictly
monotonous}: $i'<i''$ implies $i'+j<i''+j$, and $j'<j''$
implies $i+j'<i+j''$.
 The addition of ordinals is \emph{not} strictly monotonous,
however: $0+\omega=n+\omega=\omega$ for all $n<\omega$.
 The fact that $\gamma'<\gamma''<\alpha$ and $\delta<\beta$
implies $(\gamma',\delta)<(\gamma'',\delta)\in\alpha\times\beta=\eta$,
while $\gamma<\alpha$ and $\delta'<\delta''<\beta$ implies
$(\gamma,\delta')<(\gamma,\delta'')\in\alpha\times\beta=\eta$
is crucial for the argument above.
\end{proof}

\Section{Three Classes of Differential Operators}

\subsection{Differential operators of finite order}
\label{strongly-differential-operators}
 Let $K\rarrow R$ be a homomorphism of commutative rings, and let
$U$ and $V$ be two $R$\+modules.
 Consider the $R$\+$R$\+bimodule $E=\Hom_K(U,V)$ of all $K$\+linear
maps $U\rarrow V$.
 Here the right $R$\+module structure on $E$ is induced by the action
of $R$ on $U$, while the left $R$\+module structure on $E$ comes
from the action of $R$ on~$V$.

 The following definition goes back to
Grothendieck~\cite[Proposition~IV.16.8.8(b)]{EGAIV}.
 A recent exposition can be found in~\cite[Section Tag~09CH]{SP}.

 Let $n\ge0$ be an integer.
 The $R$\+$R$\+subbimodule $F_n\D_{R/K}(U,V)\subset\Hom_K(U,V)$ of
\emph{$K$\+linear $R$\+differential operators of order\/~$\le n$}
is defined as
$$
 F_n\D_{R/K}(U,V)=F_nE\subset E=\Hom_K(U,V),
$$
where $F$ denotes the natural increasing filtration on
the $R$\+$R$\+bimodule $E$ defined in
Section~\ref{strong-quasi-modules-subsecn}.
 So, given a $K$\+linear map $e\:U\rarrow V$, one has
$e\in F_n\D_{R/K}(U,V)$ if and only if
$\theta_{r_n}\theta_{r_{n-1}}\dotsm\theta_{r_1}\theta_{r_0}(e)=0$
in $E$ for all $r_0$, $r_1$,~\dots, $r_n\in R$.

 We put $\D^\st_{R/K}(U,V)=\bigcup_{n\ge0}F_n\D_{R/K}(U,V)$, where
the direct union is taken over the nonnegative integers~$n$.
 We will call the elements of $\D^\st_{R/K}(U,V)$ the \emph{$K$\+linear
strongly $R$\+differential operators $U\rarrow V$}.
 In the terminology of~\cite[Section~IV.16.8]{EGAIV}
and~\cite[Section Tag~09CH]{SP}, these are called simply ``differential
operators''.
 So the strongly differential operators are the differential operators
of finite order.

 A $K$\+linear map $e\:U\rarrow V$ is a strongly $R$\+differential
operator if and only if it is a strongly $I$\+torsion element of
the $T$\+module~$E$ (in the sense of
Section~\ref{strongly-torsion-subsecn}), where $T=R\ot_KR$ and
$I\subset T$ is the kernel ideal of the natural ring homomorphism
$R\ot_KR\rarrow R$.
 See Section~\ref{strong-quasi-modules-subsecn} for a discussion.

\begin{ex} \label{laplace-operator-example}
 Let $K=k$ be a field of characteristic zero, and let
$R=k[(x_i)_{i\in\Lambda}]$ be the ring of polynomials in an infinite
set of variables indexed by some set~$\Lambda$.
 Consider the infinite sum of second partial derivatives
(the \emph{infinitary Laplace operator})
\begin{equation} \label{D_2-Lambda-operator}
 D_2=\sum_{i\in\Lambda}\frac{\d^2}{\d x_i^2}.
\end{equation}
 Any polynomial $f\in R$ only depends on a finite subset of
variables~$x_i$; so one has $\d f/\d x_i=0$ for all but a finite
subset of indices $i\in\Lambda$.
 So all but a finite number of summands in~\eqref{D_2-Lambda-operator}
annihilate~$f$, and $D_2$ is well-defined as a $k$\+linear map
$R\rarrow R$.

 One can easily see that the infinitary Laplace operator
$D_2\:R\rarrow R$ is a $k$\+linear strongly $R$\+differential operator
of order~$2$ according to the definition above, that is
$D_2\in F_2\D_{R/k}(R,R)$ but $D_2\notin F_1\D_{R/k}(R,R)$.
 Let us only point out that, in view of
Lemma~\ref{diagonal-ideal-generators}, it suffices to check that
$\theta_{x_i}\theta_{x_j}\theta_{x_l}(D_2)=0$ for all $i$, $j$,
$l\in\Lambda$ (while $\theta_{x_i}\theta_{x_i}(D_2)\ne0$ for all
$i\in\Lambda$).
\end{ex}

\begin{cor} \label{composition-of-strongly-differential-operators}
 Let $K\rarrow R$ be a homomorphism of commutative rings, and let
$U$, $V$, and $W$ be three $R$\+modules.
 Then the composition of any two $K$\+linear strongly $R$\+differential
operators $D'\:U\rarrow V$ and $D''\:V\rarrow W$ is a $K$\+linear
strongly $R$\+differential operator $D''\circ D'\:U\rarrow W$.
 In fact, if $D'\in F_{n'}\D_{R/K}(U,V)$ is a differential operator
of order~$n'$ and $D''\in F_{n''}\D_{R/K}(V,W)$ is a differential
operator of order~$n''$, then $D''\circ D'\in F_{n'+n''}\D_{R/K}(U,W)$
is a differential operator of order at most $n'+n''$.
\end{cor}

\begin{proof}
 This widely known classical result~\cite[Proposition~IV.16.8.9]{EGAIV},
\cite[Lemma Tag~09CJ]{SP} is not difficult to prove.
 Let us spell out an argument based on (the proof of)
Proposition~\ref{strong-quasi-modules-tensor-product}.
 As above, put $T=R\ot_KR$, and let $I$ be the kernel ideal of
the natural ring homomorphism $R\ot_KR\rarrow R$.
 Put $A=\Gamma_I^\st(\Hom_K(U,V))$ and
$B=\Gamma_I^\st(\Hom_K(V,W))$ (in the notation of
Section~\ref{strongly-torsion-subsecn}).
 So $A$ and $B$ are strong quasi-modules over~$R$
(in the sense of Section~\ref{strong-quasi-modules-subsecn}).
 We have $D'\in A$ and $D''\in B$.
 The composition of $K$\+linear maps~$\circ$ is an $R$\+$R$\+bimodule
map
$$
 \circ\:\Hom_K(V,W)\ot_R\Hom_K(U,V)\lrarrow\Hom_K(U,W).
$$
 By Proposition~\ref{strong-quasi-modules-tensor-product},
the tensor product $B\ot_RA$ is a strong quasi-module over~$R$.
 As the class of strong quasi-modules over $R$ is closed under
quotients in $R\bMod R$, it follows that the image $\circ(B\ot_RA)$
of the composition
$$
 B\ot_RA\lrarrow\Hom_K(V,W)\ot_R\Hom_K(U,V)\lrarrow\Hom_K(U,W)
$$
is a strong quasi-module over $R$ as well.
 Thus $\circ(B\ot_RA)\subset\Gamma_I^\st(\Hom_K(U,W))$, and it
follows that $D''\circ D'\in\circ(B\ot_RA)$ is a strongly
$R$\+differential operator.

 This proves the first assertion of the corollary.
 To prove the second one, it remains to point out that we have
$D'\in F_{n'}A$ and $D''\in F_{n''}B$, hence $D''\ot D'\in
G_{n'+n''}(B\ot_R\nobreak A)\subset F_{n'+n''}(B\ot_R\nobreak A)$
by the proof of Proposition~\ref{strong-quasi-modules-tensor-product}.
 Hence $D''\circ D'\in F_{n'+n''}(\circ(B\ot_R\nobreak A))$ is
an $R$\+differential operator of order~$\le n'+n''$.
 Here we are also using the fact that the natural increasing filtration
$F$ on $R$\+$R$\+bimodules is preserved by homomorphisms of
$R$\+$R$\+bimodules.
\end{proof}

 In particular, it follows from
Corollary~\ref{composition-of-strongly-differential-operators} that,
for any homomorphism of commutative rings $K\rarrow R$ and any
$R$\+module $U$, the $K$\+linear strongly $R$\+differential operators
$U\rarrow U$ form a subring (in fact, a $K$\+subalgebra)
$\D_{R/K}^\st(U,U)\subset\Hom_K(U,U)$ in the $K$\+algebra of
$K$\+linear maps $U\rarrow U$.

\subsection{Differential operators of transfinite order}
\label{quite-differential-operators-subsecn}
 We keep the notation of Section~\ref{strongly-differential-operators}.
 Let $\alpha$~be an ordinal.
 The $R$\+$R$\+subbimodule $F_\alpha\D_{R/K}(U,V)\subset
\Hom_K(U,V)$ of \emph{$K$\+linear $R$\+differential operators of}
(\emph{ordinal}) \emph{order\/~$\le\alpha$} is defined as
$$
 F_\alpha\D_{R/K}(U,V)=F_\alpha E\subset E=\Hom_K(U,V),
$$
where $F$ denotes the natural ordinal-indexed increasing filtration
on the $R$\+$R$\+bi\-mod\-ule $E$ introduced in
Section~\ref{quite-quasi-modules-subsecn}.

 We put $\D^\qu_{R/K}(U,V)=\bigcup_\beta F_\beta\D_{R/K}(U,V)$,
where the direct union is taken over all ordinals~$\beta$.
 Clearly, for every fixed homomorphism of commutative rings $K\rarrow R$
and $R$\+modules $U$ and $V$, there exists an ordinal~$\alpha$ such
that $\D^\qu_{R/K}(U,V)=\bigcup_{\beta<\alpha}F_\beta\D_{R/K}(U,V)$.
 In fact, one can choose one such ordinal~$\alpha$ for all $R$\+modules
$U$ and $V$, so that $\alpha$~only depends on the ring~$R$; see
Corollary~\ref{quite-differential-operators-ordinal-order-bound}
and Theorem~\ref{all-ordinal-orders-possible} below.
 We will call the elements of $\D^\qu_{R/K}(U,V)$ the \emph{$K$\+linear
quite $R$\+differential operators $U\rarrow V$}.
 So the quite differential operators are the differential operators
of ordinal order (and the zero operator).

 In particular, nonzero $R$\+linear maps $U\rarrow V$ are quite
(in fact, strongly) $R$\+differential operators of order~$0$.
 The zero map $0\:U\rarrow V$ does not have an ordinal order; or one
can say that it is a quite differential operator ``of ordinal
order~$-1$'' (which is less than ordinal order~$0$).

 A $K$\+linear map $e\:U\rarrow V$ is a quite $R$\+differential operator
if and only if it is a quite $I$\+torsion element of the $T$\+module $E$
(in the sense of Section~\ref{quite-torsion-subsecn}).
 See Section~\ref{quite-quasi-modules-subsecn} for a discussion.
 By Corollary~\ref{quite-quasi-module-T-nilpotence}, a $K$\+linear
map $e\:U\rarrow V$ is a quite $R$\+differential operator if and only
if, for every sequence of elements $r_0$, $r_1$, $r_2$,~\dots~$\in R$
(indexed by the nonnegative integers) there exists an integer $n\ge0$
such that $\theta_{r_n}\theta_{r_{n-1}}\dotsm\theta_{r_1}
\theta_{r_0}(e)=0$ in~$E$.

 By the definition, any strongly $R$\+differential operator is a quite
$R$\+differential operator.
 Generally speaking, a quite $R$\+differential operator \emph{need not}
be a strongly $R$\+differential operator, as the following examples
demonstrate.

\begin{ex} \label{D-omega-operator-example}
 Let $K=k$ be a field of characteristic zero, and let
$R=k[x_1,x_2,x_3,\dotsc]$ be the ring of polynomials in countably
many variables over~$k$.
 Consider the infinite sum of powers of partial derivatives
\begin{equation} \label{D-omega-operator}
 D_\omega=\sum_{i=1}^\infty\frac{\d^i}{\d x_i^i}.
\end{equation}
 More generally, let $(g_i\in R)_{i=1}^\infty$ be a sequence of
polynomials in $x_1$, $x_2$, $x_3$,~\dots{}
 Then one can consider the infinite sum
\begin{equation} \label{D-prime-omega-operator}
 D'_\omega=\sum_{i=1}^\infty g_i\frac{\d^i}{\d x_i^i}.
\end{equation}
 Similarly to Example~\ref{laplace-operator-example}, for any
polynomial $f\in R$, all but a finite number of summands
in~\eqref{D-prime-omega-operator} annihilate~$f$; so $D'_\omega$ is
well-defined as a $k$\+linear map $R\rarrow R$.

 Assume that there are infinitely many integers $i\ge1$ for which
$g_i\ne0$.
 Then, given such an index~$i$, the commutator $[x_i,D'_\omega]=
\theta_{x_i}(D'_\omega)\in\Hom_k(R,R)$ is the differential operator
$-ig_i\,\d^{i-1}/\d x_i^{i-1}\:R\rarrow R$ of order~$i-1$.
 Therefore, we have $D'_\omega\notin F_n\D_{R/k}(R,R)$ for every
integer $n\ge0$; so $D'_\omega\:R\rarrow R$ is \emph{not} a strongly
$R$\+differential operator.

 On the other hand, for any polynomial $f\in R$, the commutator
$[f,D'_\omega]=\theta_f(D'_\omega)$ is expressed as a finite sum in
our coordinates~$x_i$, i.~e., it belongs to the subring of
$E=\Hom_k(R,R)$ generated by $R$ and the partial
derivatives~$\d/\d x_i$.
 In fact, the finite set of indices~$i$ for which $\d/\d x_i$ appears in
$[f,D'_\omega]$ is a subset of the finite set of indices~$i$ for which
$x_i$ appears in~$f$.
 So $[f,D'_\omega]$ is an $R$\+differential operator of finite order.
 It follows that $D'_\omega\:R\rarrow R$ is a quite $R$\+differential
operator of ordinal order~$\omega$ whenever the set $\{i\ge1
\mid g_i\ne0\}$ is infinite.
 When the latter set is finite, the differential operator $D'_\omega$
has finite order.

 In particular, $D_\omega$ is a quite $R$\+differential operator of
ordinal order~$\omega$, but \emph{not} a strongly $R$\+differential
operator.
\end{ex}

\begin{ex} \label{D-omega+n-operator-example}
 Let $K=k$ be a field of characteristic zero.
 Consider the ring $R=k[x_1,x_2,x_3,\dotsc;y]$ of polynomials
in countably many variables over~$k$.
 Let $n\ge0$ be an integer.
 Consider the infinite sum of compositions of partial derivatives
\begin{equation} \label{D-omega+n-operator}
 D_{\omega+n}=\sum_{i=1}^\infty
 \frac{\d^n}{\d y^n}\,\frac{\d^i}{\d x_i^i}
 =\frac{\d^n}{\d y^n}\,\sum_{i=1}^\infty \frac{\d^i}{\d x_i^i}.
\end{equation}
 More generally, let $(h_i\in R)_{i=1}^\infty$ be a sequence of
polynomials in $x_1$, $x_2$, $x_3$,~\dots\ and~$y$.
 Then one can consider the infinite sum
\begin{equation} \label{D-prime-omega+n-operator}
 D'_{\omega+n}=\sum_{i=1}^\infty
 h_i\frac{\d^n}{\d y^n}\,\frac{\d^i}{\d x_i^i}.
\end{equation}
 Similarly to Example~\ref{laplace-operator-example}, for any
polynomial $f\in R$, one has $\d f/\d x_i=0$ for all but a finite
subset of indices~$i$.
 So all but a finite number of summands
in~\eqref{D-prime-omega+n-operator} annihilate~$f$,
and $D'_{\omega+n}$ is well-defined as a $k$\+linear map $R\rarrow R$.

 Assume that there are infinitely many integers $i\ge1$ for which
$h_i\ne0$.
 One computes that $\theta_y^n(D'_{\omega+n})=
[y,[y,\dotsm[y,D'_{\omega+n}]\dotsm]]=(-1)^n n!D'_{\omega+0}$
($n$~nested brackets).
 Here $D'_{\omega+0}$ is an operator given by the same formula
as~\eqref{D-prime-omega-operator}, but with the polynomial
coefficients~$g_i$ that can depend on the additional variable~$y$.
 Similarly to Example~\ref{D-omega-operator-example},
\,$D'_{\omega+0}$ is a quite $R$\+differential operator of
ordinal order~$\omega$, and it follows that the operator
$D'_{\omega+n}$ cannot have ordinal order smaller than $\omega+n$.

 On the other hand, for every $i\ge0$, the commutator
$[x_i,D'_{\omega+n}]=-ih_i\,\d^n/\d y^n\,\allowbreak
\d^{i-1}/\d x_i^{i-1}$ is a strongly $R$\+differential operator of
finite order~$n+i-1$ (if $h_i\ne0$; or the zero operator otherwise).
 The commutator $[y,D'_{\omega+n}]$ is equal to $-n D'_{\omega+n-1}$
for $n\ge1$.
 Using Lemma~\ref{diagonal-ideal-generators}, one can prove
by induction on~$n$ that
$ID'_{\omega+n}\subset F_{\omega+n-1}\D_{R/k}(R,R)$ for $n\ge1$
(where $I\subset T$ is the kernel ideal of the ring homomorphism
$R\ot_KR\rarrow R$), hence $D'_{\omega+n}\in F_{\omega+n}\D_{R/k}(R,R)$.

 Thus $D'_{\omega+n}\:R\rarrow R$ is a quite $R$\+differential
operator of ordinal order $\omega+n$ whenever the set $\{i\ge1
\mid h_i\ne0\}$ is infinite.
 When the latter set is finite, the differential operator
$D'_{\omega+n}$ has finite order.

 In particular, $D_{\omega+n}$ is a quite $R$\+differential operator
of ordinal order $\omega+n$.
\end{ex}

\begin{ex} \label{D-omega+omega-operator-example}
 Let $K=k$ be a field of characteristic zero.
 Let $R=k[x_1,x_2,x_2,\dotsc;\allowbreak y_1,y_2,y_3,\dotsc]$ be
the ring of polynomials in two countably infinite families of
variables over~$k$.
 Consider the infinite sum of compositions of partial derivatives
\begin{equation} \label{D-omega+omega-operator}
 D_{\omega+\omega}=\sum_{j=1}^\infty\sum_{i=1}^\infty
 \frac{\d^j}{\d y_j^j}\,\frac{\d^i}{\d x_i^i}
 =\sum_{j=1}^\infty \frac{\d^j}{\d y_j^j}
 \,\sum_{i=1}^\infty \frac{\d^i}{\d x_i^i}.
\end{equation}
 Similarly to the previous example, all but a finite number of
summands in the middle term of~\eqref{D-omega+omega-operator}
(or all but a finite number of summands in each of the two factors
in the right-hand side of~\eqref{D-omega+omega-operator})
annihilate every particular polynomial $f\in R$.
 So $D_{\omega+\omega}$ is well-defined as a $k$\+linear map
$R\rarrow R$.

 For every $j\ge1$, one computes that $[y_j,D_{\omega+\omega}]=
-j\,\d^{j-1}/\d y_j^{j-1}\,\sum_{i=1}^\infty\d^i/\d x_i^i=
-jD_{\infty+j-1}$, which is a quite $R$\+differential operator
of ordinal order $\omega+j-1$ similar to~\eqref{D-omega+n-operator}
(with $n$ replaced by $j-1$ and $y$ replaced by~$y_j$).
 Similarly, for every $i\ge1$, one has $[x_i,D_{\omega+\omega}]=
-i\,\d^{i-1}/\d x_i^{i-1}\,\sum_{j=1}^\infty\d^j/\d y_j^j$, which
is a quite $R$\+differential operator of ordinal order $\omega+i-1$.
 As $i$ and~$j$ can be arbitrarily large positive integers, it
follows that the operator $D_{\infty+\infty}$ cannot have ordinal
order smaller than $\omega+\omega$.

 On the other hand, in view of Lemma~\ref{diagonal-ideal-generators},
we have $ID_{\omega+\omega}\subset\bigcup_{n<\omega}F_{\omega+n}
\D_{R/k}(R,R)$ (where $I\subset T$ is the kernel ideal of
$R\ot_KR\rarrow R$), hence $D_{\omega+\omega}\in F_{\omega+\omega}
\D_{R/k}(R,R)$.
 Thus $D_{\omega+\omega}\:R\rarrow R$ is a quite $R$\+differential
operator of ordinal order $\omega+\omega$.
\end{ex}

 For further constructions of quite differential operators of various
ordinal orders, see Theorem~\ref{all-ordinal-orders-possible} below.

\begin{cor} \label{composition-of-quite-differential-operators}
 Let $K\rarrow R$ be a homomorphism of commutative rings, and let
$U$, $V$, and $W$ be three $R$\+modules.
 Then the composition of any two $K$\+linear quite $R$\+differential
operators $D'\:U\rarrow V$ and $D''\:V\rarrow W$ is a $K$\+linear
quite $R$\+differential operator $D''\circ D'\:U\rarrow W$.
 In fact, if $D'\in F_\gamma\D_{R/K}(U,V)$ is a differential operator
of ordinal order~$\gamma$ and $D''\in F_\delta\D_{R/K}(V,W)$ is
a differential operator of ordinal order~$\delta$, then
$D''\circ D'\in F_\zeta\D_{R/K}(U,W)$ is a differential operator
of ordinal order at most~$\zeta$, where
$\zeta+1=\min((\gamma+1)\cdot(\delta+1),\,\allowbreak
(\delta+1)\cdot(\gamma+1))$.
\end{cor}

\begin{proof}
 Clearly, $(\gamma+1)\cdot(\delta+1)=(\gamma+1)\cdot\delta+\gamma+1$
and $(\delta+1)\cdot(\gamma+1)=(\delta+1)\cdot\gamma+\delta+1$
are successor ordinals.
 So we claim that one can have, at one's choice,
$\zeta=(\gamma+1)\cdot\delta+\gamma$ or $\zeta=(\delta+1)\cdot\gamma
+\delta$, whatever is smaller.

 The argument is similar to the proof of
Corollary~\ref{composition-of-strongly-differential-operators}
and based on (the proof of)
Proposition~\ref{quite-quasi-modules-tensor-product}.
 As above, we consider the ring $T=R\ot_KR$, and let $I$ be the kernel
ideal of the natural ring homomorphism $R\ot_KR\rarrow R$.
 Put $A=\Gamma_I^\qu(\Hom_K(U,V))$ and
$B=\Gamma_I^\qu(\Hom_K(V,W))$ (in the notation of
Section~\ref{quite-torsion-subsecn}).
 So $A$ and $B$ are quite quasi-modules over~$R$
(in the sense of Section~\ref{quite-quasi-modules-subsecn}).
 We have $D'\in A$ and $D''\in B$.
 The composition of $K$\+linear maps~$\circ$ is an $R$\+$R$\+bimodule
map
$$
 \circ\:\Hom_K(V,W)\ot_R\Hom_K(U,V)\lrarrow\Hom_K(U,W).
$$
 By Proposition~\ref{quite-quasi-modules-tensor-product},
the tensor product $B\ot_RA$ is a quite quasi-module over~$R$.
 As the class of quite quasi-modules over $R$ is closed under
quotients in $R\bMod R$, it follows that the image $\circ(B\ot_RA)$
of the composition
$$
 B\ot_RA\lrarrow\Hom_K(V,W)\ot_R\Hom_K(U,V)\lrarrow\Hom_K(U,W)
$$
is a quite quasi-module over $R$ as well.
 Thus $\circ(B\ot_RA)\subset\Gamma_I^\qu(\Hom_K(U,W))$, and it
follows that $D''\circ D'\in\circ(B\ot_RA)$ is a quite
$R$\+differential operator.

 This proves the first assertion of the corollary.
 To prove the second one, let $\alpha$ and~$\beta$ be two ordinals
such that $D'\in\bigcup_{\gamma<\alpha}F_\gamma\D_{R/K}(U,V)$ and
$D''\in\bigcup_{\delta<\beta}F_\delta\D_{R/K}(V,W)$.
 Then we have $D'\in\bigcup_{\gamma<\alpha} F_\gamma A$ and
$D''\in\bigcup_{\delta<\beta}F_\delta B$, hence
$$
 D''\ot D'\in\bigcup\nolimits_{\zeta<\beta\cdot\alpha}G_\zeta(B\ot_RA)
 \subset\bigcup\nolimits_{\zeta<\beta\cdot\alpha}F_\zeta(B\ot_RA)
$$
by the proof of Proposition~\ref{quite-quasi-modules-tensor-product}.
 Hence $D''\circ D'\in\bigcup_{\zeta<\beta\cdot\alpha}
F_\zeta(\circ(B\ot_R\nobreak A))$ is an $R$\+differential
operator of ordinal order smaller than $\beta\cdot\alpha$.
 Here we are also using the fact that the natural ordinal-indexed
increasing filtration $F$ on $R$\+$R$\+bimodules is preserved by
homomorphisms of $R$\+$R$\+bimodules.
 Similarly (e.~g., by switching the left and right sides of
the bimodules in Proposition~\ref{quite-quasi-modules-tensor-product})
one proves that the ordinal order of $D''\circ D'$ is smaller
than $\alpha\cdot\beta$.
\end{proof}

 In particular, it follows from
Corollary~\ref{composition-of-quite-differential-operators} that,
for any homomorphism of commutative rings $K\rarrow R$ and any
$R$\+module $U$, the $K$\+linear quite $R$\+differential operators
$U\rarrow U$ form a subring (in fact, a $K$\+subalgebra)
$\D_{R/K}^\qu(U,U)\subset\Hom_K(U,U)$ in the $K$\+algebra of
$K$\+linear maps $U\rarrow U$.

\subsection{Differential operators having no order}
 We still keep the notation of
Section~\ref{strongly-differential-operators}.
 Given an element $r\in R$ and an integer $n\ge0$, consider
the $R$\+$R$\+subbimodule
$$
 F^{(r)}_nE\subset E=\Hom_K(U,V)
$$
defined in Section~\ref{quasi-modules-subsecn}.
 So, for a $K$\+linear map $e\:U\rarrow V$, one has $e\in F^{(r)}_nE$
if and only if $(\theta_r)^{n+1}(e)=0$ in~$E$.

 We will say that a $K$\+linear map $e\:U\rarrow V$ is
a \emph{$K$\+linear $R$\+differential operator $U\rarrow V$}
if, for every element $r\in R$, one has $e\in\bigcup_{n\ge0}F^{(r)}_nE$.
 In other words, the $R$\+$R$\+subbimodule $\D_{R/k}(U,V)$ of
$R$\+differential operators in $\Hom_K(U,V)$ is defined by the formula
$$
 \D_{R/K}(U,V)=\bigcap\nolimits_{r\in R}\bigcup\nolimits_{n\ge0}
 F_n^{(r)}(\Hom_K(U,V)) = \Gamma_I(\Hom_K(U,V))\subset\Hom_K(U,V),
$$
in the notation of Section~\ref{torsion-modules-subsecn}
(see Corollary~\ref{quasi-module-torsion-module}); and
a $K$\+linear map $e\:U\rarrow V$ is an $R$\+differential operator
if and only if it is an $I$\+torsion element of the $T$\+module~$E$.
 Here, as usual, we put $T=R\ot_KR$, and denote by $I\subset T$
the kernel ideal of the natural ring homomorphism $R\ot_KR\rarrow R$.
 By Lemmas~\ref{torsion-checked-on-generators-of-ideal}
and~\ref{diagonal-ideal-generators}, one has
$e\in\D_{R/K}(U,V)$ whenever $e\in\bigcup_{n\ge0}F^{(r_j)}_nE$
for every element~$r_j$ from some given set of generators of
the unital $K$\+algebra~$R$.

 Given an element $r\in R$ and an integer $n\ge0$,
a $K$\+linear $R$\+differential operator $e\:U\rarrow V$ is said to
have \emph{$r$\+order\/~$\le n$} if it belongs to
the $R$\+$R$\+subbimodule
$$
 F^{(r)}\D_{R/K}(U,V)=F^{(r)}_nE\cap\D_{R/K}(U,V)\subset E.
$$

 All quite $R$\+differential operators are $R$\+differential
operators (see the discussion in Sections~\ref{quite-torsion-subsecn}
and~\ref{quite-quasi-modules-subsecn}).
 For a homomorphism of commutative rings $K\rarrow R$ making $R$
a finitely generated $K$\+algebra, a $K$\+linear map between
two $R$\+modules $U\rarrow V$ is an $R$\+differential operator
if and only if it is a quite $R$\+differential operator, and
if and only if it is a strongly $R$\+differential operator (by
Lemmas~\ref{torsion-finitely-generated-ideal}
and~\ref{diagonal-ideal-generators}, cf.\
Sections~\ref{quite-torsion-subsecn}
and~\ref{quite-quasi-modules-subsecn}).
 When the commutative $K$\+algebra $R$ is infinitely generated,
a $K$\+linear $R$\+differential operator \emph{need not} be
a quite $R$\+differential operator, generally speaking, as
the following counterexample demonstrates.

\begin{ex} \label{D-infinity-operator-example}
 Let $K=k$ be a field, and let $R=k[x_1,x_2,x_3,\dotsc]$ be the ring
of polynomials in countably many variables over~$k$.
 Consider the infinite sum of compositions of partial derivatives
\begin{equation} \label{D-infinity-operator}
 D_\infty=\frac{\d}{\d x_1}+\frac{\d^2}{\d x_1\,\d x_2}
 +\frac{\d^3}{\d x_1\,\d x_2\,\d x_3}+\dotsb
\end{equation}
 Similarly to Example~\ref{laplace-operator-example}, all but
a finite number of summands in~\eqref{D-infinity-operator} annihilate
any particular polynomial $f\in R$, so $D_\infty$ is well-defined
as a $k$\+linear map $R\rarrow R$.

 One computes that
\begin{multline*}
 (-1)^n\theta_{x_n}\theta_{x_{n-1}}\dotsm
 \theta_{x_2}\theta_{x_1}(D_\infty) \\
 =\id+\frac{\d}{\d x_{n+1}}+\frac{\d^2}{\d x_{n+1}\,\d x_{n+2}}
 +\frac{\d^3}{\d x_{n+1}\,\d x_{n+2}\,\d x_{n+3}}+\dotsb\,\ne\,0
\end{multline*}
for every integer $n\ge1$.
 By Corollary~\ref{quite-quasi-module-T-nilpotence},
the $R$\+$R$\+subbimodule spanned by $D_\infty$ in $\Hom_k(R,R)$
is \emph{not} a quite quasi-module over~$R$.
 In other words, by Proposition~\ref{quite-torsion-T-nilpotence},
\,$D_\infty$ is \emph{not} a quite $I$\+torsion element in
the $T$\+module $\Hom_k(R,R)$.
 Thus $D_\infty$ is \emph{not} a quite $R$\+differential operator.

 Still, one obviously has $\theta_{x_i}\theta_{x_i}(D_\infty)=0$
for every $i\ge1$.
 So, by Lemmas~\ref{torsion-checked-on-generators-of-ideal}
and~\ref{diagonal-ideal-generators},
\,$D_\infty$ is a $k$\+linear $R$\+differential operator $R\rarrow R$.
 The $R$\+differential operator $D_\infty$
\emph{does not even have an ordinal order}.
 However, it has finite $x_i$\+order equal to~$1$ for every $i\ge1$,
and finite $f$\+order for every polynomial $f\in R$.
\end{ex}

 Thus we have proved the strict inclusions
\begin{equation} \label{diff-operators-strict-inclusions-repeated}
 \D^\st_{R/K}(U,V)\varsubsetneq\D^\qu_{R/K}(U,V)
 \varsubsetneq\D_{R/K}(U,V)
\end{equation}
in the general case, as promised
in~\eqref{diff-operators-inclusions-inequalities}.

\begin{exs} \label{nondifferential-operators-examples}
 For the sake of completeness of the exposition, it remains to present
examples of infinitary expressions with partial derivatives that are
\emph{not differential operators at all} according to our definition.
 For example, let $K=k$ be a field of characteristic zero and
$R=k[x]$ be the ring of polynomials in one variable.
 Then
\begin{equation} \label{shift-operator}
 \Sh=\sum_{i=0}^\infty \frac{1}{i!}\frac{d^i}{dx^i}
\end{equation}
is a well-defined $k$\+linear map $\Sh\:R\rarrow R$; but $\Sh$
is \emph{not} an $R$\+differential operator, as it does not have
finite order with respect to the element $x\in R$.
 Indeed, $(\theta_x)^n(\Sh)=(-1)^n\Sh\ne0$ for every $n\ge0$.
 Actually, $\Sh(f)(x)=f(x+1)$ is the shift operator; see the discussion
in Section~\ref{introd-examples-subsecn} of the Introduction.

 Similarly, consider the ring of polynomials in countably many
variables $R=k[x;y_0,y_1,y_2,\dotsc]$.
 Then
\begin{equation} \label{another-nondifferential-operator}
 \Sh'=\sum_{i=0}^\infty \frac{1}{i!}\frac{\d^i}{\d x^i}\frac{\d}{\d y_i}
\end{equation}
is a well-defined $k$\+linear map $\Sh'\:R\rarrow R$; but $\Sh'$ is
\emph{not} an $R$\+differential operator, since it does not
have finite order with respect to the element $x\in R$.
 One easily computes that $(\theta_x)^n(\Sh')\ne0$ in
$\Hom_k(R,R)$ for every $n\ge1$.
\end{exs}

\begin{cor} \label{composition-of-differential-operators}
 Let $K\rarrow R$ be a homomorphism of commutative rings, and let
$U$, $V$, and $W$ be three $R$\+modules.
 Then the composition of any two $K$\+linear $R$\+differential
operators $D'\:U\rarrow V$ and $D''\:V\rarrow W$ is a $K$\+linear
$R$\+differential operator $D''\circ D'\:U\rarrow W$.
 In fact, let $r\in R$ be an element.
 If $D'\in F^{(r)}_{n'}\D_{R/K}(U,V)$ is a differential operator
of $r$\+order~$n'$ and $D''\in F^{(r)}_{n''}\D_{R/K}(V,W)$ is
a differential operator of $r$\+order~$n''$, then $D''\circ D'\in
F^{(r)}_{n'+n''}\D_{R/K}(U,W)$ is a differential operator of
$r$\+order at most $n'+n''$.
\end{cor}

\begin{proof}
 The proof is similar to that of
Corollaries~\ref{composition-of-strongly-differential-operators}
and~\ref{composition-of-quite-differential-operators}, and
based on (the proof of) Proposition~\ref{quasi-modules-tensor-product}.
 As above, put $T=R\ot_KR$, and let $I$ be the kernel ideal of
the natural ring homomorphism $R\ot_KR\rarrow R$.
 Put $A=\Gamma_I(\Hom_K(U,V))$ and
$B=\Gamma_I(\Hom_K(V,W))$ (in the notation of
Section~\ref{torsion-modules-subsecn}).
 So $A$ and $B$ are quasi-modules over~$R$
(in the sense of Section~\ref{quasi-modules-subsecn}).
 We have $D'\in A$ and $D''\in B$.
 The composition of $K$\+linear maps~$\circ$ is an $R$\+$R$\+bimodule
map
$$
 \circ\:\Hom_K(V,W)\ot_R\Hom_K(U,V)\lrarrow\Hom_K(U,W).
$$
 By Proposition~\ref{quasi-modules-tensor-product},
the tensor product $B\ot_RA$ is a quasi-module over~$R$.
 As the class of quasi-modules over $R$ is closed under quotients in
$R\bMod R$, it follows that the image $\circ(B\ot_RA)$
of the composition
$$
 B\ot_RA\lrarrow\Hom_K(V,W)\ot_R\Hom_K(U,V)\lrarrow\Hom_K(U,W)
$$
is a quasi-module over $R$ as well.
 Thus $\circ(B\ot_RA)\subset\Gamma_I(\Hom_K(U,W))$, and it
follows that $D''\circ D'\in\circ(B\ot_RA)$ is
an $R$\+differential operator.

 This proves the first assertion of the corollary.
 To prove the second one, it remains to say that we have
$D'\in F^{(r)}_{n'}A$ and $D''\in F^{(r)}_{n''}B$, hence $D''\ot D'\in
G^{(r)}_{n'+n''}(B\ot_R\nobreak A)\subset
F^{(r)}_{n'+n''}(B\ot_R\nobreak A)$
by the proof of Proposition~\ref{quasi-modules-tensor-product}.
 Hence $D''\circ D'\in F^{(r)}_{n'+n''}(\circ(B\ot_R\nobreak A))$ is
an $R$\+differential operator of $r$\+order~$\le n'+n''$.
 Here we are also using the fact that the increasing filtration
$F^{(r)}$ on $R$\+$R$\+bimodules is preserved by homomorphisms of
$R$\+$R$\+bimodules.
\end{proof}

 In particular, it follows from
Corollary~\ref{composition-of-differential-operators} that, for any
homomorphism of commutative rings $K\rarrow R$ and any $R$\+module $U$,
the $K$\+linear $R$\+differential operators $U\rarrow U$ form a subring
(in fact, a $K$\+subalgebra) $\D_{R/K}(U,U)\subset\Hom_K(U,U)$ in
the $K$\+algebra of $K$\+linear maps $U\rarrow U$.

\Section{Bounding and Realizing Ordinals}
\label{bounding-realizing-ordinals-secn}

 This section uses a bit more of basic set theory than the rest of this
paper.
 We refer to the initial chapters of the books~\cite{Lev} or~\cite{Kun}
for the relevant discussions of cardinals, ordinals, and regular
cardinals.
 The key result for the purposes of our exposition is that
the successor cardinals are regular~\cite[Proposition~IV.3.11]{Lev},
\cite[Lemma~I.10.37]{Kun}.
 This is a corollary of the fact that every infinite set is
equicardinal to its Cartesian square~\cite[Proposition~III.3.22]{Lev},
\cite[Theorem~I.10.12]{Kun}.

 We start with proving upper bounds and then present a construction
showing that these bounds are sharp.

\begin{prop} \label{quite-torsion-module-filtration-bound}
 Let $\kappa$~be a infinite cardinal, $T$ be a commutative ring, and
$I\subset T$ be an ideal generated by fewer than~$\kappa$ elements.
 Let $M$ be a quite $I$\+torsion $T$\+module (as defined in
Section~\ref{quite-torsion-subsecn}).
 Then $M=\bigcup_{\beta<\kappa}F^{(I)}_\beta M$.
\end{prop}

\begin{proof}
 If the ideal $I\subset T$ is finitely generated, one can take
$\kappa'=\aleph_0$.
 If the ideal $I$ has an infinite set of generators of
the cardinality $\lambda$, one can take $\kappa'=\lambda^+$ to be
the successor cardinal of~$\lambda$.
 In both cases, the cardinal~$\kappa'$ is regular and
$\kappa'\le\kappa$.
 Replacing if necessary $\kappa$ by~$\kappa'$, we can assume without
loss of generality that the cardinal~$\kappa$ in the formulation of
the proposition is regular.

 It suffices to show that $F^{(I)}_\kappa M=\bigcup_{\beta<\kappa}
F^{(I)}_\beta M$ (then one can easily prove by induction that
$F^{(I)}_\alpha M=\bigcup_{\beta<\kappa}F^{(I)}_\beta M$ for all
ordinals $\alpha\ge\kappa$).
 Let $G\subset I$ be a set of generators of the cardinality less
than~$\kappa$, and let $m\in F^{(I)}_\kappa M$ be an element.
 Then, for every $s\in G$, there exists an ordinal $\gamma_s<\kappa$
such that $sm\in F^{(I)}_{\gamma_s}M$.
 Now $\{\,\gamma_s\mid s\in G\,\}$ is a family of ordinals, each of
them smaller than~$\kappa$, and the cardinality of the family is
also smaller than~$\kappa$.
 Since $\kappa$~is a regular cardinal, it follows that there exists
an ordinal $\beta<\kappa$ such that $\gamma_s<\beta$ for all $s\in G$.
 Then we have $m\in F^{(I)}_\beta M$.
\end{proof}

\begin{cor} \label{quite-quasi-module-filtration-bound}
 Let $\kappa$~be an infinite cardinal and $K\rarrow R$ be a homomorphism
of commutative rings such that the $K$\+algebra $R$ is generated by
fewer than~$\kappa$ elements.
 Let $B$ be an $R$\+$R$\+bimodule over~$K$.
 Assume that $B$ is a quite quasi-module over~$R$ (as defined in
Section~\ref{quite-quasi-modules-subsecn}).
 Then $B=\bigcup_{\beta<\kappa} F_\beta B$.
\end{cor}

\begin{proof}
 This is the particular case of
Proposition~\ref{quite-torsion-module-filtration-bound} for
the commutative ring $T=R\ot_KR$ and the kernel ideal $I$ of
the multiplication homomorphism $R\ot_KR\rarrow R$.
 Indeed, the kernel ideal $I$ is generated by fewer than~$\kappa$
elements by Lemma~\ref{diagonal-ideal-generators}.
\end{proof}

\begin{cor} \label{quite-differential-operators-ordinal-order-bound}
 Let $\kappa$~be an infinite cardinal and $K\rarrow R$ be a homomorphism
of commutative rings such that the $K$\+algebra $R$ is generated by
fewer than~$\kappa$ elements.
 Let $U$ and $V$ be two $R$\+modules, and let $D\:U\rarrow V$ be
a $K$\+linear quite $R$\+differential operator (see
Section~\ref{quite-differential-operators-subsecn} for the definition).
 Then the ordinal order of the $R$\+differential operator $D$ is
smaller than~$\kappa$.
\end{cor}

\begin{proof}
 This is the particular case of
Corollary~\ref{quite-quasi-module-filtration-bound} for the quite
quasi-module $B=\D_{R/K}^\qu(U,V)\subset\Hom_K(U,V)$ over the ring~$R$.
\end{proof}

 The following theorem implies that the upper bound in
Corollary~\ref{quite-differential-operators-ordinal-order-bound}
is sharp (generally speaking).
 It follows that the bounds in
Corollary~\ref{quite-quasi-module-filtration-bound} and
Proposition~\ref{quite-torsion-module-filtration-bound} are sharp, too.

\begin{thm} \label{all-ordinal-orders-possible}
 Let $K=k$ be a field, $\kappa$~be a regular cardinal, and $\alpha$~be
an ordinal of cardinality smaller than~$\kappa$.
 Then there is a set $\Lambda$ of cardinality smaller than~$\kappa$
such that, for the $k$\+algebra $R=k[(x_i)_{i\in\Lambda}]$
of polynomials in the set of variables indexed by~$\Lambda$, there
exists a $k$\+linear quite $R$\+differential operator $D_\alpha\:
R\rarrow R$ of ordinal order~$\alpha$.
\end{thm}

 The assertion of the theorem can be rephrased by saying that if
the cardinality of an ordinal~$\alpha$ is equal to~$\lambda$, then
$\alpha$~is the ordinal order of a $k$\+linear quite $R$\+differential
operator $R\rarrow R$, where $R$ is the ring of polynomials in
$\lambda$~variables over any given field~$k$.
 Indeed, the case of a finite cardinal~$\lambda$ is easy, and one can
take the regular cardinal $\kappa=\lambda^+$ in the context of
the theorem if $\lambda$~is infinite.
 In particular, every countable ordinal can be realized as
the ordinal order of a $k$\+linear quite differential operator in
a countable set of variables (take $\kappa=\aleph_1$).

\begin{proof}[Proof of
Theorem~\ref{all-ordinal-orders-possible}]
 We proceed by transfinite induction on~$\alpha$.
 For $\alpha=0$, it suffices to take the empty set of variables
$\Lambda=\varnothing$; so $R=k$.
 Then the identity map $\id\:R\rarrow R$ is a $k$\+linear
$R$\+differential operator of order~$0$.
 For every ordinal $\alpha>0$, we will construct the differential
operator $D_\alpha$ with the additional property that $D_\alpha$
\emph{has no free term with respect to the partial derivatives},
that is, $D_\alpha(1)=0$ in~$R$.

 In the case of a limit ordinal~$\alpha$, we essentially use
the construction from Example~\ref{D-omega-operator-example},
suitably adopted to the situation at hand.
 Pick a family of ordinals
$(0<\beta_\upsilon<\alpha)_{\upsilon\in\Upsilon}$ indexed by
a set $\Upsilon$ of cardinality smaller than~$\kappa$ such
that $\alpha$ is the supremum of~$\beta_\upsilon$ over
$\upsilon\in\Upsilon$.
 For every $\upsilon\in\Upsilon$, pick a set $\Lambda_\upsilon$ of
cardinality smaller than~$\kappa$ and a $k$\+linear
quite $R_\upsilon$\+differential operator $D_{\beta_\upsilon}\:
R_\upsilon\rarrow R_\upsilon$ of ordinal order~$\beta_\upsilon$
acting on the $k$\+algebra of polynomials
$R_\upsilon=k[(x_i)_{i\in\Lambda_\upsilon}]$.
 We presume that $D_{\beta_\upsilon}(1)=0$ in $R_\upsilon$ for all
$\upsilon\in\Upsilon$.

 Let $\Lambda=\coprod_{\upsilon\in\Upsilon}\Lambda_\upsilon$ be
the disjoint union of the sets~$\Lambda_\upsilon$ over
all $\upsilon\in\Upsilon$.
 Then, for every $\upsilon\in\Upsilon$, the $k$\+algebra
$R=k[(x_i)_{i\in\Lambda}]$ decomposes naturally into the tensor product
$R=R_\upsilon\ot_k S_\upsilon$ of two $k$\+algebras of polynomials
$R_\upsilon$ and $S_\upsilon=
k[(x_j)_{j\in\Lambda\setminus\Lambda_\upsilon}]$.
 Denote by $D'_{\beta_\upsilon}\:R\rarrow R$ the tensor product
of $k$\+linear operators $D'_{\beta_\upsilon}=D_{\beta_\upsilon}
\ot\id_{S_\upsilon}$, where $\id_{S_\upsilon}\:S_\upsilon\rarrow
S_\upsilon$ is the identity map.
 Simply put, $D'_{\beta_\upsilon}$ is the notation for (a version of)
the differential operator $D_{\beta_\upsilon}$ acting on the ring $R$
rather than~$R_\upsilon$.
 The assumption that $D_{\beta_\upsilon}(1)=0$ implies that
the restriction of $D'_{\beta_\upsilon}$ to the subalgebra
$S_\upsilon\subset R$ vanishes, $D'_{\beta_\upsilon}({S_\upsilon})=0$.

 Consider the infinite sum
\begin{equation} \label{limit-ordinal-D-alpha}
 D_\alpha=\sum\nolimits_{\upsilon\in\Upsilon} D'_{\beta_\upsilon}.
\end{equation}
 Every polynomial $f\in R$ only depends on a finite number of
variables $x_i$, \,$i\in\Lambda$, which come from a finite number
of subsets $\Lambda_\upsilon\subset\Lambda$, \,$\upsilon\in\Upsilon$.
 So we have $f\in S_\upsilon\subset R$ for all but a finite subset of
indices $\upsilon\in\Upsilon$.
 Therefore, all but a finite subset of summands
in~\eqref{limit-ordinal-D-alpha} annihilate~$f$, and $D_\alpha$
is well-defined as a $k$\+linear operator $R\rarrow R$.

 For every index $i\in\Lambda_\upsilon\subset\Lambda$, denote
by $C_{\beta_\upsilon,i}\:R_\upsilon\rarrow R_\upsilon$
the $k$\+linear map $C_{\beta_\upsilon,i}=[x_i,D_{\beta_\upsilon}]=
\theta_{x_i}(D_{\beta_\upsilon})$.
 The assumption that $D_{\beta_\upsilon}\:R_\upsilon\rarrow R_\upsilon$
is a $k$\+linear quite $R_\upsilon$\+differential operator of
ordinal order~$\beta_\upsilon$ means that $C_{\beta_\upsilon,i}$
is (either the zero map or) a $k$\+linear quite
$R_\upsilon$\+differential operator of ordinal order
$\gamma_{\upsilon,i}<\beta_\upsilon$, and~$\beta_{\upsilon}$ is
the minimal ordinal such that $\gamma_{\upsilon,i}<\beta_\upsilon$
for all $i\in\Lambda_\upsilon$.
 (We are using Lemma~\ref{diagonal-ideal-generators} here.)

 Now we have $[x_i,D_\alpha]=\theta_{x_i}(D_\alpha)=
C'_{\beta_\upsilon,i}\:R\rarrow R$, where $C'_{\beta_\upsilon,i}=
C_{\beta_\upsilon,i}\ot_k\id_{S_\upsilon}$ is (a version of)
the differential operator $C_{\beta_\upsilon,i}$ acting on
the ring $R$ rather than~$R_\upsilon$.
 The map $C'_{\beta_\upsilon,i}\:R\rarrow R$ is a $k$\+linear quite
$R$\+differential operator of ordinal order~$\gamma_{\upsilon,i}$.
 Thus $D_\alpha\:R\rarrow R$ is a $k$\+linear quite $R$\+differential
operator whose ordinal order is the minimal ordinal~$\alpha'$
such that $\gamma_{\upsilon,i}<\alpha'$ for all $\upsilon\in\Upsilon$
and $i\in\Lambda_\upsilon$.
 As $\alpha$~is the supremum of~$\beta_\upsilon$ and $\alpha$~is
a limit ordinal, it follows that $\alpha'=\alpha$.

 In the case of a successor ordinal $\alpha=\beta+1$,
we essentially use the construction from 
Example~\ref{D-omega+n-operator-example} (for $n=1$), suitably
generalized to fit the situation at hand.
 Pick a set $\Delta$ of cardinality smaller than~$\kappa$ and
a $k$\+linear quite $Q$\+differential operator $D_\beta\:Q\rarrow Q$
of ordinal order~$\beta$ acting on the polynomial ring
$Q=k[(x_i)_{i\in\Delta}]$.
 Let $\Lambda=\Delta\sqcup\{*\}$ be the set obtained by adjoining
one element~$*$ to~$\Delta$.
 Denote the related variable by $x_*=y$; so $R=k[(x_i)_{i\in\Lambda}]
=k[(x_i)_{i\in\Delta};y]=Q\ot_kk[y]$.

 Let $D_\alpha\:R\rarrow R$ be the $k$\+linear map
$D_\alpha=D_\beta\ot_k\d/\d y$, or equivalently,
\begin{equation} \label{successor-ordinal-alpha}
 D_\alpha=D'_\beta\frac{\d}{\d y},
\end{equation}
where $D'_\beta=D_\beta\ot_k\id_{k[y]}$ is (a version of)
the differential operator $D_\beta$ acting on the ring $R$ rather
than~$Q$.

 For every index $i\in\Delta$, denote by $C_{\beta,i}\:Q\rarrow Q$
the $k$\+linear map $C_{\beta,i}=[x_i,D_\beta]=\theta_{x_i}(D_\beta)$.
 The assumption that $D_\beta\:Q\rarrow Q$ is a $k$\+linear quite
$Q$\+differential operator of ordinal order~$\beta$ means that
$C_{\beta,i}$ is (either the zero map or) a $k$\+linear quite
$Q$\+differential operator of ordinal order $\gamma_i<\beta$,
and~$\beta$ is the minimal ordinal such that $\gamma_i<\beta$ for all
$i\in\Delta$.

 Proceeding by transfinite induction on~$\beta$, we will prove that
$D_\alpha$ is a quite $R$\+differential operator of ordinal
order $\alpha=\beta+1$.
 Indeed, we have $[y,D_\alpha]=\theta_y(D_\alpha)=-D'_\beta$.
 As $D'_\beta\:R\rarrow R$ is a $k$\+linear quite $R$\+differential
operator of ordinal order~$\beta$, the ordinal order of~$D_\alpha$
cannot be smaller than~$\alpha$.

 On the other hand, for every $i\in\Delta$, we have
$[x_i,D_\alpha]=\theta_{x_i}(D_\alpha)=C_{\beta,i}\ot_k\d/\d y\:
R\rarrow R$.
 Since $C_{\beta,i}\:Q\rarrow Q$ is a $k$\+linear quite
$Q$\+differential operator of ordinal order $\gamma_i<\beta$,
the assumption of transfinite induction on~$\beta$ implies that
the ordinal order of $C_{\beta,i}\ot_k\d/\d y$ is equal
to $\gamma_i+1$.
 It remains to observe that $\gamma_i+1<\beta+1=\alpha$.
 Thus $D_\alpha\:R\rarrow R$ is a $k$\+linear quite
$R$\+differential operator of ordinal order~$\alpha$, as desired.
 (Once again, we are using Lemma~\ref{diagonal-ideal-generators} here.)
\end{proof}

\Section{Flat Epimorphisms of Commutative Rings and Quasi-Modules}

 The aim of this section is to prove
Lemma~\ref{quite-quasi-module-extension-of-scalars}
and Proposition~\ref{all-flat-epis-are-FQM}, which will play a crucial
role in the final Sections~\ref{localizing-diffops-secn}\+-%
\ref{colocalizing-diffops-secn}.
 We also deduce
Corollary~\ref{fl-ring-epi-quasi-module-restriction-of-scalars}
and Lemma~\ref{ring-epi-restr-of-scalars-quite-quasi-modules}.

\subsection{Flat epimorphisms and quite quasi-modules}
\label{flat-epis-and-quite-quasi-modules}
 A homomorphism of commutative ring $\sigma\:R\rarrow S$ is said to be
a \emph{commutative ring epimorphism} if it is an epimorphism in
the category of commutative rings, or equivalently, if $\sigma$~is
an epimorphism in the category of associative rings.
 Equivalently, $\sigma$~is an epimorphism if and only if the related
functor of restriction of scalars $\sigma_*\:S\Modl\rarrow R\Modl$
is fully faithful.
 If this is the case, then the full subcategory $\sigma_*(S\Modl)$ is
closed under kernels, cokernels, infinite direct sums, and infinite
products in $R\Modl$.

 For any homomorphism of commutative rings $\sigma\:R\rarrow S$,
there are two induced homomorphisms of commutative rings $\sigma\ot\id$
and $\id\ot\sigma\:S\rightrightarrows S\ot_RS$, and the multiplication
homomorphism of commutative rings $S\ot_RS\rarrow S$.
 A commutative ring homomorphism $\sigma\:R\rarrow S$ is an epimorphism
if and only if \emph{any one} of these three homomorphisms of
commutative rings is an isomorphism, or equivalently, if and only if
all the three of them are isomorphisms, or equivalently, if and only
if the two ring homomorphisms $S\rightrightarrows S\ot_RS$ are equal to
each other.
 The details can be found in the book~\cite[Section~XI.1]{Ste}.

 An epimorphism of commutative rings $\sigma\:R\rarrow S$ is said to be 
\emph{flat} if $S$ is flat as an $R$\+module with the $R$\+module
structure induced by~$\sigma$.
 A discussion of flat epimorphisms of (not necessarily commutative)
rings can be found in~\cite[Sections~XI.2\+-3]{Ste}.
 For a flat epimorphism~$\sigma$, the full subcategory
$\sigma_*(S\Modl)$ is closed under extensions in $R\Modl$, as one
can see from~\cite[Theorem~4.4]{GL}.

 A particular case of the following lemma was stated
in~\cite[Lemma~2.2]{Pedg}.
 Notice that ``quasi-modules'' in the terminology of~\cite{Pedg} are
what we call strong quasi-modules in
Section~\ref{strong-quasi-modules-subsecn} of this paper; these are
particular cases of quite quasi-modules in the sense of
Section~\ref{quite-quasi-modules-subsecn}.

\begin{lem} \label{quite-quasi-module-extension-of-scalars}
 Let $\sigma\:R\rarrow S$ be a flat epimorphism of commutative rings,
and let $B$ be a quite quasi-module over~$R$ (in the sense of
the definition in Section~\ref{quite-quasi-modules-subsecn}).
 In this context: \par
\textup{(a)} The map~$\sigma$ induces isomorphisms of
$R$\+$R$\+bimodules
$$
 S\ot_RB\lrarrow S\ot_RB\ot_RS\llarrow B\ot_RS
$$
(hence $S\ot_RB\simeq B\ot_RS$ is an $S$\+$S$\+bimodule). \par
\textup{(b)} The $S$\+$S$\+bimodule $S\ot_RB$ is a quite quasi-module
over~$S$.
 Moreover, if $B=\bigcup_{\beta<\alpha}F_\beta B$ for some
ordinal~$\alpha$, then $S\ot_RB=\bigcup_{\beta<\alpha}
F_\beta(S\ot_RB)$.
 Here $F$ denotes the natural ordinal-indexed increasing filtrations on
the $R$\+$R$\+bimodule $B$ and the $S$\+$S$\+bimodule $S\ot_RB$.
\end{lem}

\begin{proof}
 Part~(a): it suffices to show that the $R$\+$S$\+bimodule $B\ot_RS$,
viewed as an $R$\+module with the left action of $R$, belongs to
the full subcategory $\sigma_*(S\Modl)\subset R\Modl$.
 Indeed, put $G_\alpha(B\ot_RS)=F_\alpha B\ot_RS$ for every
ordinal~$\alpha$ (where $F$ is the natural ordinal-indexed increasing
filtration on the $R$\+$R$\+bimodule~$B$).
 Then we have $B\ot_RS=\bigcup_\beta G_\alpha(B\ot_RS)$.
 Furthermore, for every ordinal~$\alpha$, the left and right actions
of $R$ on the successive quotient bimodule
$$
 G_\alpha(B\ot_RS)\Big/\bigcup\nolimits_{\beta<\alpha}
 G_\beta(B\ot_RS)\simeq
 \biggl(F_\alpha B\Big/\bigcup\nolimits_{\beta<\alpha} F_\beta B\biggr)
 \ot_RS
$$
agree, because the left and right actions of $R$ on
$F_\alpha B\big/\bigcup_{\beta<\alpha} F_\beta B$ agree by
the definition of the filtration~$F$.
 Therefore, the left $R$\+module structures on the successive
quotients $G_\alpha(B\ot_RS)\big/\bigcup_{\beta<\alpha}G_\beta(B\ot_RS)$
come via the restriction of scalars from some $S$\+module structures.
 Since the full subcategory $\sigma_*(S\Modl)\subset R\Modl$ is
closed under extensions and direct limits, the desired assertion
follows.

 Part~(b): it follows from the proof of part~(a) that
$G_\beta(S\ot_RB)\subset F_\beta(S\ot_RB)$ for all ordinals~$\beta$,
where $G_\beta(S\ot_RB)=S\ot_RF_\beta B$.
 Therefore, if $B=\bigcup_{\beta<\alpha}F_\beta B$, then
$S\ot_RB=\bigcup_{\beta<\alpha}G_\beta(S\ot_RB)=\bigcup_{\beta<\alpha}
F_\beta(S\ot_RB)$.
\end{proof}

\subsection{Flat epimorphisms and quasi-modules}
\label{flat-epis-and-quasi-modules}
 In this section we prove the following proposition.
 Notice that part~(a) is a generalization of
Lemma~\ref{quite-quasi-module-extension-of-scalars}(a), while
part~(b) has less restrictive assumptions but weaker conclusion.

\begin{prop} \label{all-flat-epis-are-FQM}
 Let $\sigma\:R\rarrow S$ be a flat epimorphism of commutative rings,
and let $B$ be a quasi-module over~$R$ (in the sense of the definition
in Section~\ref{quasi-modules-subsecn}).
 Then \par
\textup{(a)} the $R$\+$R$\+bimodule maps
$$
 S\ot_RB\lrarrow S\ot_RB\ot_RS\llarrow B\ot_RS
$$
induced by~$\sigma$ are isomorphisms (hence $S\ot_RB\simeq B\ot_RS$
is an $S$\+$S$\+bimodule); \par
\textup{(b)} the $S$\+$S$\+bimodule $S\ot_RB$ is a quasi-module
over~$S$.
\end{prop}

 Let us say temporarily that a flat epimorphism of commutative rings
$R\rarrow S$ is an \emph{FQM\+epimorphism} if the assertions~(a)
and~(b) of Proposition~\ref{all-flat-epis-are-FQM} hold for all
quasi-modules $B$ over~$R$.
 We will prove that all flat epimorphisms of commutative rings are
FQM\+epimorphisms.

 A morphism of commutative rings $R\rarrow S$ is said to be \emph{of
finite presentation} if it makes $S$ a finitely presented commutative
$R$\+algebra.
 Flat epimorphisms of finite presentation $\sigma\:R\rarrow S$ are
precisely all the morphisms of commutative rings for which the related
morphism of affine schemes $\Spec S\rarrow\Spec R$ is an open
immersion~\cite[Th\'eor\`eme~IV.17.9.1]{EGAIV}.
 The following lemma is proved in~\cite[Section~1.4]{Pdomc}.

\begin{lem} \label{affine-open-immersions-are-FQM}
 All flat epimorphisms of finite presentation (between commutative
rings) are FQM\+epimorphisms.
 In particular, for any commutative ring $R$ and an element
$r\in R$, the localization morphism $R\rarrow R[r^{-1}]$ is
an FQM\+epimorphism.
\end{lem}

\begin{proof}
 This is~\cite[Lemmas~1.14 and~1.15]{Pdomc}.
\end{proof}

\begin{lem} \label{inductive-limits-quasi-modules}
 Let $(S_\xi)_{\xi\in\Xi}$ be a filtered diagram of commutative rings,
indexed by a directed poset\/~$\Xi$, and let $(B_\xi)_{\xi\in\Xi}$
be a diagram of $S_\xi$\+bimodules indexed by\/~$\Xi$.
 Assume that the bimodule $B_\xi$ is a quasi-module over $S_\xi$ for
every index $\xi\in\Xi$.
 Consider the direct limits $S=\varinjlim_{\xi\in\Xi}S_\xi$
and $B=\varinjlim_{\xi\in\Xi}B_\xi$.
 Then the $S$\+$S$\+bimodule $B$ is a quasi-module over~$S$.
\end{lem}

\begin{proof}
 Given an element $s\in S$, we need to prove that the increasing
filtration $F^{(s)}$ on the $S$\+$S$\+bimodule $B$ is exhaustive.
 Let $\xi_0\in\Xi$ be an index such that the element $s\in S$
comes from an element $s_{\xi_0}\in S_{\xi_0}$.
 For every index $\xi\ge\xi_0$, denote by $s_\xi\in S_\xi$
the image of the element $s_{\xi_0}$ under the transition map
$S_{\xi_0}\rarrow S_\xi$.
 Then we have $F^{(s)}_nB=\varinjlim_{\xi\ge\xi_0}F^{(s_\xi)}_nB_\xi$
for every $n\ge0$ (cf.\ the proof of
Lemma~\ref{quasi-module-characterized-by-localizations} below).
 Since the filtrations $F^{(s_\xi)}$ on the $S_\xi$\+$S_\xi$\+bimodules
$B_\xi$ are exhaustive by assumption, so is the filtration $F^{(s)}$
on the $S$\+$S$\+bimodule~$B$.
\end{proof}

\begin{lem} \label{inductive-limits-FQM-epimorphisms}
 Let $(S_\xi)_{\xi\in\Xi}$ be a filtered diagram of commutative rings,
indexed by a directed poset\/~$\Xi$.
 Let $R$ be a commutative ring endowed with a morphism of diagrams
$(\sigma_\xi)_{\xi\in\Xi}\:R\rarrow(S_\xi)_{\xi\in\Xi}$, where $R$ is
viewed as a constant diagram.
 Assume that, for every $\xi\in\Xi$, the commutative ring homomorphism
$\sigma_\xi\:R\rarrow S_\xi$ is an FQM\+epimorphism.
 Then the commutative ring homomorphism
$\sigma=\varinjlim_{\xi\in\Xi}\sigma_\xi\:R\rarrow
\varinjlim_{\xi\in\Xi}S_\xi=S$ is an FQM\+epimorphism, too.
\end{lem}

\begin{proof}
 The characterization of ring epimorphisms $\sigma\:R\rarrow S$ by
the property that the induced maps $S\rightrightarrows S\ot_RS
\rarrow S$ are isomorphisms clearly implies that the filtered
direct limits of ring epimorphisms are ring epimorphisms.
 The filtered direct limits also preserve flatness.

 Now let $B$ be any given quasi-module over~$R$.
 Then condition~(a) for the quasi-module $B$ and the ring homomorphism
$\sigma\:R\rarrow S=\varinjlim_{\xi\in\Xi}S_\xi$ follows from
condition~(a) for the same quasi-module $B$ and the ring homomorphisms
$\sigma_\xi\:R\rarrow S_\xi$ by passing to the direct limit.
 Finally, condition~(b) for $B$ and~$\sigma$ follows from condition~(b)
for $B$ and~$\sigma_\xi$ by virtue of
Lemma~\ref{inductive-limits-quasi-modules}.
\end{proof}

\begin{cor} \label{localizations-by-mult-subsets-FQM}
 Let $R$ be a commutative ring and\/ $\Sigma\subset R$ be
a multiplicative subset.
 Then the localization morphism $R\rarrow \Sigma^{-1}R$ is
an FQM\+epimorphism.
\end{cor}

\begin{proof}
 For any finite subset $\xi=(s_1,\dotsc,s_m)\subset\Sigma$,
denote by $\Sigma_\xi\subset\Sigma$ the multiplicative subset in $R$
generated by $s_1$,~\dots,~$s_m$.
 Then we have $\Sigma_\xi^{-1}R=R[r^{-1}]$ for the element
$r=s_1\dotsm s_m\in R$.
 So the localization map $R\rarrow\Sigma_\xi^{-1}R$ is
an FQM\+epimorphism by Lemma~\ref{affine-open-immersions-are-FQM}.
 Now let $\Xi$ be the directed poset of all finite subsets
$\xi\subset\Sigma$, ordered by inclusion.
 Then we have $\Sigma^{-1}R=\varinjlim_{\xi\in\Xi}\Sigma_\xi^{-1}R$, and
it remains to refer to Lemma~\ref{inductive-limits-FQM-epimorphisms}.
\end{proof}

 Given a prime ideal~$\p$ in a commutative ring $R$, we denote by
$R_\p=(R\setminus\p)^{-1}R$ the localization of $R$ at~$\p$
(as in Section~\ref{torsion-modules-subsecn}).

\begin{lem} \label{quasi-module-characterized-by-localizations}
 Let $R\rarrow S$ be a homomorphism of commutative rings and $B$ be
an $S$\+$S$\+bimodule.
 Then $B$ is a quasi-module over $S$ if and only if, for every prime
ideal\/ $\p\subset R$, the $S$\+$S$\+bimodule $R_\p\ot_RB$ is
a quasi-module over~$S$.
 It suffices to consider maximal ideals\/ $\p\subset R$.
\end{lem}

\begin{proof}
 The point is that, for each fixed element $s\in S$, the construction
of the natural increasing filtration $F^{(s)}$ on an $S$\+$S$\+bimodule
$B$ can be expressed in terms of iterated passages to kernels and
cokernels of natural morphisms.
 In particular, $F^{(s)}_0B\subset B$ is the kernel of the map
$\theta_s\:B\rarrow B$, etc.
 Since the functor $R_\p\ot_R{-}$ preserves kernels and cokernels,
one has $F^{(s)}_n(R_\p\ot_RB)=F_\p\ot_R F^{(s)}_nB$ for all $n\ge0$.
 It remains to recall that $R_\p\ot_RM=0$ for a given $R$\+module $M$
and all maximal ideals $\p\subset R$ implies $M=0$.
 Thus one has $B=\bigcup_{n\ge0}F^{(s)}_nB$ if and only if
$F_\p\ot_RB=\bigcup_{n\ge0}F^{(s)}_n(F_\p\ot_RB)$ for all prime
(or maximal) ideals $\p\subset R$.
\end{proof}

\begin{lem} \label{FQM-epimorphisms-determined-on-local-rings}
 Let $\sigma\:R\rarrow S$ be a flat epimorphism of commutative rings.
 Assume that, for every prime ideal\/ $\q\subset S$, the composition
$R\overset\sigma\rarrow S\rarrow S_\q$ is an FQM\+epimorphism.
 Then $\sigma$~is an FQM\+epimorphism.
 It suffices to consider maximal ideals\/ $\q\subset S$.
\end{lem}

\begin{proof}
 Notice first of all that the compositions of (flat) epimorphisms of
commutative rings are (flat) epimorphisms of commutative rings.
 So the composition $R\rarrow S\rarrow S_\q$ is always a flat
epimorphism.

 Let $B$ be a quasi-module over~$R$.
 Let us prove that the map $f\:S\ot_RB\rarrow S\ot_RB\ot_RS$ is
an isomorphism.
 By assumption, we know that the map $S_\q\ot_RB\rarrow S_\q\ot_RB\ot_R
S_\q$ is an isomorphism for all maximal ideals $\q\subset S$.
 The latter map factorizes into the composition
$$
 S_\q\ot_RB\lrarrow S_\q\ot_RB\ot_R S\lrarrow S_\q\ot_RB\ot_RS_\q.
$$
 Since the composition is an isomorphism, it follows that
$S_\q\ot_RB\rarrow S_\q\ot_RB\ot_RS$ is the embedding of a direct
summand in the category of $S_\q$\+$R$\+bimodules.

 Let us speak of ``right $R$\+modules'' and ``right $S$\+modules''
in order to emphasize that we are interested in the right module
structures on our bimodules over the commutative rings $R$ and~$S$.
 The functor of restriction of scalars $\sigma_*\:\Modr S\rarrow\Modr R$
acting between the categories of right $S$\+modules and right
$R$\+modules is fully faithful, and its essential image
$\sigma_*(\Modr S)$ is closed under kernels and cokernels in $\Modr R$.
 In particular, the full subcategory $\sigma_*(\Modr S)\subset
\Modr R$ is closed under direct summands.
 The right $R$\+module $S_\q\ot_RB$ is a direct summand of the right
$S$\+module $S_\q\ot_RB\ot_RS$, so it follows that the right
$R$\+module structure on $S_\q\ot_RB$ arises from a right $S$\+module
structure.

 For any right $S$\+module $N$, the natural maps $N\rarrow N\ot_RS
\rarrow N$ are isomorphisms (since the natural maps $S\rightrightarrows
S\ot_RS\rarrow S$ are isomorphisms).
 In particular, this applies to $N=S_\q\ot_RB$.
 We have proved that the map
$$
 S_\q\ot_Sf\: S_\q\ot_RB\lrarrow S_\q\ot_RB\ot_RS
$$
is an isomorphism.
 As this holds for all maximal ideals $\q\subset S$, we can conclude
that the map $f\:S\ot_RB\rarrow S\ot_RB\ot_RS$ is an isomorphism.
 Similarly one proves that the map $B\ot_RS\rarrow S\ot_RB\ot_RS$ is
an isomorphism.

 Let us show that $S\ot_RB$ is a quasi-module over~$S$.
 By assumption, $S_\q\ot_RB$ is a quasi-module over~$S_\q$.
 By~\cite[Lemma~1.13]{Pdomc} (see also
Corollary~\ref{fl-ring-epi-quasi-module-restriction-of-scalars} below),
it follows that $S_\q\ot_RB$ is a quasi-module over~$S$.
 It remains to point out once again the isomorphism $S_\q\ot_S(S\ot_RB)
\simeq S_\q\ot_RB$ and refer to
Lemma~\ref{quasi-module-characterized-by-localizations} (for
the identity ring homomorphism $S\rarrow S$).
\end{proof}

 In order to finish the proof of
Proposition~\ref{all-flat-epis-are-FQM}, we now need to recall
some more advanced parts of well-known material about flat epimorphisms
of commutative rings.
 We are interested in flat epimorphisms of commutative rings
$R\rarrow S$ viewed up to isomorphism, where the ring $R$ is fixed
and the ring $S$ varies.
 So we say that $\sigma'\:R\rarrow S'$ is isomorphic to
$\sigma''\:R\rarrow S''$ if there is a ring isomorphism $S'\simeq S''$
forming a commutative triangular diagram with $\sigma'$ and~$\sigma''$.

 A prime ideal $\p''\in\Spec R$ is said to be a \emph{generalization}
of a prime ideal $\p'\in\Spec R$ if $\p''\subset\p'\subset R$.
 The following lemma is a well-known result.

\begin{lem} \label{flat-epimorphisms-determined-on-spectra}
 A flat epimorphism of commutative rings $\sigma\:R\rarrow S$ is
uniquely determined, up to isomorphism, by the image of the induced map
of the spectra\/ $\Spec\sigma\:\Spec S\rarrow\Spec R$.
 Here the image of\/ $\Spec\sigma$ is viewed just as a subset in\/
$\Spec R$.
 In fact, $\Spec\sigma$ is an injective map, and its image is
a generalization-closed subset in\/ $\Spec R$.
\end{lem}

\begin{proof}
 According to~\cite[Theorem~XI.2.1]{Ste}, a flat epimorphism of
commutative rings $\sigma\:R\rarrow S$ is uniquely determined, up to
isomorphism, by the Gabriel filter $F_\sigma$ of all ideals
$I\subset R$ such that $S\sigma(I)=S$ (i.~e., the extension of $I$ in
$S$ is the unit ideal), or equivalently, $S\ot_RR/I=0$.
 It is clear from the former definition of $F_\sigma$ that it has a base
of finitely generated ideals, i.~e., any ideal $I\in F_\sigma$
contains a finitely generated ideal belonging to~$F_\sigma$.

 According to the discussion in~\cite[Section~VI.6.6]{Ste}, a Gabriel
filter $F$ in a commutative ring $R$ such that $F$ has a base of
finitely generated ideals is characterized by the set of all prime
ideals $\p\subset R$ belonging to~$F$.
 Simply put, an ideal $I\subset R$ belongs to $F$ if and only
if all the prime ideals~$\p$ containing $I$ belong to~$F$.
 In particular, this applies to the Gabriel filters~$F_\sigma$.

 Now we return to the map $\Spec\sigma\:\Spec S\rarrow\Spec R$.
 For any commutative ring epimorphism~$\sigma$, the map $\Spec\sigma$
is injective~\cite[Lemma Tag~04VW]{SP}.
 For any commutative ring homomorphism $\sigma\:R\rarrow S$ making $S$
a flat $R$\+module, the image of the map $\Spec\sigma$ is closed under
generalizations.
 To repeat, this means that $\p'=\sigma^{-1}(\q')$ for a prime ideal
$\q'\subset S$ and prime ideals $\p''\subset\p'\subset R$ implies
existence of a prime ideal $\q''\subset S$ such that
$\sigma^{-1}(\q'')=\p''$ \,\cite[Lemma Tag~00HS]{SP}.

 Clearly, if a prime ideal $\p\subset R$ belongs to the image of
$\Spec\sigma$, then the extension of~$\p$ in $S$ is \emph{not}
the unit ideal; so $\p\notin F_\sigma$.
 Conversely, if $\p\notin F_\sigma$, then the extension of~$\p$
in $S$ is contained in some prime ideal $\q'\subset R$, i.~e.,
$S\sigma(\p)\subset\q'$.
 Put $\p'=\sigma^{-1}(\q')$; then we have $\p\subset\p'\subset R$.
 According to the previous paragraph, it follows that $\p$~belongs
to the image of $\Spec\sigma$.
 We have shown that the set of prime ideals in $R$ belonging to
$F_\sigma$ is precisely the complement to the image of $\Spec\sigma$
in $\Spec R$.

 All the assertions of the lemma are proved now, based on
the cited results from~\cite{Ste}.
\end{proof}

 The following key corollary describes flat epimorphisms of commutative
rings whose codomain is a local ring.

\begin{cor} \label{flat-epimorphisms-with-local-codomains}
 Let $S$ be a commutative local ring with the maximal ideal\/~$\n$,
and let $\sigma\:R\rarrow S$ be a flat epimorphism of commutative rings.
 Let\/ $\p=\sigma^{-1}(\n)\subset R$ be the image of\/~$\n$ under
the induced map of spectra\/ $\Spec S\rarrow\Spec R$.
 Then there is a natural isomorphism of commutative rings
$S\simeq R_\p$ forming a commutative triangular diagram with
the map~$\sigma$ and the localization map $R\rarrow R_\p$.
\end{cor}

\begin{proof}
 Since $S$ is a local ring with the maximal ideal~$\n$, all
the points of $\Spec S$ are generalizations of~$\n$.
 The map $\Spec\sigma\:\Spec S\rarrow\Spec R$ is continuous, so it
takes generalizations to generalizations.
 Hence all points in the image of $\Spec\sigma$ are generalizations
of $\p=\sigma^{-1}(\n)$.
 By Lemma~\ref{flat-epimorphisms-determined-on-spectra}, the image of
the map $\Spec\sigma$ is a generalization-closed subset in $\Spec R$.
 Therefore, the image of $\Spec\sigma$ is precisely the set of all prime
ideals $\p'\subset R$ such that $\p'\subset\p$.
 This subset of $\Spec R$ is well-known to coincide with the image of
the map $\Spec R_\p\rarrow\Spec R$.
 As both $R\rarrow S$ and $R\rarrow R_\p$ are flat epimorphisms of
commutative rings, we can conclude that these two flat ring
epimorphisms are isomorphic, in view of
Lemma~\ref{flat-epimorphisms-determined-on-spectra}.

 The following simpler alternative proof of the corollary was suggested
to the author by M.~Hrbek.
 Clearly, the ring homomorphism $R\rarrow S$ into the local ring $S$
induces a ring homomorphism $R_\p\rarrow S$ forming a commutative
triangular diagram $R\rarrow R_\p\rarrow S$.
 The map $R_\p\rarrow S$ is a ring epimorphism since the map
$R\rarrow S$ is a ring epimorphism; and $S$ is flat as an $R_\p$\+module
since it is flat as an $R$\+module.
 Moreover, $R_\p\rarrow S$ is a \emph{local} homomorphism of
commutative rings (i.~e., the image of the maximal ideal is contained
in the maximal ideal).
 It remains to show that any local flat epimorphism of commutative local
rings is an isomorphism.
 Indeed, any flat local homomorphism of commutative local rings is
faithfully flat by~\cite[Lemma Tag~00HR]{SP}, and any faithfully flat
epimorphism of commutative rings is an isomorphism
by~\cite[Lemma Tag~04VU]{SP}.
\end{proof}

\begin{proof}[Proof of
Proposition~\ref{all-flat-epis-are-FQM}]
 We need to prove that any flat epimorphism of commutative rings
$\sigma\:R\rarrow S$ is an FQM\+epimorphism.
 By Lemma~\ref{FQM-epimorphisms-determined-on-local-rings},
it suffices to check that the composition $R\rarrow S\rarrow S_\q$
is an FQM\+epimorphism for every prime ideal $\q\subset S$.
 The map $R\rarrow S_\q$ is a flat epimorphism of commutative
rings, as mentioned in the beginning of the proof of
Lemma~\ref{FQM-epimorphisms-determined-on-local-rings}.
 By Corollary~\ref{flat-epimorphisms-with-local-codomains}, we have
$S_\q\simeq R_\p$, where $\p=\sigma^{-1}(\q)$.
 Finally, $R\rarrow R_\p$ is an FQM\+epimorphism by
Corollary~\ref{localizations-by-mult-subsets-FQM}, as
$R_\p=\Sigma^{-1}R$ for $\Sigma=R\setminus\p$.
\end{proof}

\subsection{Restriction of scalars}
 In Sections~\ref{flat-epis-and-quite-quasi-modules}\+-%
\ref{flat-epis-and-quasi-modules} we considered the extension of
scalars of quasi-modules with respect to flat epimorphisms of
commutative rings.
 In this section we briefly discuss the restriction of scalars.

\begin{cor} \label{fl-ring-epi-quasi-module-restriction-of-scalars}
 Let $R\rarrow S$ be a flat epimorphism of commutative rings and $B$ be
an $S$\+$S$\+bimodule.
 Then $B$ is a quasi-module over $S$ if and only if, viewed as
an $R$\+$R$\+bimodule, $B$ is a quasi-module over~$R$.
\end{cor}

\begin{proof}
 One implication holds for any homomorphism of commutative rings
$R\rarrow S$: any quasi-module over $S$ is also a quasi-module
over~$R$.
 This is~\cite[Lemma~1.13]{Pdomc}.

 The converse implication depends on the assumption that $\sigma$~is
a flat ring epimorphism.
 If this is the case and an $S$\+$S$\+bimodule $B$ is a quasi-module
over $R$, then $S\ot_RB$ is a quasi-module over $S$
by Proposition~\ref{all-flat-epis-are-FQM}(b).
 It remains to recall that the isomorphism $S\ot_RS\simeq S$ for
a ring epimorphism $R\rarrow S$ implies a natural isomorphism
$S\ot_RN\simeq N$ for all $S$\+modules~$N$.
 So, in particular, we have an isomorphism of $S$\+$S$\+bimodules
$B\simeq S\ot_RB$.
\end{proof}

\begin{lem} \label{ring-epi-restr-of-scalars-quite-quasi-modules}
 Let $\sigma\:R\rarrow S$ be an epimorphism of commutative rings
and $B$ be an $S$\+$S$\+bimodule.
 Then $B$ is a quite quasi-module over $S$ if and only if, viewed
as an $R$\+$R$\+bimodule, $B$ is a quite quasi-module over~$R$.
 Moreover, denoting by $F$ and $F'$ the natural ordinal-indexed
increasing filtrations on the $R$\+$R$\+bimodule $B$ and on
the $S$\+$S$\+bimodule $B$, one has $F_\alpha B=F'_\alpha B$ for all
ordinals~$\alpha$.
\end{lem}

\begin{proof}
 This is an analogue of
Corollary~\ref{fl-ring-epi-quasi-module-restriction-of-scalars},
with the difference that we do not assume the ring epimorphism~$\sigma$
to be flat in the present lemma.
 One implication and inclusion hold for any homomorphism of
commutative rings $R\rarrow S$: any quite quasi-module over $S$ is also
a quite quasi-module over~$R$.
 In fact, one has $F'_\alpha B\subset F_\alpha B$ for any
$S$\+$S$\+bimodule $B$ in this context.
 Furthermore, both $F_\alpha B$ and $F'_\alpha B$ are
$S$\+$S$\+subbimodules of $B$, as one can see from the construction.

 The converse implication and inclusion depend on the assumption
that $\sigma$~is a ring epimorphism.
 In the case of a flat ring epimorphism, a suitable version of
the argument from the proof of
Corollary~\ref{fl-ring-epi-quasi-module-restriction-of-scalars}
is applicable.
 In the general case we notice that, by the definition, $F_0B$ is
the unique maximal $R$\+$R$\+subbimodule of $B$ on which the left and
right actions of $R$ agree, while $F'_0B$ is the unique maximal
$S$\+$S$\+subbimodule of $B$ on which the left and right actions of
$S$ agree.
 Now, for a ring epimorphism~$\sigma$, the action of $S$ on any
$S$\+module is uniquely determined by its restriction to $R$, as
the functor $\sigma_*\:S\Modl\rarrow R\Modl$ is fully faithful.
 Therefore, the left and right actions of $S$ on a given
$S$\+$S$\+bimodule $A$ agree if and only if the left and right actions
of $R$ on $A$ agree.
 Thus we have $F_0B=F'_0B$.
 Proceeding by induction, one easily concludes that $F_\alpha B=
F'_\alpha B$ for all ordinals~$\alpha$.
\end{proof}

\Section{Localizing Differential Operators}
\label{localizing-diffops-secn}

 In this section we discuss generalizations
of~\cite[Lemma Tag~0G36]{SP}, i.~e., the question of extension of
differential operators to localizations of commutative rings
and modules.
 The assertion of~\cite[Lemma Tag~0G36]{SP} is stated for (what we call)
\emph{strongly differential operators} in the context of localizations
of commutative rings with respect to multiplicative subsets.
 Our versions of this lemma in this paper apply to wider classes
of \emph{differential operators} and \emph{quite differential operators}
in the more general context of \emph{flat ring epimorphisms}.
 In addition, we establish versions of the sheaf axiom for affine
open coverings of affine open subschemes in the context of sheaves of
differential operators acting between quasi-coherent sheaves on schemes.

\subsection{Localizing differential operators without order}
 We start with an easy corollary of
Corollary~\ref{fl-ring-epi-quasi-module-restriction-of-scalars}
before passing to a more difficult theorem.

\begin{cor} \label{fl-ring-epi-restriction-of-scalars-diff-operators}
 Let $K\rarrow R$ be a homomorphism of commutative rings and
$\sigma\:R\rarrow S$ be a flat epimorphism of commutative rings.
 Let $U$ and $V$ be two $S$\+modules.
 Then a $K$\+linear map $U\rarrow V$ is an $S$\+differential
operator if and only if it is an $R$\+differential operator.
\end{cor}

\begin{proof}
 One implication is simple and holds for any homomorphism of commutative
rings $\sigma\:R\rarrow S$: any $S$\+differential operator $U\rarrow V$
is an $R$\+differential operator.
 The converse implication depends on the assumption that $\sigma$~is
a flat epimorphism.
 Both the assertions follow from~\cite[Lemma~1.13]{Pdomc} and
Corollary~\ref{fl-ring-epi-quasi-module-restriction-of-scalars} above.
 The point is that $\D_{R/K}(U,V)$ is the maximal $S$\+$S$\+subbimodule
of $\Hom_K(U,V)$ that is a quasi-module over $R$, while $\D_{S/K}(U,V)$
is the maximal $S$\+$S$\+subbimodule of $\Hom_K(U,V)$ that is
a quasi-module over~$S$.
 It is helpful to notice that $\D_{R/K}(U,V)$ is always
an $S$\+$S$\+subbimodule in $\Hom_K(U,V)$, since the operators of
multiplication with elements $s\in S$, being $R$\+linear maps,
are $R$\+differential operators $U\rarrow U$ and $V\rarrow V$, and
compositions of $R$\+differential operators are $R$\+differential
operators by Corollary~\ref{composition-of-differential-operators}.
\end{proof}

 The following theorem is our promised generalization
of~\cite[Lemma Tag~0G36]{SP}.

\begin{thm} \label{localizing-diff-operators-thm}
 Let $K\rarrow R$ be a homomorphism of commutative rings, and let
$U$ and $V$ be two $R$\+modules.
 Let $R\rarrow S$ be a flat epimorphism of commutative rings and
$D_R\:U\rarrow V$ be a $K$\+linear $R$\+differential operator.
 Then there exists a unique $K$\+linear $S$\+differential operator
$D_S\:S\ot_RU\rarrow S\ot_RV$ for which the square diagram
\begin{equation} \label{diffoperator-extension-commutative-diagram}
\begin{gathered}
 \xymatrix{
   U \ar[rr]^{D_R} \ar[d] && V \ar[d] \\
   S\ot_RU \ar@{..>}[rr]^{D_S} && S\ot_RV
 }
\end{gathered}
\end{equation}
is commutative.
\end{thm}

\begin{proof}
 To prove existence, notice that, for any $R$\+$R$\+bimodule $B$
over $K$, there is a natural bijective correspondence between
$R$\+$R$\+bimodule maps $B\rarrow\Hom_K(U,V)$ and left $R$\+module
maps $B\ot_RU\rarrow V$.
 Put $B=\D_{R/K}(U,V)\subset\Hom_K(U,V)$; then $B$ is
an $R$\+$R$\+bimodule over $K$ and a quasi-module over~$R$.
 The correspondence mentioned above provides a left $R$\+module map
$f\:B\ot_RU\rarrow V$.
 Furthermore, the $K$\+linear $R$\+differential operator
$D_R\:U\rarrow V$ corresponds to an element of~$B$; let us denote
this element by $b\in B$.

 By Proposition~\ref{all-flat-epis-are-FQM}(a), the tensor
product $S\ot_RB$ is naturally an $S$\+$S$\+bimodule.
 Therefore, we have
$$
 (S\ot_RB)\ot_S(S\ot_RU)\simeq S\ot_RB\ot_RU,
$$
and the left $R$\+module map $f\:B\ot_RU\rarrow V$ induces
a left $S$\+module map
$$
 (S\ot_RB)\ot_S(S\ot_RU)\simeq S\ot_RB\ot_RU
 \xrightarrow{S\ot_Rf}S\ot_RV.
$$
 The left $S$\+module map $S\ot_Rf$ corresponds to
an $S$\+$S$\+bimodule map
$$
 g\:S\ot_RB\lrarrow\Hom_K(S\ot_RU,\>S\ot_RV).
$$

 Now the ring homomorphism $R\rarrow S$ induces an $R$\+$R$\+bimodule
map $B\rarrow S\ot_RB$.
 The image of the element $b\in B$ under the composition
$$
 B\lrarrow S\ot_RB\overset g\lrarrow\Hom_K(S\ot_RU,\>S\ot_RV)
$$
provides an element of $\Hom_K(S\ot_RU,\>S\ot_RV)$, i.~e.,
a $K$\+linear map $D_S\:S\ot_RU\rarrow S\ot_RV$.
 It is straightforward to check that
the diagram~\eqref{diffoperator-extension-commutative-diagram}
is commutative.

 Finally, by Proposition~\ref{all-flat-epis-are-FQM}(b) we know that
$S\ot_RB$ is a quasi-module over~$S$.
 One can observe that the class of all quasi-modules over $S$ is closed
under homomorphic images in $S\bMod S$, and any $S$\+$S$\+subbimodule
in $\Hom_K(S\ot_RU,\>S\ot_RV)$ that is a quasi-module over $S$ is
contained in $\D_{S/K}(S\ot_RU,\>S\ot_RV)$.
 Therefore, the inclusion $g(S\ot_RB)\subset
\D_{S/K}(S\ot_RU,\>S\ot_RV)$ holds, and it follows that $D_S$ is
an $S$\+differential operator $S\ot_RU\rarrow S\ot_RV$.

 To prove uniqueness, it suffices to consider a $K$\+linear
$S$\+differential operator $D_S\:S\ot_RU\rarrow S\ot_RV$ such that
the composition $U\rarrow S\ot_RU\overset{D_S}\rarrow S\ot_RV$
vanishes.
 We need to show that $D_S=0$.
 
 For this purpose, put $V'=S\ot_RV$.
 Consider the $R$\+$R$\+bimodule $B=\D_{R/K}(U,V')$ of $K$\+linear
$R$\+differential operators $U\rarrow V'$.

 By (the simple implication in)
Corollary~\ref{fl-ring-epi-restriction-of-scalars-diff-operators},
all $K$\+linear $S$\+differential operators $S\ot_RU\rarrow S\ot_RV$
are also $R$\+differential operators.
 The natural $R$\+module map $U\rarrow S\ot_RU$ is an $R$\+differential
operator (in fact, a strongly $R$\+differential operator of ordinal 
order~$0$, as any $R$\+module map).
 By Corollary~\ref{composition-of-differential-operators},
the compositions of $R$\+differential operators are
$R$\+differential operators.
 So, for any $K$\+linear $S$\+differential operator
$D\:S\ot_RU\rarrow S\ot_RV$, the composition $U\rarrow S\ot_RU
\overset D\rarrow S\ot_RV$ is a $K$\+linear $R$\+differential operator
$U\rarrow V'$.
 We denote the resulting map by
$$
 f\:\D_{S/K}(S\ot_RU,\>S\ot_RV)\lrarrow\D_{R/K}(U,V').
$$

 Furthermore, $S$\+differential operators $S\ot_R U\rarrow
S\ot_RV$ form an $S$\+$S$\+sub\-bi\-mod\-ule in
$\Hom_K(S\ot_RU,\>S\ot_RV)$.
 In particular, for every element $s\in S$, the map $D_S\circ s\:
S\ot_R U\rarrow S\ot_RV$ is a $K$\+linear $S$\+differential operator.
 Here $s\:S\ot_RU\rarrow S\ot_RU$ is the operator of multiplication
with~$s$.

 Consider the map $h\:S\rarrow\Hom_K(U,V')$ assigning to each element
$s\in S$ the $K$\+linear map $u\longmapsto D_S(s\ot u)\:U\rarrow V'$.
 We have shown that $h(s)=f(D_S\circ s)$ is a $K$\+linear
$R$\+differential operator $U\rarrow V'$.
 So we have a map
$$
 h\:S\lrarrow B=\D_{R/K}(U,V').
$$
 One can easily see from the construction that $h$~is a right
$R$\+module morphism.
 Indeed, $h(sr)(u)=D_S(sr\ot u)=D_S(s\ot ru)=h(s)(ru)\in V'$ for all
$s\in S$, \ $r\in R$, and $u\in U$.
 (We recall that, by the definition, the right $R$\+module structure
on $\Hom_K(U,V')$ is induced by the action of $R$ on~$U$.)

 The operators $s\:V'\rarrow V'$ of multiplication with elements
$s\in S$ are $R$\+linear maps $V'\rarrow V'$, so they are strongly
$R$\+differential operators of order~$0$.
 Once again, by
Corollary~\ref{composition-of-differential-operators},
the compositions of $R$\+differential operators are
$R$\+differential operators.
 Hence, postcomposing an $R$\+differential operator
$U\rarrow V'$ with the map $s\:V'\rarrow V'$, we obtain another
$R$\+differential operator $U\rarrow V'$.
 Therefore, the left $R$\+module structure on $\D_{R/K}(U,V')$
can be extended to a left $S$\+module structure.
 So $B$ is naturally an $S$\+$R$\+bimodule.

 On the other hand, the $R$\+$R$\+bimodule $B=\D_{R/K}(U,V')$ is
a quasi-module over~$R$.
 Recall that $S\ot_RS\simeq S$, since $R\rarrow S$ is a ring
epimorphism.
 Hence, for any $S$\+module $N$, we have a natural isomorphism
of $S$\+modules $S\ot_RN\simeq N$.
 Using Proposition~\ref{all-flat-epis-are-FQM}(a),
we now have $R$\+$R$\+bimodule isomorphisms $B\simeq S\ot_RB\simeq
B\ot_RS$.
 Thus $B$ is actually an $S$\+$S$\+bimodule.
 In particular, the right $R$\+module structure on $B$ can be also
extended to a right $S$\+module structure.

 Finally, we return to our right $R$\+module morphism $h\:S\rarrow B$.
 Both the right $R$\+module structures on $S$ and $B$ arise from
right $S$\+module structures.
 Since $R\rarrow S$ is a ring epimorphism, the functor of restriction
of scalars $\Modr S\rarrow\Modr R$ is fully faithful.
 It follows that $h$~is a right $S$\+module morphism.

 Now the map $h(1)\:U\rarrow V'$ is, by the definition, given by
the rule $h(1)(u)=D_S(1\ot u)$ for all $u\in U$.
 By assumption, we have $D_S(1\ot u)=0$ for all $u\in U$; so $h(1)=0$.
 As $h$~is a right $S$\+module map $S\rarrow B$, it follows that
the whole map~$h$ vanishes, $h=0$.
 We have proved that $D_S(s\ot u)=h(s)(u)=0$ for all $s\in S$ and
$u\in U$.
 Thus $D_S=0$, as desired.
\end{proof}

\begin{lem} \label{differential-operator-restricted-to-submodules}
 Let $K\rarrow R$ be a homomorphism of commutative rings, and let
$f\:U\rarrow U'$ and $g\:V\rarrow V'$ be two homomorphisms of
$R$\+modules.
 Assume that the map~$g$ is injective, and suppose given a commutative
diagram of $K$\+linear maps
$$
 \xymatrix{
  U \ar[r]^D \ar[d]_f & V \ar@{>->}[d]^g \\
  U' \ar[r]^{D'} & V'
 }
$$
 In this setting, if $D'\:U'\rarrow V'$ is an $R$\+differential
operator, then $D\:U\rarrow V$ is an $R$\+differential operator.
\end{lem}

\begin{proof}
 The map $f\:U\rarrow U'$ is $R$\+linear, so it is a strongly
$R$\+differential operator of order~$0$.
 By Corollary~\ref{composition-of-differential-operators},
the composition $D'\circ f\:U\rarrow V'$ is an $R$\+differential
operator.
 Now consider two $R$\+$R$\+bimodules $E=\Hom_K(U,V)$ and
$E'=\Hom_K(U,V')$.
 The injective $R$\+module map~$g$ induces an injective
$R$\+$R$\+bimodule map $g_*\:E\rarrow E'$.
 For convenience of notation, put $T=R\ot_KR$, and denote by $I$
the kernel ideal of the multiplication map $R\ot_KR\rarrow R$.
 Then we have $g_*(D)=D'\circ f\in\Gamma_I(E')$, and by
Lemma~\ref{Gamma-I-left-exact} it follows that $D\in\Gamma_I(E)$.
\end{proof}

 The next proposition establishes the sheaf axiom for affine open
coverings of affine schemes in the context of the construction of
a sheaf $\D_{X/T}(\U,\V)$ for quasi-coherent sheaves $\U$ and $\V$
on a scheme $X$, as per the discussion in
Section~\ref{introd-localizations-of-diff-operators}.

\begin{prop} \label{sheaf-axiom-differential-operators}
 Let $K\rarrow R$ be a homomorphism of commutative rings, and let
$R\rarrow S_l$, \,$1\le l\le n$, be a finite collection of
homomorphisms of commutative rings such that the collection of induced
maps of the spectra\/ $\Spec S_l\rarrow\Spec R$ is an affine open
covering of the affine scheme\/ $\Spec R$.
 Let $U$ and $V$ be two $R$\+modules, and let $D_l\: S_l\ot_RU\rarrow
S_l\ot_R V$ be $K$\+linear $S_l$\+differential operators, defined for
all indices\/ $1\le l\le n$.
 For every pair of indices $j$ and~$l$, put $S_{jl}=S_j\ot_RS_l$.
 Assume that, for every pair of indices $j$ and~$l$,
the $S_{jl}$\+differential operator $D_{jl}\:S_{jl}\ot_RU\rarrow
S_{jl}\ot_RV$ induced by $D_l$, as per the construction of
Theorem~\ref{localizing-diff-operators-thm}, is equal to
the $S_{lj}$\+differential operator $D_{lj}\:S_{lj}\ot_RU\rarrow
S_{lj}\ot_RV$ induced by~$D_j$.
 Then there exists a unique $R$\+differential operator $D\:U\rarrow V$
such that, for every index~$l$, the $S_l$\+differential operator
$D_l$ is induced by $D$ as per the construction of
Theorem~\ref{localizing-diff-operators-thm}.
\end{prop}

\begin{proof}
 For any $R$\+module $M$, the \v Cech coresolution
\begin{multline}
 0\lrarrow M\lrarrow\bigoplus\nolimits_{1\le l\le n}S_l\ot_RM\lrarrow
 \bigoplus\nolimits_{1\le j<l\le n}S_j\ot_RS_l\ot_RM \\
 \lrarrow\dotsb\lrarrow S_1\ot_R\dotsb\ot_R S_n\ot_RM\lrarrow0
\end{multline}
is a finite exact sequence of $R$\+modules.
 In the situation at hand, consider the diagram of
$K$\+linear maps
\begin{equation}
\begin{gathered}
 \xymatrix{
  0 \ar[r] & U \ar[r] \ar@{..>}[d]^D
  & \bigoplus\nolimits_{1\le l\le n}S_l\ot_RU \ar[r] \ar[d]^{(D_l)}
  & \bigoplus\nolimits_{1\le j<l\le n}S_j\ot_RS_l\ot_RU
  \ar[d]^{(D_{jl})} \\
  0 \ar[r] & V \ar[r]
  & \bigoplus\nolimits_{1\le l\le n}S_l\ot_RV \ar[r]
  & \bigoplus\nolimits_{1\le j<l\le n}S_j\ot_RS_l\ot_RV
 }
\end{gathered}
\end{equation}
where the middle and rightmost vertical arrows are the direct sums
of the operators $D_l$ and~$D_{jl}$.
 The rightmost square is commutative by assumption, so passing to
the kernels provides a $K$\+linear map $D\:U\rarrow V$.
 Put $U'=\bigoplus_{1\le l\le n}S_l\ot_RU$ and
$V'=\bigoplus_{1\le l\le n}S_l\ot_RV$.
 For every index~$l$, the map $D_l\:S_l\ot_RU\rarrow S_l\ot_R V$
is an $S_l$\+differential operator; hence by (the simple implication in)
Corollary~\ref{fl-ring-epi-restriction-of-scalars-diff-operators}
it is also an $R$\+differential operator.
 It follows that the direct sum $(D_l)_{l=1}^n\:U'\rarrow V'$ is
an $R$\+differential operator.
 Applying Lemma~\ref{differential-operator-restricted-to-submodules},
we conclude that $D$ is also an $R$\+differential operator.
 This proves the existence; the uniqueness follows immediately from
injectivity of the natural map $V\rarrow V'$.
\end{proof}

\subsection{Localizing differential operators of transfinite order}
 We start with an easy corollary of
Lemma~\ref{ring-epi-restr-of-scalars-quite-quasi-modules}.

\begin{cor} \label{ring-epi-restr-of-scalars-quite-diff-operators}
 Let $K\rarrow R$ be a homomorphism of commutative rings and
$\sigma\:R\rarrow S$ be an epimorphism of commutative rings.
 Let $U$ and $V$ be two $S$\+modules.
 Then a $K$\+linear map $U\rarrow V$ is a quite $S$\+differential
operator of ordinal order~$\alpha$ if and only if it is
a quite $R$\+differential operator of ordinal order~$\alpha$.
\end{cor}

\begin{proof}
 Once again, the simpler implication and inequality hold for any
homomorphism of commutative rings $\sigma\:R\rarrow S$: any quite
$S$\+differential operator $U\rarrow V$ of ordinal order~$\alpha$
is a quite $R$\+differential operator of ordinal order at most~$\alpha$.
 The converse implication and inequality depend on the assumption that
$\sigma$~is a ring epimorphism.
 All these assertions follow from the respective versions of
the second assertion of
Lemma~\ref{ring-epi-restr-of-scalars-quite-quasi-modules}
applied to the $S$\+$S$\+bimodule $B=E=\Hom_K(U,V)$.
\end{proof}

 The next proposition is our second generalization
of~\cite[Lemma Tag~0G36]{SP}.

\begin{prop} \label{localizing-quite-diff-operators-prop}
 Let $K\rarrow R$ be a homomorphism of commutative rings, and let
$U$ and $V$ be two $R$\+modules.
 Let $R\rarrow S$ be a flat epimorphism of commutative rings and
$D_R\:U\rarrow V$ be a $K$\+linear quite $R$\+differential operator
of ordinal order~$\alpha$.
 Then there exists a unique $K$\+linear quite $S$\+differential operator
$D_S\:S\ot_RU\rarrow S\ot_RV$ for which the square diagram
\begin{equation} \label{diffoperator-extension-commutative-diagram-II}
\begin{gathered}
 \xymatrix{
   U \ar[rr]^{D_R} \ar[d] && V \ar[d] \\
   S\ot_RU \ar@{..>}[rr]^{D_S} && S\ot_RV
 }
\end{gathered}
\end{equation}
is commutative.
 The ordinal order of the differential operator $D_S$ does not
exceed~$\alpha$.
\end{prop}

\begin{proof}
 The uniqueness follows from
Theorem~\ref{localizing-diff-operators-thm}.
 The existence is provable similarly to the proof of
Theorem~\ref{localizing-diff-operators-thm}, using
Lemma~\ref{quite-quasi-module-extension-of-scalars}
instead of Proposition~\ref{all-flat-epis-are-FQM}.
 Instead of the $R$\+$R$\+bimodule $B=\D_{R/K}(U,V)$, consider
the $R$\+$R$\+bimodule $B=F_\alpha\Hom_K(U,V)\subset\Hom_K(U,V)$.
 Then the quite $R$\+differential operator $D_R$ corresponds to
an element $b\in B$.
 By Lemma~\ref{quite-quasi-module-extension-of-scalars}(b), we have
$S\ot_RB=F_\alpha(S\ot_RB)$.
 The fact that the natural ordinal-indexed increasing filtration $F$ on
$S$\+$S$\+bimodules is preserved by $S$\+$S$\+bimodule maps
needs to be used in order to conclude that the operator $D_S$
constructed as in Theorem~\ref{localizing-diff-operators-thm}
belongs to $F_\alpha\Hom_K(S\ot_RU,\>S\ot_RV)$.
\end{proof}

\begin{lem} \label{quite-diff-operator-restricted-to-submodules}
 Let $K\rarrow R$ be a homomorphism of commutative rings, and let
$f\:U\rarrow U'$ and $g\:V\rarrow V'$ be two homomorphisms of
$R$\+modules.
 Assume that the map~$g$ is injective, and suppose given a commutative
diagram of $K$\+linear maps
$$
 \xymatrix{
  U \ar[r]^D \ar[d]_f & V \ar@{>->}[d]^g \\
  U' \ar[r]^{D'} & V'
 }
$$
 In this setting, if $D'\:U'\rarrow V'$ is a quite $R$\+differential
operator of ordinal order~$\alpha$, then $D\:U\rarrow V$ is a quite
$R$\+differential operator of ordinal order at most~$\alpha$.
\end{lem}

\begin{proof}
 Similar to the proof of
Lemma~\ref{differential-operator-restricted-to-submodules}
and based on Lemma~\ref{Gamma-I-qu-left-exact}.
\end{proof}

 The following proposition claims the sheaf axiom for affine open
coverings of affine schemes in the context of the construction of
an ordinal-filtered sheaf $\D_{X/T}^\qu(\U,\V)$ of quite differential
operators $\U\rarrow\V$ for quasi-coherent sheaves $\U$ and $\V$
on a scheme~$X$
(see Section~\ref{introd-localizations-of-diff-operators}).

\begin{prop} \label{sheaf-axiom-quite-differential-operators}
 Let $K\rarrow R$ be a homomorphism of commutative rings, and let
$R\rarrow S_l$, \,$1\le l\le n$, be a finite collection of
homomorphisms of commutative rings such that the collection of induced
maps of the spectra\/ $\Spec S_l\rarrow\Spec R$ is an affine open
covering of the affine scheme\/ $\Spec R$.
 Let $\alpha$~be an ordinal.
 Let $U$ and $V$ be two $R$\+modules, and let $D_l\: S_l\ot_RU\rarrow
S_l\ot_R V$ be $K$\+linear quite $S_l$\+differential operators of
ordinal order at most~$\alpha$, defined for all indices\/ $1\le l\le n$.
 For every pair of indices $j$ and~$l$, put $S_{jl}=S_j\ot_RS_l$.
 Assume that, for every pair of indices $j$ and~$l$,
the quite $S_{jl}$\+differential operator $D_{jl}\:S_{jl}\ot_RU\rarrow
S_{jl}\ot_RV$ induced by $D_l$, as per the construction of
Proposition~\ref{localizing-quite-diff-operators-prop}, is equal to
the quite $S_{lj}$\+differential operator $D_{lj}\:S_{lj}\ot_RU\rarrow
S_{lj}\ot_RV$ induced by~$D_j$.
 Then there exists a unique quite $R$\+differential operator
$D\:U\rarrow V$ of ordinal order at most~$\alpha$ such that, for every
index~$l$, the quite $S_l$\+differential operator $D_l$ is induced
by $D$ as per the construction of
Proposition~\ref{localizing-quite-diff-operators-prop}.
\end{prop}

\begin{proof}
 Similar to Proposition~\ref{sheaf-axiom-differential-operators}, and
using (the simple implication in)
Corollary~\ref{ring-epi-restr-of-scalars-quite-diff-operators} together
with Lemma~\ref{quite-diff-operator-restricted-to-submodules}.
\end{proof}

\Section{Colocalizing Differential Operators}
\label{colocalizing-diffops-secn}

 In this section we prove dual-analogous versions of our generalizations
of~\cite[Lemma Tag~0G36]{SP}, pertaining to the colocalizations of
modules rather than the localizations.
 We also establish versions of the sheaf axiom for affine open coverings
of affine open subschemes in the context of sheaves of differential
operators acting between contraherent cosheaves on schemes.

\subsection{Colocalizing differential operators without order}
 The following theorem is dual-analogous to
Theorem~\ref{localizing-diff-operators-thm}.

\begin{thm} \label{colocalizing-diff-operators-thm}
 Let $K\rarrow R$ be a homomorphism of commutative rings, and let
$U$ and $V$ be two $R$\+modules.
 Let $R\rarrow S$ be a flat epimorphism of commutative rings and
$D_R\:U\rarrow V$ be a $K$\+linear $R$\+differential operator.
 Then there exists a unique $K$\+linear $S$\+differential operator
$D_S\:\Hom_R(S,U)\rarrow\Hom_R(S,V)$ for which the square diagram
\begin{equation} \label{diffoperator-coextension-commutative-diagram}
\begin{gathered}
 \xymatrix{
   \Hom_R(S,U) \ar@{..>}[rr]^{D_S} \ar[d] && \Hom_R(S,V) \ar[d] \\
   U \ar[rr]^{D_R} && V
 }
\end{gathered}
\end{equation}
is commutative.
\end{thm}

\begin{proof}
 To prove existence, notice that, for any $R$\+$R$\+bimodule $B$ over
$K$, there is a natural bijective correspondence between
$R$\+$R$\+bimodule maps $B\rarrow\Hom_K(U,V)$ and left $R$\+module
maps $U\rarrow\Hom_R(B,V)$.
 Here $\Hom_R(B,V)$ denotes the abelian group of all maps $B\rarrow V$
that are $R$\+linear with respect to the left $R$\+module structure
on~$B$ (and the only given $R$\+module structure on~$V$).
 The left $R$\+module structure on $\Hom_R(B,V)$ is induced by
the right $R$\+module structure on~$B$.

 As in the proof of existence in
Theorem~\ref{localizing-diff-operators-thm}, we put
$B=\D_{R/K}(U,V)\subset\Hom_K(U,V)$; so $B$ is an $R$\+$R$\+bimodule
over $K$ and a quasi-module over~$R$.
 The correspondence mentioned above provides a left $R$\+module map
$f\:U\rarrow\Hom_R(B,V)$.
 The $K$\+linear $R$\+differential operator $D_R\:U\rarrow V$
corresponds to an element of~$B$; we denote this element by $b\in B$.

 By Proposition~\ref{all-flat-epis-are-FQM}(a), the tensor
product $B\ot_RS$ is naturally an $S$\+$S$\+bimodule.
 Therefore, we have isomorphisms of left $S$\+modules
$$
 \Hom_S(B\ot_RS,\>\Hom_R(S,V))\simeq\Hom_R(B\ot_RS,\>V)
 \simeq\Hom_R(S,\Hom_R(B,V)).
$$
 Here all the symbols $\Hom_R$ and $\Hom_S$ denote the groups of
homomorphisms of modules with respect to the left module structures,
while the left $X$\+module structure on $\Hom_Y(Z,W)$ is always induced
by the right $X$\+module structure on~$Z$.
 The left $R$\+module map $f\:U\rarrow\Hom_R(B,V)$ induces a left
$S$\+module map
$$
 \Hom_R(S,U)\xrightarrow{\Hom_R(S,f)}\Hom_R(S,\Hom_R(B,V))\simeq
 \Hom_S(B\ot_RS,\>\Hom_R(S,V)).
$$
 The left $S$\+module map $\Hom_R(S,f)$ corresponds to
an $S$\+$S$\+bimodule map
$$
 g\:B\ot_RS\lrarrow\Hom_K(\Hom_R(S,U),\Hom_R(S,V)).
$$

 As in the proof of Theorem~\ref{localizing-diff-operators-thm},
the ring homomorphism $R\rarrow S$ induces an $R$\+$R$\+bimodule map
$B\rarrow B\ot_RS$.
 The image of the element $b\in B$ under the composition
$$
 B\lrarrow B\ot_RS\overset g\lrarrow
 \Hom_K(\Hom_R(S,U),\Hom_R(S,V))
$$
provides an element of $\Hom_K(\Hom_R(S,U),\Hom_R(S,V))$, i.~e.,
a $K$\+linear map $D_S\:\Hom_R(S,U)\rarrow\Hom_R(S,V)$.
 It is straightforward to check that
the diagram~\eqref{diffoperator-coextension-commutative-diagram}
is commutative.

 Finally, by Proposition~\ref{all-flat-epis-are-FQM}(b) we know that
$B\ot_RS$ is a quasi-module over~$S$.
 Similarly to the proof of Theorem~\ref{localizing-diff-operators-thm},
it follows that the inclusion $g(B\ot_RS)\subset
\D_{R/K}(\Hom_R(S,U),\Hom_R(S,V))$ holds.
 Thus $D_S$ is an $S$\+differential operator $\Hom_R(S,U)\rarrow
\Hom_R(S,V)$.  {\emergencystretch=1em\par}

 To prove uniqueness, it suffices to consider a $K$\+linear
$S$\+differential operator $D_S\:\Hom_R(S,U)\rarrow\Hom_R(S,V)$
such that the composition $\Hom_R(S,U)\overset{D_S}\rarrow
\Hom_R(S,V)\rarrow V$ vanishes.
 We need to show that $D_S=0$.

 For this purpose, put $U'=\Hom_R(S,U)$.
 Denote by $p\:\Hom_R(S,V)\rarrow V$ the natural $R$\+module map.
 Consider the $R$\+$R$\+bimodule $B=\D_{R/K}(U',V)$ of $K$\+linear
$R$\+differential operators $U'\rarrow V$.

 By (the simple implication in)
Corollary~\ref{fl-ring-epi-restriction-of-scalars-diff-operators},
all $K$\+linear $S$\+differential operators $\Hom_R(S,U)\rarrow
\Hom_R(S,V)$ are also $R$\+differential operators.
 The natural $R$\+module map~$p$ is a strongly $R$\+differential
operator of ordinal order~$0$.
 By Corollary~\ref{composition-of-differential-operators},
the compositions of $R$\+differential operators are
$R$\+differential operators.
 So, for any $K$\+linear $S$\+differential operator
$D\:\Hom_R(S,U)\rarrow\Hom_R(S,V)$, the composition $\Hom_R(S,U)
\overset D\rarrow\Hom_R(S,V)\overset p\rarrow V$ is a $K$\+linear
$R$\+differential operator $U'\rarrow V$.
 We denote the resulting map by
$$
 f\:\D_{S/K}(\Hom_R(S,U),\Hom_R(S,V))\lrarrow\D_{R/K}(U',V).
$$

 Furthermore, $S$\+differential operators $\Hom_R(S,U)\rarrow
\Hom_R(S,V)$ form an $S$\+$S$\+subbimodule in
$\Hom_K(\Hom_R(S,U),\Hom_R(S,V))$.
 In particular, for every element $s\in S$, the map
$s\circ D_S\:\Hom_R(S,U)\rarrow\Hom_R(S,V)$ is a $K$\+linear
$S$\+differential operator.
 Here $s\:\Hom_R(S,V)\rarrow\Hom_R(S,V)$ is the operator of
multiplication with~$s$.

 Consider the map $h\:S\rarrow\Hom_K(U',V)$ assigning to each
element $s\in S$ the $K$\+linear map $p\circ s\circ D_S\:U'
\rarrow\Hom_R(S,V)\rarrow V$.
 We have shown that $h(s)=f(s\circ D_S)$ is a $K$\+linear
$R$\+differential operator $U'\rarrow V$.
 So we have a map
$$
 h\:S\rarrow B=\D_{R/K}(U',V).
$$
 One can easily see from the construction that $h$~is
a left $R$\+module morphism.
 Indeed, $h(rs)(u')=p(rsD_S(u'))=rp(sD_S(u'))$ for all $r\in R$,
\ $s\in S$, and $u'\in U'$.
 (We recall that, by the definition, the left $R$\+module structure
on $\Hom_K(U',V)$ is induced by the action of $R$ on~$V$.)

 The operators $s\:U'\rarrow U'$ of multiplication with elements
$s\in S$ are $R$\+linear maps $U'\rarrow U'$, so they are strongly
$R$\+differential operators of order~$0$.
 Once again, by
Corollary~\ref{composition-of-differential-operators},
the compositions of $R$\+differential operators are
$R$\+differential operators.
 Hence, precomposing an $R$\+differential operator
$U'\rarrow V$ with the map $s\:U'\rarrow U'$, we obtain another
$R$\+differential operator $U'\rarrow V$.
 Therefore, the right $R$\+module structure on $\D_{R/K}(U',V)$
can be extended to a right $S$\+module structure.
 So $B$ is naturally an $R$\+$S$\+bimodule.

 On the other hand, the $R$\+$R$\+bimodule $B=\D_{R/K}(U',V)$ is
a quasi-module over~$R$.
 Similarly to the proof of Theorem~\ref{localizing-diff-operators-thm},
one can use the isomorphisms $B\simeq B\ot_RS\simeq S\ot_RB$ in order
to show that the left $R$\+module structure on $B$ can be also extended
to a left $S$\+module structure.
 As $R\rarrow S$ is a ring epimorphism, it follows that $h$~is a left
$S$\+module morphism.

 Finally, the map $h(1)\:U'\rarrow V$ is, by the definition, given
by the formula $h(1)=p\circ D_S$.
 By assumption, we have $h(1)=0$.
 As $h$~is a left $S$\+module map $S\rarrow B$, it follows that
the whole map~$h$ vanishes, $h=0$.
 Recall that $v'(s)=(sv')(1)=p(sv')\in V$ for all $v'\in\Hom_R(S,V)$.
 We have proved that $D_S(u')(s)=(s\circ D_S)(u')(1)=
p\circ s\circ D_S(u')=h(s)(u')=0$ for all $s\in S$ and $u'\in U'$.
 Thus $D_S=0$, as desired.
\end{proof}

\begin{lem} \label{differential-operator-corestricted-to-quotients}
 Let $K\rarrow R$ be a homomorphism of commutative rings, and let
$f\:U'\rarrow U$ and $g\:V'\rarrow V$ be two homomorphisms of
$R$\+modules.
 Assume that the map~$f$ is surjective, and suppose given a commutative
diagram of $K$\+linear maps
$$
 \xymatrix{
  U' \ar[r]^{D'} \ar@{->>}[d]_f & V' \ar[d]^g \\
  U \ar[r]^D & V
 }
$$
 In this setting, if $D'\:U'\rarrow V'$ is an $R$\+differential
operator, then $D\:U\rarrow V$ is an $R$\+differential operator.
\end{lem}

\begin{proof}
 This is dual-analogous to
Lemma~\ref{differential-operator-restricted-to-submodules}.
 The map $g\:V'\rarrow V$ is $R$\+linear, so it is a strongly
$R$\+differential operator of order~$0$.
 By Corollary~\ref{composition-of-differential-operators},
the composition $g\circ D'\:U'\rarrow V$ is an $R$\+differential
operator.
 Now consider two $R$\+$R$\+bimodules $E=\Hom_K(U,V)$ and
$E'=\Hom_K(U',V)$.
 The surjective $R$\+module map~$f$ induces an injective
$R$\+$R$\+bimodule map $f^*\:E\rarrow E'$.
 Put $T=R\ot_KR$, and denote by $I$ the kernel ideal of
the multiplication map $R\ot_KR\rarrow R$.
 Then we have $f^*(D)=g\circ D'\in\Gamma_I(E')$, and by
Lemma~\ref{Gamma-I-left-exact} it follows that $D\in\Gamma_I(E)$.
\end{proof}

 Let $R$ be a commutative ring and $C$ be an $R$\+module.
 One says that the $R$\+module $C$ is \emph{contraadjusted} if
$\Ext^1_R(R[r^{-1}],C)=0$ for every $r\in R$.
 Here the commutative $R$\+algebra $R[r^{-1}]$ is viewed as
an $R$\+module.
 Notice that the projective dimension of the $R$\+module $R[r^{-1}]$
can never exceed~$1$ \,\cite[proof of Lemma~2.1]{Pcta}; that is
the reason why the $\Ext_R^n$ vanishing condition for $n\ge2$
is not imposed.
 It follows that any quotient module of a contraadjusted $R$\+module
is contraadjusted.
 The class of contraadjusted $R$\+modules is also closed under
extensions and infinite direct products, and contains all injective
$R$\+modules.
 So any $R$\+module $M$ admits a two-term coresolution $0\rarrow M
\rarrow C^0\rarrow C^1\rarrow0$ with contraadjusted $R$\+modules
$C^0$ and~$C^1$.
 We refer to~\cite[Section~1.1]{Pcosh}, \cite[Section~2]{Pcta},
and~\cite[Section~4.3]{Pphil} for further details on contraadjusted
modules over commutative rings.

 The following proposition establishes the sheaf axiom for affine open
coverings of affine schemes in the context of the construction of
a sheaf $\D_{X/T}(\fU,\fV)$ for contraherent cosheaves $\fU$ and $\fV$
on a scheme $X$, as per the discussion in
Section~\ref{introd-colocalizations-of-diff-operators}.

\begin{prop} \label{cosheaf-axiom-differential-operators}
 Let $K\rarrow R$ be a homomorphism of commutative rings, and let
$R\rarrow S_l$, \,$1\le l\le n$, be a finite collection of
homomorphisms of commutative rings such that the collection of induced
maps of the spectra\/ $\Spec S_l\rarrow\Spec R$ is an affine open
covering of the affine scheme\/ $\Spec R$.
 Let $U$ be a contraadjusted $R$\+module and $V$ be an $R$\+module.
 Let $D_l\:\Hom_R(S_l,U)\rarrow\Hom_R(S_l,V)$ be $K$\+linear
$S_l$\+differential operators, defined for all indices\/ $1\le l\le n$.
 For every pair of indices $j$ and~$l$, put $S_{jl}=S_j\ot_RS_l$.
 Assume that, for every pair of indices $j$ and~$l$,
the $S_{jl}$\+differential operator $D_{jl}\:\Hom_R(S_{jl},U)\rarrow
\Hom_R(S_{jl},V)$ induced by $D_j$, as per the construction of
Theorem~\ref{colocalizing-diff-operators-thm}, is equal to
the $S_{lj}$\+differential operator $D_{lj}\:\Hom_R(S_{lj},U)\rarrow
\Hom_R(S_{lj},V)$ induced by~$D_l$.
 Then there exists a unique $R$\+differential operator $D\:U\rarrow V$
such that, for every index~$l$, the $S_l$\+differential operator
$D_l$ is induced by $D$ as per the construction of
Theorem~\ref{colocalizing-diff-operators-thm}.
\end{prop}

\begin{proof}
 This proposition is dual-analogous to
Proposition~\ref{sheaf-axiom-differential-operators}.
 For any contraadjusted $R$\+module $C$, the \v Cech resolution
\begin{multline} \label{cech-resolution}
 0\lrarrow\Hom_R(S_1\ot_R\dotsb\ot_RS_n,\>C)\lrarrow\dotsb \\
 \lrarrow\bigoplus\nolimits_{1\le j<l\le n}\Hom_R(S_j\ot_RS_l,\>C) \\
 \lrarrow\bigoplus\nolimits_{1\le l\le n}\Hom_R(S_l,C)\lrarrow
 C\lrarrow0
\end{multline}
is a finite exact sequence of $R$\+modules~\cite[formula~(1.3)
in Lemma~1.2.6(b)]{Pcosh}.
 For an arbitrary (not necessarily contraadjusted) $R$\+module $C$,
the sequence~\eqref{cech-resolution} is a (not necessarily exact)
complex of $R$\+modules.

 In the situation at hand, consider the diagram of $K$\+linear maps
\begin{equation}
\begin{gathered}
 \xymatrix{
  \bigoplus\nolimits_{1\le j<l\le n}\Hom_R(S_j\ot_RS_l,\>U)
  \ar[r] \ar[d]^{(D_{jl})}
  & \bigoplus\nolimits_{1\le l\le n}\Hom_R(S_l,U)
  \ar[r] \ar[d]^{(D_l)}
  & U \ar[r] \ar@{..>}[d]^D & 0 \\
  \bigoplus\nolimits_{1\le j<l\le n}\Hom_R(S_j\ot_RS_l,\>V) \ar[r]
  & \bigoplus\nolimits_{1\le l\le n}\Hom_R(S_l,V) \ar[r]
  & V
 }
\end{gathered}
\end{equation}
where the leftmost and middle vertical arrows are the direct sums of
the operators $D_{jl}$ and~$D_l$.
 The leftmost square is commutative by assumption, so passing to
the cokernel provides a $K$\+linear map $D\:U\rarrow V$.
 Put $U'=\bigoplus_{1\le l\le n}\Hom_R(S_l,U)$ and
$V'=\bigoplus_{1\le l\le n}\Hom_R(S_l,V)$.
 For every index~$l$, the map $D_l\:\Hom_R(S_l,U)\rarrow\Hom_R(S_l,V)$
is an $S_l$\+differential operator; hence by (the simple implication in)
Corollary~\ref{fl-ring-epi-restriction-of-scalars-diff-operators}
it is also an $R$\+differential operator.
 It follows that the direct sum $(D_l)_{l=1}^n\:U'\rarrow V'$ is
an $R$\+differential operator.
 Applying Lemma~\ref{differential-operator-corestricted-to-quotients},
we conclude that $D$ is also an $R$\+differential operator.
 This proves the existence; the uniqueness follows immediately from
surjectivity of the natural map $U'\rarrow U$.
\end{proof}

\begin{rem} \label{contraadjustedness-necessary-remark}
 The following counterexample shows that the contraadjustedness
assumption on the $R$\+module $U$ in
Proposition~\ref{cosheaf-axiom-differential-operators}
\emph{cannot} be dropped.
 Let $K=k$ be a field and $R=k[x]$ be the ring of polynomials in
one variable~$x$ over~$k$.
 Let $f_1$ and $f_2\in R$ be two coprime polynomials of
degrees~$\ge\nobreak1$ in~$x$; e.~g., one can take
$f_1=x$ and $f_2=x+1$.
 Consider the rings $S_1=R[f_1^{-1}]$ and $S_2=R[f_2^{-1}]$.
 Then the natural inclusions of commutative rings $R\rarrow S_1$
and $R\rarrow S_2$ have the property that the induced maps of
the spectra $\Spec S_1\rarrow\Spec R$ and $\Spec S_2\rarrow R$
form a covering of the affine scheme $\Spec R$ by two affine open
subschemes.
 Consider the free $R$\+module $U=V=R$, which is \emph{not}
contraadjusted.
 Then we have $\Hom_R(S_1,U)=\Hom_R(S_2,U)=0$ (as there are no nonzero
infinitely $f_1$\+divisible or $f_2$\+divisible elements in~$U$).
 Consider two $R$\+linear maps $D'$ and $D''\:U\rarrow V$, viz.,
the identity map $D'=\id_U$ and the zero map $D''=0$.
 Then both $D'$ and $D''$ are $K$\+linear strongly $R$\+differential
operators of order~$\le0$.
 The induced maps $D'_1$ and $D''_1\:\Hom_R(S_1,U)\rarrow\Hom_R(S_1,V)$,
as well as $D'_2$ and $D''_2\:\Hom_R(S_2,U)\rarrow\Hom_R(S_2,V)$, are
all zero maps in this example, as these are maps of zero modules.
 So we have $D'_1=D''_1$ and $D'_2=D''_2$ while $D'\ne D''$.
 Thus, in this example with a noncontraadjusted $R$\+module $U$,
the uniqueness assertion of
Proposition~\ref{cosheaf-axiom-differential-operators}
does \emph{not} hold.
\end{rem}

\subsection{Colocalizing differential operators of transfinite order}
 The following proposition is dual-analogous to
Proposition~\ref{localizing-quite-diff-operators-prop}.

\begin{prop} \label{colocalizing-quite-diff-operators-prop}
 Let $K\rarrow R$ be a homomorphism of commutative rings, and let
$U$ and $V$ be two $R$\+modules.
 Let $R\rarrow S$ be a flat epimorphism of commutative rings and
$D_R\:U\rarrow V$ be a $K$\+linear quite $R$\+differential operator
of ordinal order~$\alpha$.
 Then there exists a unique $K$\+linear quite $S$\+differential operator
$D_S\:\Hom_R(S,U)\rarrow\Hom_R(S,V)$ for which the square diagram
\begin{equation} \label{diffoperator-coextension-commutative-diagram-II}
\begin{gathered}
 \xymatrix{
   \Hom_R(S,U) \ar@{..>}[rr]^{D_S} \ar[d] && \Hom_R(S,V) \ar[d] \\
   U \ar[rr]^{D_R} && V
 }
\end{gathered}
\end{equation}
is commutative.
 The ordinal order of the differential operator $D_S$ does not
exceed~$\alpha$.
\end{prop}

\begin{proof}
 The uniqueness follows from
Theorem~\ref{colocalizing-diff-operators-thm}.
 The existence is provable similarly to the proof of
Theorem~\ref{colocalizing-diff-operators-thm}, using
Lemma~\ref{quite-quasi-module-extension-of-scalars}
instead of Proposition~\ref{all-flat-epis-are-FQM}.
 We refer to the proof of
Proposition~\ref{localizing-quite-diff-operators-prop} for the details.
\end{proof}

\begin{lem} \label{quite-diff-operator-corestricted-to-quotients}
 Let $K\rarrow R$ be a homomorphism of commutative rings, and let
$f\:U'\rarrow U$ and $g\:V'\rarrow V$ be two homomorphisms of
$R$\+modules.
 Assume that the map~$f$ is surjective, and suppose given a commutative
diagram of $K$\+linear maps
$$
 \xymatrix{
  U' \ar[r]^{D'} \ar@{->>}[d]_f & V' \ar[d]^g \\
  U \ar[r]^D & V
 }
$$
 In this setting, if $D'\:U'\rarrow V'$ is a quite $R$\+differential
operator of ordinal order~$\alpha$, then $D\:U\rarrow V$ is a quite
$R$\+differential operator of ordinal order at most~$\alpha$.
\end{lem}

\begin{proof}
 This is similar to
Lemma~\ref{differential-operator-corestricted-to-quotients}
and dual-analogous to
Lemma~\ref{quite-diff-operator-restricted-to-submodules}.
 The argument is based on Lemma~\ref{Gamma-I-qu-left-exact}.
\end{proof}

 Our final proposition is dual-analogous to
Proposition~\ref{sheaf-axiom-quite-differential-operators}.

\begin{prop} \label{cosheaf-axiom-quite-differential-operators}
 Let $K\rarrow R$ be a homomorphism of commutative rings, and let
$R\rarrow S_l$, \,$1\le l\le n$, be a finite collection of
homomorphisms of commutative rings such that the collection of induced
maps of the spectra\/ $\Spec S_l\rarrow\Spec R$ is an affine open
covering of the affine scheme\/ $\Spec R$.
 Let $\alpha$ be an ordinal, $U$ be a contraadjusted $R$\+module,
and $V$ be an $R$\+module.
 Let $D_l\:\Hom_R(S_l,U)\rarrow\Hom_R(S_l,V)$ be $K$\+linear quite
$S_l$\+differential operators of ordinal order at most~$\alpha$,
defined for all indices\/ $1\le l\le n$.
 For every pair of indices $j$ and~$l$, put $S_{jl}=S_j\ot_RS_l$.
 Assume that, for every pair of indices $j$ and~$l$,
the $S_{jl}$\+differential operator $D_{jl}\:\Hom_R(S_{jl},U)\rarrow
\Hom_R(S_{jl},V)$ induced by $D_j$, as per the construction of
Proposition~\ref{colocalizing-quite-diff-operators-prop}, is equal to
the $S_{lj}$\+differential operator $D_{lj}\:\Hom_R(S_{lj},U)\rarrow
\Hom_R(S_{lj},V)$ induced by~$D_l$.
 Then there exists a unique quite $R$\+differential operator
$D\:U\rarrow V$ of ordinal order at most~$\alpha$ such that, for every
index~$l$, the $S_l$\+differential operator $D_l$ is induced by $D$
as per the construction of
Proposition~\ref{colocalizing-quite-diff-operators-prop}.
\end{prop}

\begin{proof}
 Similar to Proposition~\ref{cosheaf-axiom-differential-operators}, and
using (the simple implication in)
Corollary~\ref{ring-epi-restr-of-scalars-quite-diff-operators} together
with Lemma~\ref{quite-diff-operator-corestricted-to-quotients}.
\end{proof}

 The counterexample from
Remark~\ref{contraadjustedness-necessary-remark} shows that
the contraadjustedness assumption on $U$ cannot be dropped in
Proposition~\ref{cosheaf-axiom-quite-differential-operators}.

\end{document}